\documentclass[sort&compress,3p]{elsarticle}

\usepackage{algorithm}
\usepackage{algpseudocode}
\usepackage{amsmath}
\usepackage{amssymb}
\usepackage{amsthm}
\usepackage{array}
\usepackage{booktabs}
\usepackage{breqn}
\setkeys{breqn}{breakdepth={1}}
\usepackage{changepage}
\usepackage{collcell}
\usepackage{color}
\usepackage{diagbox}
\usepackage{enumerate}
\usepackage{enumitem}
\usepackage{epstopdf}
\usepackage{float}
\usepackage[T1]{fontenc}
\usepackage{geometry}
\usepackage{graphics}
\usepackage{graphicx}
\usepackage{hyperref}
\usepackage[labelfont=bf,labelsep=space]{caption}
\usepackage{latexsym}
\usepackage{longtable}
\usepackage{lscape}
\usepackage{mathtools}
\usepackage{mdframed}
\usepackage{multicol}
\usepackage{pdfpages}
\usepackage{pgf}
\usepackage{pifont}
\usepackage{placeins}
\usepackage{siunitx}
\usepackage{soul}
\usepackage{subcaption}
\usepackage{tabularx}
\usepackage{tikz}
\usetikzlibrary{shapes}
\usepackage{todonotes}
\usepackage{wasysym}
\usepackage{ulem}
\usepackage[utf8]{inputenc}
\usepackage{xcolor}
\usepackage{xspace}
\usepackage{cleveref}

\usepackage{newfloat}
\usepackage{caption}
\DeclareFloatingEnvironment[fileext=frm,placement={!htb},name=Formulation]{formulation}
\DeclareFloatingEnvironment[fileext=mcoeff,placement={!htb},name=Method]{method}

\newtheorem{theorem}{Theorem}[section]
\newtheorem{lemma}[theorem]{Lemma}
\newtheorem{corollary}{Corollary}[theorem]

\theoremstyle{remark}
\newtheorem{remark}{Remark}[section]

\theoremstyle{definition}
\newtheorem{definition}{Definition}[section]

\theoremstyle{example}
\newtheorem{example}{Example}[section]

\DeclarePairedDelimiter\abs{\lvert}{\rvert}%
\DeclarePairedDelimiter\norm{\lVert}{\rVert}%

\makeatletter
\let\oldabs\abs
\def\abs{\@ifstar{\oldabs}{\oldabs*}}
\let\oldnorm\norm
\def\norm{\@ifstar{\oldnorm}{\oldnorm*}}
\makeatother

\newcommand{\fone}{f^{\{1\}}}
\newcommand{\ftwo}{f^{\{2\}}}
\newcommand{\Lb}{\mathbf{L}}
\newcommand{\Lone}{\Lb^{\{1\}}}
\newcommand{\Ltwo}{\Lb^{\{2\}}}
\newcommand{\Jb}{\mathbf{J}}
\newcommand{\Jone}{\Jb^{\{1\}}}
\newcommand{\Jtwo}{\Jb^{\{2\}}}
\newcommand{\N}{\mathcal{N}}
\renewcommand{\P}{\mathcal{P}}
\newcommand{\None}{\N^{\{1\}}}
\newcommand{\Ntwo}{\N^{\{2\}}}
\newcommand{\Wb}{\mathbf{W}}
\newcommand{\Wone}{\Wb^{\{1\}}}
\newcommand{\Wtwo}{\Wb^{\{2\}}}
\newcommand{\R}{\mathbb{R}}
\renewcommand{\Re}{\mathbb{R}}

\newcommand{\J}{\mathbf{J}}
\newcommand{\Mb}{\mathbf{M}}
\newcommand{\Pn}{P}

\newcommand{\Jn}{\mathbf{J}_{n}}
\newcommand{\pfrac}[2]{\frac{\partial #1}{\partial #2}}

\newcommand{\Tree}{\mathbb{T}}
\newcommand{\RemoveRoot}{\mathbb{R}}

\newcommand{\mbf}[1]{\mathbf{#1}}

\newcommand\arrd[1]{
    \tikz[baseline, inner xsep=-1cm]{
        \draw[->] (0,.4) -- (0,-.2) -- (-.2,-.4) node[below, font=\tiny] {#1};
    }\ % for spacing
}

\newcommand\arru[1]{
	\tikz[baseline, inner xsep=-4cm]{
		\draw[->] (0,.4) -- (0,-.1) -- (.2,-.2) node[below, font=\tiny] {#1};
    }\ % for spacing
}

\newcolumntype{P}[1]{>{\centering\arraybackslash\vspace{-1ex}}m{#1}<{\vspace*{-1ex}}}

\newif\iflong
\longtrue

\title{Partitioned Exponential Methods for Coupled Multiphysics Systems}
\author[aff1,em1]{Mahesh Narayanamurthi}
\ead{maheshnm@vt.edu}
\author[aff1,em2]{Adrian Sandu\corref{cor1}}
\ead{sandu@cs.vt.edu}
\address[aff1]{Computational Science Laboratory, Department of Computer Science, Virginia Tech, Blacksburg, VA 24060}
\address[em1]{E-mail: maheshnm@vt.edu}
\address[em2]{E-mail: sandu@vt.edu}
\cortext[cor1]{Corresponding author}

\begin{document}
\iflong
\includepdf[landscape=false,pages=-]{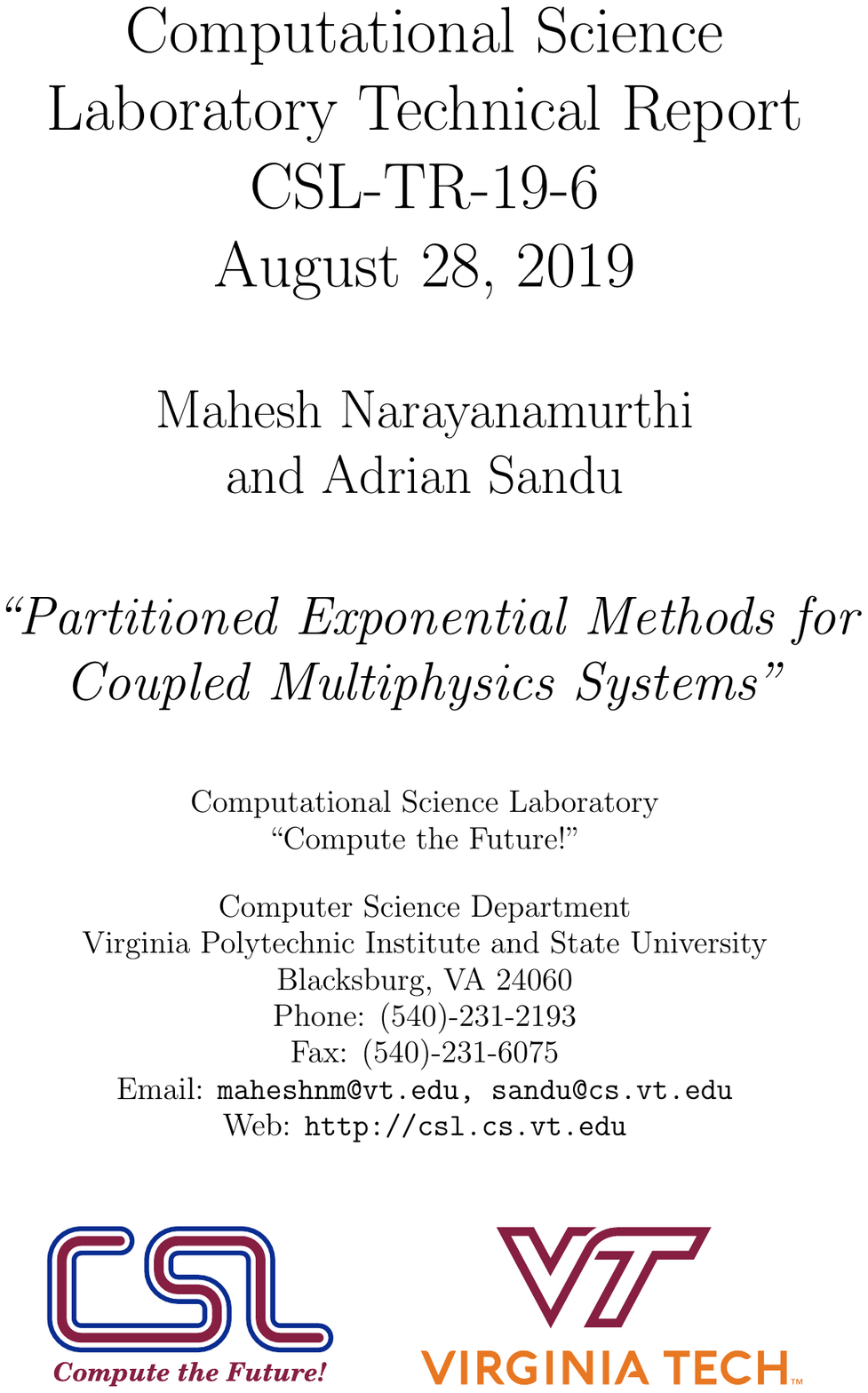}
\fi

\begin{abstract}
Multiphysics problems involving two or more coupled physical phenomena are ubiquitous in science and engineering. This work develops a new partitioned exponential approach for the time integration of multiphysics problems. After a possible semi-discretization in space, the class of problems under consideration is modeled by a system of ordinary differential equations where  the right-hand side is a summation of two component functions, each corresponding to a given set of physical processes. 

The partitioned-exponential methods  proposed herein evolve each component of the system via an exponential integrator, and information between partitions is exchanged via coupling terms. The traditional approach to constructing exponential methods, based on the variation-of-constants formula, is not directly applicable to partitioned systems. Rather, our approach to developing new partitioned-exponential families is based on a general-structure additive formulation of the schemes. Two method formulations are considered, one based on a linear-nonlinear splitting of the right hand component functions, and another based on approximate Jacobians. The paper develops classical (non-stiff) order conditions theory for partitioned exponential schemes based on particular families of T-trees and B-series theory. Several practical methods of third order are constructed that extend the Rosenbrock-type and EPIRK families of exponential integrators. Several implementation optimizations specific to the application of these methods to reaction-diffusion systems are also discussed. Numerical experiments reveal that the new partitioned-exponential methods can perform better than traditional unpartitioned exponential methods on some problems.

\end{abstract}

\maketitle

\iflong
\tableofcontents
\fi

\newpage
%%%%%%%%%%%%%%%%%%%%%%%%%%%%%%%%%%%%%%%
% !TEX root = Nonstiff_pexpw.tex
%%%%%%%%%%%%%%%%%%%%%%%%%%%%%%%%%%%%%%%
\section{Introduction}
\label{sec:introduction}
%%%%%%%%%%%%%%%%%%%%%%%%%%%%%%%%%%%%%%%
Multiphysics problems involve two or more coupled physical phenomena that take place simultaneously in both space and time. After semi-discretization in space, multiphysics problems are modeled as a system of ordinary differential equations where the time tendency (the ``right-hand side function'') is additively partitioned, with each component modeling a different physical phenomenon. 
%In reaction-diffusion problems formulated as:
%\begin{eqnarray*}
%	\pfrac{u}{t} = d_1(u) + r_1(u, v),\\
%	\pfrac{v}{t} = d_2(v) + r_2(u, v),
%\end{eqnarray*}
%%
%$u$, $v$ are the concentrations of the reactants that vary in both space and time; $d_1(u)$, $d_2(v)$ are pieces of the full right-hand side function that correspond to diffusion of the reactants; and $r_1(u, v)$, $r_2(u, v)$ correspond to reactions between the reactants. 
Physical phenomena may have different dynamical characteristics, i.e., some can be stiff and some non-stiff.
%
% measure, where we loosely define stiffness as a property of the terms of the differential equation that forces explicit time integration methods to take very small timesteps. 
For example, in atmospheric composition problems, one has non-stiff advection, mildly stiff turbulent diffusion, and stiff chemical and aerosol processes \cite{Sandu_2003_aerosolFramework,Sandu_2004_multiScale}.

Multi-methods for time integration, such as implicit-explicit (IMEX) schemes \cite{Ascher_1995_IMEX,Ascher_1997_IMEX_RK,Verwer_2004_IMEX_RKC,Sandu_2010_extrapolatedIMEX,Sandu_2012_ICCS-IMEX,Sandu_2014_IMEX_GLM_Extrap,Sandu_2014_IMEX-GLM,Sandu_2014_IMEX-RK,Sandu_2015_IMEX-TSRK,Sandu_2015_Stable_IMEX-GLM,Sandu_2016_highOrderIMEX-GLM}, take advantage of this structure. The stiff term (e.g., diffusion) is handled by the implicit method and the non-stiff term (e.g., reaction) by the explicit method. In a two-partition system, if both terms are stiff, IMEX methods can be inefficient; the explicit method may be stability bound and suffer from timestep restriction. Some authors have used Implicit-Implicit methods \cite{belytschko1979,Zienkiewicz1988,Farhat1991,Dettmer2012,Sandu_2016_GARK-MR} to solve such problems, where both partitions are treated implicitly.  Implicit-Implicit methods require the two linear or non-linear systems be solved in a staggered manner with one partition going first, and extrapolating the other partition. Extrapolation can result in a method that is less stable compared to solving the unpartitioned problem with an implicit method \cite{belytschko1979}. Alternatively, solving the unpartitioned problem with an implicit method cannot draw benefits from the structure of the two individual partitions, which can be lost when treating them as one whole.  If a pre-conditioner is available for the full right-hand side or for each of the partitions, Implicit and Implicit-Implicit methods can be very efficient. However, if either partition is non-linear, obtaining a pre-conditioner can be a non-trivial task in itself.

Exponential integrators have proven to be more stable than explicit integrators, and computationally cheaper than implicit integrators for some problems \cite{loffeld2013,Sandu_2018_EPIRK-adjoint}. Although they were first developed several decades ago \cite{hersch1958,certaine1960,lawson1967,ehle1975}, research on exponential integrators faded due to the lack of efficient means to compute matrix-exponential-like operations. Hochbruck et al. \cite{Hochbruck_1997_exp,Hochbruck_1998_exp} revived interest in the field by building adaptive timestep methods and efficiently evaluating matrix-exponential-like functions using Krylov subspace approximations. Since then numerous authors have contributed to the field as summarized in these review articles \cite{minchev2005,Hochbruck_2010_exp}.

Exponential integrators for split/partitioned systems have recently enjoyed considerable attention. Implicit-Exponential integrators like IMEXP \cite{luan2016} and IIF \cite{nie2006, chou2007, zhao2011} treat one of the partitions implicitly and the other partition exponentially. In the context of reaction-diffusion equations,  IMEXP methods have been used to treat the linear operator corresponding to diffusion implicitly and the non-linear reaction exponentially \cite{luan2016}. IMEXP admits the use of preconditioners for solving the Laplacian operator. IIF schemes, on the other hand, have been used to treat the diffusion term exponentially and the reaction term implicitly  \cite{nie2006, chou2007}. They have been structured so that the implicit solve involving the nonlinear reaction terms is dealt with independent of the exponential treatment of the diffusion term. Therefore, the implicit solves are performed only on nonlinear equations local to each grid-point \cite{chou2007}. For convection-diffusion problems Celledoni et al. \cite{celledoni2009} introduce semi-Lagrangian methods that treat the convective part exactly and the diffusive part implicitly. Structure preserving exponential methods have been studied in \cite{bhatt2017} for ODEs perturbed by a linear, damping term that can be time-dependent. There the linear operator corresponding to the linear term is scalar and Lawson transformation \cite{lawson1967} is used to rewrite the ODE. A variety of RK methods is applied to the modified ODE to construct new methods. This procedure to construct Exponential Runge-Kutta (ERK) and Partitioned Exponential Runge-Kutta (PERK) schemes preserve some desirable properties of the exact solution as elucidated further in the article \cite{bhatt2017}.  Tranquilli and Sandu \cite{Sandu_2014_expK,Sandu_2014_ROK} apply matrix exponential only in a Krylov subspace, and integrate the remaining dynamics explicitly. Flexible Exponential Integrators (FEI) are developed in \cite{li2015}, and seek to generalize Exponential-Rosenbrock \cite{Hochbruck_2009_exp,luan2014a,Hochbruck_1998_exp} and Exponential Runge--Kutta \cite{Hochbruck_2005_expRK} methods. They split the right-hand side into stiff and non-stiff terms. They then perform a continuous linearization of the nonstiff term about the current solution, and arrive at a formulation that admits different combinations of exponential-like matrix functions to act on the stiff and non-stiff remainder terms.

This work develops partitioned exponential methods where each component of the system is discretized with a (different) exponential integrator. Our exponential-exponential approach distinguishes this work from previous exponential-implicit, exponential-explicit, or exponential-Lagrangian schemes. We build
%% ------------ Removing the text below (AVG) --------------------%%
% split and 
%% ---------------------------------------------------------------%% 
W-type partitioned-exponential methods by casting unpartitioned methods from EPIRK \cite{Tokman_2006_EPI, Tokman_2011_EPIRK} and EXP \cite{Hochbruck_1997_exp,Hochbruck_1998_exp} families into a structure-revealing, GARK-like \cite{Sandu_2015_GARK,Sandu_2016_GARK-MR} framework. W-methods were first introduced in \cite{steihaug1979} to build implicit time integration methods that admit Jacobian approximations in the method formulation.
%% ------------ Removing the text below (AVG) --------------------%%
% We also discuss a family of methods constructed by averaging two unpartitioned sEPIRK \cite{Rainwater_2014_semilinear} methods.
%% ---------------------------------------------------------------%%

In section \ref{sec:partitioned-problems}, we discuss the partitioned problem formulation and the construction of partitioned methods using the variation-of-constants formula. We highlight the difficulty of building methods via the variation-of-constants approach and consider some strategies to work around the restriction. Section \ref{sec:pexpw} presents alternate formulations of partitioned-exponential methods using the GARK framework.
%% ------------ Removing the text below (AVG) --------------------%%
% and by averaging two unpartitioned exponential methods. 
%% ---------------------------------------------------------------%%
The structure of trees and derivatives, and B-series theory for the formulations presented in section \ref{sec:pexpw} are discussed in section \ref{sec:order}. Section \ref{sec:construction} delves briefly into the construction of three stage third order methods. Implementation issues and various computational optimizations specific to reaction-diffusion systems are addressed in section \ref{sec:implementation}. Numerical experiments are studied in section \ref{sec:numerics} and conclusions are drawn in section \ref{sec:conclusions}.

%%%%%%%%%%%%%%%%%%%%%%%%%%%%%%%%%%%%%%%

%%%%%%%%%%%%%%%%%%%%%%%%%%%%%%%%%%%%%%%
% !TEX root = Nonstiff_pexpw.tex
%%%%%%%%%%%%%%%%%%%%%%%%%%%%%%%%%%%%%%%
\section{Partitioned Problems and the Variation of Constants Formula}
\label{sec:partitioned-problems}
%%%%%%%%%%%%%%%%%%%%%%%%%%%%%%%%%%%%%%%

We are concerned with the numerical solution of complex differential equations \eqref{eqn:functional_splitting_ode_system} that can be split into multiple components. Splitting \cite{mclachlan2002, macnamara2017} is a general concept for partitioning a problem into its constituent pieces that are each simpler to solve than the original problem. In the context of ODEs (and PDEs), splitting can be applied in a number of ways such as: separate the right-hand side function into linear and non-linear pieces and treat each of them differently \cite{Rainwater_2014_semilinear,fornberg1999,Calvo_2001_IMEX_RK,akrivis2004}; perform dimensional splitting and integrate along one spatial dimension at a time \cite{verwer1984,nakamura2001,lie1998}; partition the right-hand side function into distinct operators and perform time evolution of the PDE one operator at a time \cite{issa1986,farago2008,sportisse2000,farago2007,karlsen1997}; partition space into non-overlapping domains and solve the PDE on each domain (in parallel) and adjust the solution at the interfaces of domains \cite{toselli2005,smith2004}.

To be specific, consider the initial value problem
\begin{equation}
	\label{eqn:functional_splitting_ode_system}
	\begin{split}
		y' &= F(y) = \sum_{m=1}^{\Pn} f^{\{m\}}(y),
		%\fone(y) + \ftwo(y), 
		\qquad y(t_0) = y_0 \in \Re^N, \qquad t \ge t_0,
	\end{split}
\end{equation}
where $F$ is the (unpartitioned) full right-hand side function, which is composed of $\Pn$ processes $f^{\{1\}}$ $\dots$ $f^{\{\Pn\}}$ acting simultaneously.  
To derive a partitioned exponential formulation, we further split each right-hand side function component as follows:
\begin{subequations}
\label{eqn:linearized_rhs_funcs}
\begin{eqnarray}
	\label{eqn:linearized_rhs_split}
		f^{\{m\}}(y) &=& \Lb^{\{m\}}\, y + \N^{\{m\}}(y), \quad m=1,\dots,\Pn, \\
%	\label{eqn:linearized_rhs_split2}
%		\ftwo(y) &= \Ltwo\, y + \Ntwo(y), \\
	\label{eqn:linearized_rhs}
		F(y) &=& \Lb\,y + \N(y),
		\quad \Lb = \sum_{m=1}^{\Pn}  \Lb^{\{m\}}, \quad \N(\cdot) = \sum_{m=1}^{\Pn}  \N^{\{m\}}(\cdot),
\end{eqnarray}
\end{subequations}
where $\Lb^{\{m\}} y$ are the linear components and $\N^{\{m\}}(y)$ are the corresponding non-linear remainders of each $f^{\{m\}}(y)$, respectively. We assume that $\Lb^{\{m\}}$ capture the stiffness in the corresponding partitions.

%%%%%%%%%%%%%%%%%%%%%%%%
\subsection{Classical exponential methods}
\label{sec:classical-exp}
%%%%%%%%%%%%%%%%%%%%%%%%

Classical exponential schemes solve the unpartitioned, linearized equation \eqref{eqn:linearized_rhs}. In this paper, we study partitioned derivatives of two families of classical exponential integrators -- Rosenbrock-style exponential methods (EXP \cite{Hochbruck_1998_exp}) and Runge--Kutta style exponential methods (EPIRK \cite{Tokman_2006_EPI,Tokman_2011_EPIRK}). 

Matrix-exponential-like functions and their products with vectors serve as major building blocks of exponential time integrators. Most commonly appearing matrix-exponential-like operators are the analytical functions, $\varphi_k$, defined as follows \cite{Tokman_2011_EPIRK}:
\begin{equation}
	\label{eqn:phi_function_definition}
	\begin{split}
		\varphi_{0}(z) = \exp(z), \qquad
		\varphi_{k}(z) &= \int_{0}^{1} e^{(1 - \theta) z} \cdot \frac{\theta^{k - 1}}{(k-1)!} \,d\theta,\qquad k \ge 1,
	\end{split}
\end{equation}
and they satisfy the recurrence relation:
\begin{equation}
	\label{eqn:phi-recurrence}
	\begin{split}
		\varphi_{k+1}(z) &= \cfrac{\varphi_k(z) - \varphi_{k}(0)}{z},\qquad \varphi_{k}(0) = 1/k!.\\
	\end{split}
\end{equation} 
The series expansion of $\varphi_k(z)$ is:
\begin{equation}
	\label{eqn:phi-series-expansion}
	\begin{split}
		\varphi_{k}(z) &= \sum_{i=0}^{\infty} \cfrac{z^i}{(k + i)!}.
	\end{split}
\end{equation} 

Rosenbrock-style exponential methods proposed in \cite{Hochbruck_1998_exp} (and denoted EXP herein) use only $\varphi_{1}$ functions in their formulation. The unpartitioned EXP method is summarized in Formulation \ref{frm:unpartitioned_exp_method}. Here, $\Jn = \pfrac{F}{y} (y_n)$ is the Jacobian of the right-hand side function; $s$ is the number of stages of the method; lastly, $\alpha$, $\gamma$ (and $\gamma_{i,j}$) and $b$ are the coefficients of the method that one has to determine. W-type EXP methods can replace the Jacobian $\Jn$ by an arbitrary matrix $\Wb$ while preserving the order of accuracy  \cite{Hochbruck_1998_exp}.

\begin{formulation}
	\begin{eqnarray*}
		k_{i} &=& \varphi_{1}(\gamma h \Jn) \left(F(u_{i}) + h \Jn \sum_{j = 1}^{i - 1} \gamma_{i,j} k_j\right), \nonumber \\
		u_{i} &=& y_n + h \sum_{j = 1}^{i - 1} \alpha_{i,j} k_j, \qquad i = 1,\hdots,s\nonumber \\
		y_{n+1} &=& y_n + h \sum_{i=1}^{s} b_i k_i.
	\end{eqnarray*}
	\caption{Standard Rosenbrock-exponential (EXP) method \cite{Hochbruck_1998_exp} applied to the unpartitioned system \eqref{eqn:functional_splitting_ode_system}.}
	\label{frm:unpartitioned_exp_method}
\end{formulation}

EPIRK (Exponential Propagation Iterative Methods of Runge--Kutta type) methods \cite{Tokman_2006_EPI,Tokman_2011_EPIRK} have a very general formulation among exponential integrators of Runge--Kutta type. Unlike the EXP method, EPIRK methods rely on $\varphi_{1}$ and higher-order $\varphi_k$ functions and their linear combinations ($\Psi$, see \eqref{eqn:psi_definition}) in their construction. In return, we can build higher-order methods with low stage count \cite{Tokman_2011_EPIRK,tokman2012}.

Consider also the sEPIRK scheme \cite{Rainwater_2014_semilinear} in Formulation \ref{frm:unpartitioned_sepirk_method}. $\Lb$ is the linear operator and $\N$ is the non-linear remainder after linearization of the right-side function, $F$, i.e., $F(y) = \Lb y + \N(y)$; $\Psi$ is the linear combination of $\varphi_k$ functions and is defined in \eqref{eqn:psi_definition}; the forward difference on the non-linear remainder, $\Delta^{(l - 1)}\N\bigl(y_n\bigr)$, is defined in equation \eqref{eqn:forward_difference_on_N}. The coefficients $a$'s, $b$'s, $g$'s, and $p$'s (embedded in $\Psi$) have to be determined to build new methods.

\begin{formulation}
	\begin{eqnarray*}
		Y_i &=& y_n +  {a}_{i,1} {\Psi}_{i,1}\bigl(h{g}_{i,1}\Lb\bigr) \,h f(y_n) + \sum_{l=2}^{i} {a}_{i,l} {\Psi}_{i,l}\bigl(h{g}_{i,l}\Lb\bigr)  h\Delta^{(l - 1)}\N\bigl(y_n\bigr), \qquad i = 1, \hdots, s - 1, \nonumber\\
		y_{n+1} &=& y_n + {b}_{1} {\Psi}_{s,1}\bigl(h{g}_{s,1}\Lb\bigr) \,h f(y_n) +  {\sum_{l=2}^{s} {b}_{l} {\Psi}_{{s},l}\bigl(h{g}_{{s},l}\Lb\bigr)  h\Delta^{(l - 1)}\N\bigl(y_n\bigr)}.
	\end{eqnarray*}
	\caption{Standard sEPIRK method \cite{Rainwater_2014_semilinear} applied to the unpartitioned system \eqref{eqn:functional_splitting_ode_system}.}
	 \label{frm:unpartitioned_sepirk_method}
\end{formulation}
The $\Psi$ functions that appear in Formulation~\ref{frm:unpartitioned_sepirk_method} are linear combinations of $\varphi_k$ functions \eqref{eqn:phi_function_definition}:
\begin{equation}
	\label{eqn:psi_definition}
	\begin{split}
		\Psi_{i,j}(z) = \Psi_{j}(z) = \sum_{k = 1}^{j} p_{j,k}\, \varphi_{k}(z).
	\end{split}
\end{equation}
The forward difference operator acting on the nonlinear remainder terms in Formulation~\ref{frm:unpartitioned_sepirk_method} arises from their interpolation and is defined using a recurrence relation:

\begin{equation}
	\label{eqn:forward_difference_on_N}
	\Delta^{(1)}\N(Y_i) = \N(Y_{i+1}) - \N(Y_i), \qquad \Delta^{(j)}\N(Y_i) = \Delta^{(j-1)}\N(Y_{i+1}) - \Delta^{(j-1)}\N(Y_i).
\end{equation}

%%%%%%%%%%%%%%%%%%%%%%%%
\subsection{A split solution approach}
%%%%%%%%%%%%%%%%%%%%%%%%

Using the variation-of-constants approach we write the solution to the ODE system \eqref{eqn:functional_splitting_ode_system} with $P=2$ processes at a future time instant, $t_n + h$, as:
\begin{equation}
	\label{eqn:variation_of_constants_approach}
	\begin{split}
		y(t_n + h) &= y(t_n)  + \varphi_1\left(h\bigl(\Lone+\Ltwo\bigr)\right)\,h\Lone\, y(t_n) + \varphi_1\left(h\bigl(\Lone+\Ltwo\bigr)\right)\,h\Ltwo\, y(t_n)\\
		&+ h \LaTeXunderbrace{\displaystyle\int_0^1 e^{\left(h\bigl(\Lone+\Ltwo\bigr)\right)(1-\theta)} \LaTeXunderbrace{\None\left(y(t_n + h\theta)\right)}_{\textnormal{interpolate}} d\theta}_{\textnormal{numerical quadrature}} + h \LaTeXunderbrace{\displaystyle\int_0^1 e^{\left(h\bigl(\Lone+\Ltwo\bigr)\right)(1-\theta)} \LaTeXunderbrace{\Ntwo\left(y(t_n + h\theta)\right)}_{\textnormal{interpolate}} d\theta}_{\textnormal{numerical quadrature}}.
	\end{split}
\end{equation}
Following the standard sEPIRK \cite{Rainwater_2014_semilinear} approach, a discretization method is obtained by using (possibly different) numerical quadratures to approximate the integrals, and interpolating (with possibly different formulas) terms involving the non-linear remainder. The additively partitioned sEPIRK method is summarized in Formulation \ref{frm:epirk_like_methods} below. 
%
%\begin{formulation}
%	\begin{equation*}
%		\begin{split}
%			Y_i &= y_n + a_{i,1} \Psi_{i,1}\left(hg_{i,1}\bigl(\Lone+\Ltwo\bigr)\right) h\Lone y_n
%			+ \hat{a}_{i,1} \hat{\Psi}_{i,1}\left(h\hat{g}_{i,1}\bigl(\Lone+\Ltwo\bigr)\right) h\Ltwo y_n \\ &+ \sum_{j=1}^{i} a_{i,j} \Psi_{i,j}\left(hg_{i,j}\bigl(\Lone+\Ltwo\bigr)\right)  h\Delta^{(j - 1)}\None\bigl(y_n\bigr)\\ &+ \sum_{j=1}^{i} \hat{a}_{i,j} \hat{\Psi}_{i,j}\left(h\hat{g}_{i,j}\bigl(\Lone+\Ltwo\bigr)\right)  h\Delta^{(j - 1)}\Ntwo\bigl(y_n\bigr), \qquad i = 1, 2, \hdots s - 1,\\
%			y_{n+1} &= y_n + b_{1} \Psi_{s,1}\left(hg_{s,1}\bigl(\Lone+\Ltwo\bigr)\right) h\Lone y_n
%			+ \hat{b}_{1} \hat{\Psi}_{s,1}\left(h\hat{g}_{s,1}\bigl(\Lone+\Ltwo\bigr)\right) h\Ltwo y_n \\ &+ \sum_{j=1}^{s} b_{j} \Psi_{s,j}\left(hg_{s,j}\bigl(\Lone+\Ltwo\bigr)\right)  h\Delta^{(j - 1)}\None\bigl(y_n\bigr)\\ &+ \sum_{j=1}^{s} \hat{b}_{j} \hat{\Psi}_{s,j}\left(h\hat{g}_{s,j}\bigl(\Lone+\Ltwo\bigr)\right)  h\Delta^{(j - 1)}\Ntwo\bigl(y_n\bigr).
%		\end{split}
%	\end{equation*}
%	\caption{Additively partitioned sEPIRK method. \sandu{use commas to separate the two indices in the subscripts here and everywhere}}
%	 \label{frm:epirk_like_methods}
%\end{formulation}
\begin{formulation}
	\begin{equation*}
		\begin{split}
			Y_i &= y_n + a_{i,1} \Psi_{i,1}\left(hg_{i,1}\Lb\right) h\Lone y_n
			+ \hat{a}_{i,1} \hat{\Psi}_{i,1}\left(h\hat{g}_{i,1}\Lb\right) h\Ltwo y_n \\ &+ \sum_{j=1}^{i} a_{i,j} \Psi_{i,j}\left(hg_{i,j}\Lb\right)  h\Delta^{(j - 1)}\None\bigl(y_n\bigr)\\ &+ \sum_{j=1}^{i} \hat{a}_{i,j} \hat{\Psi}_{i,j}\left(h\hat{g}_{i,j}\Lb\right)  h\Delta^{(j - 1)}\Ntwo\bigl(y_n\bigr), \qquad i = 1, 2, \hdots s - 1,\\
			y_{n+1} &= y_n + b_{1} \Psi_{s,1}\left(hg_{s,1}\Lb\right) h\Lone y_n
			+ \hat{b}_{1} \hat{\Psi}_{s,1}\left(h\hat{g}_{s,1}\Lb\right) h\Ltwo y_n \\ &+ \sum_{j=1}^{s} b_{j} \Psi_{s,j}\left(hg_{s,j}\Lb\right)  h\Delta^{(j - 1)}\None\bigl(y_n\bigr)\\ &+ \sum_{j=1}^{s} \hat{b}_{j} \hat{\Psi}_{s,j}\left(h\hat{g}_{s,j}\Lb\right)  h\Delta^{(j - 1)}\Ntwo\bigl(y_n\bigr).
		\end{split}
	\end{equation*}
	\caption{Additively partitioned sEPIRK method.}
	 \label{frm:epirk_like_methods}
\end{formulation}
In the formulation, $\Psi$ and $\hat{\Psi}$ are distinct linear combinations of $\varphi_k$ functions as defined in equation \eqref{eqn:psi_definition}. When the same quadrature and interpolation formulas are used for both partitions, $\hat{\Psi}_{i,j}= \Psi_{i,j}$ and $\hat{g}_{i,j}= g_{i,j}$, the original sEPIRK method \cite{Rainwater_2014_semilinear} is recovered.

The drawback of Formulation~\ref{frm:epirk_like_methods} is that it does not fully benefit from partitioning the right-hand side function. Specifically, as with any exponential method, the computational cost is dominated by the evaluation of matrix function vector products of the form $\hat{\Psi}_{i,j}\bigl(h \hat{g}_{i,j}\Lb\bigr)\,u$ and  $\Psi_{i,j}\bigl(h {g}_{i,j}\Lb\bigr)\,u$. The matrix functions are computed on the matrix $\Lb = \Lone + \Ltwo$, and therefore Formulation~\ref{frm:epirk_like_methods} 
is equivalent to having linearized the full right-hand side and then used the variation-of-constants approach. To realize in full the potential computational benefits of treating each partition on its own, one needs to be able to compute matrix functions of individual linear components, $\hat{\Psi}_{i,j}\bigl(h \hat{g}_{i,j}\Lone\bigr)\,u$
and  $\Psi_{i,j}(h {g}_{i,j}\Ltwo)\,u$.

A possible idea is to use splitting methods \cite{mclachlan2002, macnamara2017} to approximate the exponential 
function in \eqref{eqn:variation_of_constants_approach}:
\begin{equation}
	\label{eqn:exponential-approximation}
	e^{h g \Lb} \approx  e^{h g \Lone} \cdot e^{h g \Ltwo}.
\end{equation}
However, we cannot easily arrive at a discrete formulation involving the $\varphi_k$ functions. Alternatively, one can start with the discrete formulation of Formulation \eqref{frm:epirk_like_methods}, and replace the $\varphi_k(h \hat{g}_{i,j}\Lb)\,u$ terms by expressions computed using the exponential split approximation \eqref{eqn:exponential-approximation} in the function definition \eqref{eqn:phi_function_definition}. The errors resulting from these approximations need to be quantified appropriately. We do not pursue further this line of thinking in the current paper.

%, and their linear combinations as the definition of $\varphi_k$ functions involves only a single exponential term  inside the integral, see~\eqref{eqn:phi_function_definition}, whereas by using the first-order splitting we end up with a product of two exponential terms inside the integral. Higher order-splittings for the exponential of the sum of matrices exist [need citation], they too suffer from the same weakness -- of not being able to write the formulation using $\varphi_k$ functions. Higher-order splittings also require that the system be time-reversible and that makes them unsuitable for use in our context \cite{macnamara2017,sheng1989, suzuki1991, geiser2011a}.

Rather, we build partitioned methods using a structure-revealing formulation similar to the one employed in the ``Generalized Additive Runge--Kutta'' framework (GARK) \cite{Sandu_2015_GARK,Sandu_2016_GARK-MR}. In 
\ref{sec:AVG-SEPIRK} we discuss an alternative procedure to build partitioned exponential methods that involves averaging the stages of two unpartitioned methods, each using a separate linear operator as argument to the $\Psi$ functions. We next discuss GARK-based formulations.

%%%%%%%%%%%%%%%%%%%%%%%%%%%%%%%%%%%%%%%

%%%%%%%%%%%%%%%%%%%%%%%%%%%%%%%%%%%%%%%
% !TEX root = Nonstiff_pexpw.tex
%%%%%%%%%%%%%%%%%%%%%%%%%%%%%%%%%%%%%%%
\section{Partitioned-Exponential Formulations}
\label{sec:pexpw}
%%%%%%%%%%%%%%%%%%%%%%%%%%%%%%%%%%%%%%%

In this section, we construct partitioned-exponential methods using a general-structure additive exponential strategy (GAXP), which extends an unpartitioned exponential scheme to a partitioned scheme in structure-revealing, GARK-like form. The unpartitioned base methods used here to construct the new schemes fall into one of two well-known exponential time integration families, EXP \cite{Hochbruck_1997_exp,Hochbruck_1998_exp} and EPIRK \cite{Tokman_2006_EPI,Tokman_2011_EPIRK}, discussed in Section \ref{sec:classical-exp}. 

The structure-revealing GARK methods for partitioned systems allow the component functions ($\fone$ and $\ftwo$) to be evaluated on different stage values \cite{Sandu_2015_GARK,Sandu_2016_GARK-MR}. The stage values are computed in a Runge--Kutta framework. Borrowing this approach, GAXP builds different stage vectors for different partitions. Each stage computation includes coupling terms with the other components. The GAXP formulation of Rosenbrock-style EXP-W schemes \cite{hochbruck1998} is given in Formulation \ref{frm:gark_exp_w_method} and that for Runge--Kutta style EPIRK-W schemes \cite{Sandu_2019_EPIRKW} is given in Formulation \ref{frm:gark_type_epirkw_like_method}.

The practical methods we develop herein and report in \ref{sec:coefficients-PEXPW} and \ref{sec:coefficients-PEPIRKW}, 
use different stage vectors for all partitions, analogous to the philosophy of classical generalized additive Runge--Kutta methods (GARK) \cite{Sandu_2015_GARK}. The method we report in
\iflong
\ref{sec:coefficients-PSEPIRK}, 
\else
{\cite[Appendix D]{narayanamurthi2019}, }
\fi
uses same stage vectors for all partitions, analogous to the philosophy of classical additive Runge--Kutta methods (ARK) \cite{Kennedy_2003_additiveRK}. We found this to be a reasonable compromise to build partitioned-exponential methods. Similar to how classical exponential methods degenerate to classical Runge--Kutta methods when the linear operator  $\Lb$ in \eqref{eqn:linearized_rhs_funcs} is a zero matrix, our methods degenerate to a GARK or an ARK method scheme in this case.

%\sandu{I am not sure how this comments brings anything new. $\Lone$, $\Ltwo$ are already arbitrary, isn't it?}
%\mahesh{Although the two procedures can lead to the same method, the way we arrive at the methods are markedly different. The order conditions are not exactly the same but as far as I have noticed, they are equivalent and have the same solution. I cannot prove this generically for all W and split-methods. At a high-level, one can argue as follows: in the unpartitioned case, split methods are obtained by applying variation-of-constants to split right-hand side, whereas W methods are obtained by replacing all occurrences of the Jacobian J by W in a classical method. But since the classical method is itself obtained by variation of constants after a continuous linearization of the right-hand side, the residual term has a J which gets replaced by an arbitrary W. It is at this point that we can say that starting with an arbitrary linearization and performing a variation-of-constants is equivalent to starting with a continous linearization, performing variation of constants and replacing the J by an arbitrary W. Also note that the way we treat the expansion of the true solution in either case is different. If we start with a split right-hand side, we try to express the true solution using the L and the N trees and we have non-zero coefficients on the exact solution on each tree that we can form. Whereas for W methods, coefficients on trees with W nodes are zeo. I have included some version of this discussion in the remarks below.}

%%%%%%%%%%%%%%%%%%%%%%%%%%%%%%%%%%%%%%%
\subsection{The split-right-hand-side approach and the W approach for constructing partitioned schemes}
\label{sec:split-vs-W}
%%%%%%%%%%%%%%%%%%%%%%%%%%%%%%%%%%%%%%%

We seek to build methods for initial value problems (IVP) 	\ref{eqn:functional_splitting_ode_system}--\ref{eqn:linearized_rhs_funcs} where the right-hand side is additively partitioned (for two partitions, $F = \fone + \ftwo$). Two approaches, the split and the W-approach, are possible to formulate the methods and analyze the order conditions. 

In the split-RHS approach \cite{Rainwater_2014_semilinear}, the component functions are split into their corresponding linear ($\Lone$ and $\Ltwo$) and non-linear ($\None$ and $\Ntwo$) parts, and Taylor series of the exact and numerical solutions are constructed using derivatives of these parts. Order conditions are obtained by combining these terms such as to recover the derivatives of the exact solution, which involve derivatives of the un-split right hand side functions  $\fone,\ftwo$.

W formulations \cite{steihaug1979,Hairer_book_I} are obtained by replacing Jacobians ($\Jone$ and $\Jtwo$) by arbitrary square matrices ($\Wone$ and $\Wtwo$) in the formulation of an exponentially split scheme. Taylor series of the numerical solution are constructed using derivatives of the right hand side functions  $\fone,\ftwo$, and order conditions are imposed to cancel all the terms that involve the arbitrary matrices $\Wone$ and $\Wtwo$.

We argue, informally, that the split-RHS and W approaches to formulating partitioned exponential methods are equivalent.
%		
%	In the case of unpartitioned exponential methods, we obtain split methods by applying the variation-of-constants formula to the right-hand side function linearized using (an arbitrary) $\Lb$ matrix as in equation \ref{eqn:linearized_rhs}. Whereas, we obtain W-methods by replacing the exact Jacobian inside $\varphi_k$ (and $\Psi$) functions of classical exponential integrators, which are in turn constructed by applying the variation-of-constants formula to a continuously linearized right-hand side. Starting with an arbitrary linearization and performing a variation-of-constants is equivalent to: starting with a continous linearization, performing variation-of-constants and replacing the Jacobian, $\Jn$, by an arbitrary $\Wb$.
%	
In Section \ref{sec:order} we detail the procedure we follow to build methods using the partitioned-split and partitioned-W formalisms, and it turns out that the resulting machinery is largely similar.

%%%%%%%%%%%%%%%%%%%%%%%%%%%%%%%%%%%%%%%
\subsection{Partitioned methods of Rosenbrock-exponential type (PEXPW)}
\label{sec:PEXPW}
%%%%%%%%%%%%%%%%%%%%%%%%%%%%%%%%%%%%%%%

We now consider the Rosenbrock-exponential schemes described in Formulation \ref{frm:unpartitioned_exp_method} and extend them to solve partitioned systems of the form \eqref{eqn:functional_splitting_ode_system}--\eqref{eqn:linearized_rhs_funcs}.

W-methods were first introduced in \cite{steihaug1979} to  admit the use of inexact Jacobians in implicit schemes. The Jacobian matrix in the formulation of Rosenbrock methods is replaced by an arbitrary square matrix $\Wb$. Order conditions are solved while taking into account the arbitrary approximation, so as to ensure a desired order of accuracy. While W-methods admit arbitrary approximations to the Jacobian, in practice one needs some form of Jacobian approximation $\Wb \approx \J_n$  to ensure stability of the methods \cite{steihaug1979} and \cite[Section IV.7 and IV.11]{Hairer_book_II}. Exponential methods of W-type are obtained by replacing the exact Jacobian in Formulation \ref{frm:unpartitioned_exp_method} with an arbitrary square matrix $\Wb$. The resulting (unpartitioned) EXP-W methods were first proposed and studied in \cite{Hochbruck_1998_exp}. W-methods of EPIRK type are discussed in \cite{Sandu_2019_EPIRKW}.

We extend EXP-W methods of Formulation \ref{frm:unpartitioned_exp_method} to work for partitioned multiphysics systems \eqref{eqn:functional_splitting_ode_system}. To this end we choose a different arbitrary square matrix $\Wb^{\{m\}}$ matrix for each partition $m = 1,\dots,\Pn$. The resulting partitioned family of methods is called PEXPW (partitioned EXP-W).

The generic PEXPW  method is defined in Formulation \ref{frm:gark_exp_w_method}. Here  $k_{i}^{\{m\}}$ defines the $i$-the internal stage of the $m$-th partition, and $s^{\{m\}}$ is the number of stages of the method applied to the $m$-th partition.  Each partition $m$ is solved with its own EXP-W method with  $s^{\{m\}}$  stages, coefficients $\alpha_{i,j}^{\{m,m\}}$ and $\gamma_{i,j}^{\{m,m\}}$, and its arbitrary Jacobian approximation $\Wb^{\{m\}}$. The vectors $k_{i}^{\{m\}}$ are the $i$-the internal stage of the $m$-th partition. The coefficients $\alpha_{i,j}^{\{q,m\}}$ and $\gamma_{i,j}^{\{q,m\}}$ define the coupling between partitions $q$ and $m$. The method is of general structure (GAXP) Rosenbrock-exponential type.

\begin{formulation}
	\begin{eqnarray*}
		k_{i}^{\{q\}} &=&  \varphi\Bigl(h\gamma_{i,i}^{\{q,q\}}\Wb^{\{q\}}\Bigr)\, \left(h f^{\{q\}}(u_{i}^{\{q\}}) + h \Wb^{\{q\}} \sum_{m=1}^{\Pn}\sum_{j = 1}^{\min(i - 1, s^{\{m\}})} \gamma_{i,j}^{\{q,m\}} k_{j}^{\{m\}}\right),\nonumber\\
%		\widetilde{k}_{i}^{\{q\}} &=& h{k}_{i}^{\{q\}} =  \varphi(h\gamma_{i,i}^{\{q,q\}}\Wb^{\{q\}}) \left(h f^{\{q\}}(u_{i}^{\{q\}}) + h \Wb^{\{q\}} \sum_{m=1}^{\Pn}\sum_{l = 1}^{\min(i - 1, s^{\{m\}})} \gamma_{i,l}^{\{q,m\}} \widetilde{k}_{l}^{\{m\}}\right)\nonumber\\
		u_{i}^{\{q\}} &=& y_{n} +  \sum_{m = 1}^{\Pn} \sum_{j=1}^{\min(i - 1, s^{\{m\}})} \alpha_{i,j}^{\{q,m\}} {k}_{j}^{\{m\}}, \qquad i = 1, \hdots, s^{\{q\}}, \qquad q = 1,\dots,\Pn, \nonumber \\
		y_{n+1} &=& y_{n} +   \sum_{q = 1}^{\Pn} \sum_{i=1}^{s^{\{q\}}} b_{i}^{\{q\}} {k}_{i}^{\{q\}}.	
		\end{eqnarray*}
	\caption{Partitioned EXP-W method (PEXPW) of GAXP-type.}
	\label{frm:gark_exp_w_method}
\end{formulation}

\begin{remark}
The traditional EXP-W methods of Formulation \ref{frm:unpartitioned_exp_method} use a single matrix-vector product calculation with $\varphi_1(h \gamma_{i,i} \Wb)$ per stage. PEXPW methods of Formulation \eqref{frm:gark_exp_w_method} requires the evaluation of $\Pn$ matrix-vector products $\varphi_1(h \gamma_{i,i}^{\{m,m\}} \Wb^{\{m\}})$ for $m=1,\dots,\Pn$. This offers the opportunity to replace on evaluation of a function of a complex matrix into multiple evaluations of functions of simpler, or sparse, matrices. In addition, the coefficients $\gamma_{i,i}^{\{m,m\}}$ can vary between partitions, allowing possible speed-ups of the Arnoldi (or Lanczos) iterations. Therefore, the partitioned formulation \eqref{frm:gark_exp_w_method} brings additional flexibility in constructing new methods that could result into improved computational performance.
\end{remark}

%%%%%%%%%%%%%%%%%%%%%%%%%%%%%%%%%%%%%%%
\subsection{Partitioned exponential methods of EPIRK type (PEPIRKW)}
\label{sec:PEPIRKW}
%%%%%%%%%%%%%%%%%%%%%%%%%%%%%%%%%%%%%%%

Recall the Runge--Kutta-exponential schemes discussed in section \ref{sec:classical-exp}. We will further focus on building partitioned methods of EPIRK and sEPIRK type using the GARK framework. 

%%%%%%%%%%%%%%%%%%%%%%%%%%%%%%%%%%%%%%%
\subsubsection{Partitioned sEPIRK with a GAXP structure}
%%%%%%%%%%%%%%%%%%%%%%%%%%%%%%%%%%%%%%%

Consider now the classical sEPIRK scheme described in Formulation \ref{frm:unpartitioned_sepirk_method} \cite{Rainwater_2014_semilinear}. The partitioned sEPIRK scheme built in the GAXP framework is summarized in Formulation \ref{frm:gark_sepirk_method}. Unfortunately, this structure does not lead to useful partitioned exponential methods (see Lemma \ref{lemma:drawback_of_sepirk_method}).

\begin{formulation}
	\begin{eqnarray*}
		{Y}_i^{\{j\}} &=& {y}_n + \sum_{k=1}^{\Pn}{a}_{i,1}^{\{j,k\}} {\Psi}_{i,1}^{\{j,k\}}\bigl(h{g}_{i,1}^{\{j,k\}}\Lb^{\{k\}}\bigr) h{f}^{\{k\}}({y}_n) + \sum_{k=1}^{\Pn}{\sum_{l=2}^{i_Y^{\{k\}}} {a}_{i,l}^{\{j,k\}} {\Psi}_{i,l}^{\{j,k\}}\bigl(h{g}_{i,l}^{\{j,k\}}\Lb^{\{k\}}\bigr)  h\Delta_{{Y}^{\{k\}}}^{(l - 1)}\N^{\{k\}}\big({y}_n\big)}, \\
		&&\qquad\qquad\qquad\biggl(\textnormal{where}\; i = 1, \hdots, s^{\{j\}}, \; j \in \mathcal{P}\biggr), \nonumber\\
		{y}_{n + 1} &=& {y}_n + \sum_{k=1}^{\Pn}{b}_{1}^{\{k\}} {\Psi}_{s^{\{k\}},1}^{\{k, k\}}\bigl(h{g}_{s^{\{k\}},1}^{\{k,k\}}\Lb^{\{k\}}\bigr) h{f}^{\{k\}}({y}_n) + \sum_{k=1}^{\Pn}{\sum_{l=2}^{s^{\{k\}}} {b}_{l}^{\{k\}} {\Psi}_{{s^{\{k\}}},l}^{\{k, k\}}\bigl(h{g}_{{s^{\{k\}}},l}^{\{k,k\}}\Lb^{\{k\}}\bigr)  h\Delta_{{Y}^{\{k\}}}^{(l - 1)}\N^{\{k\}}\big({y}_n\big)}. \nonumber
	\end{eqnarray*}
	\caption{Partitioned sEPIRK method of GAXP-type.}
	\label{frm:gark_sepirk_method}
\end{formulation}
	
\begin{lemma}
	For the initial value problem \eqref{eqn:functional_splitting_ode_system},  the partitioned sEPIRK method of Formulation \ref{frm:gark_sepirk_method} only has a first order solution.
	\label{lemma:drawback_of_sepirk_method}
\end{lemma}

\begin{proof}
	We specialize the proof for $\mathcal{P} = 2$. Consider the initial value problem \eqref{eqn:functional_splitting_ode_system}--\eqref{eqn:linearized_rhs_funcs}.
The true solution at a future time instant, $t_n +h$, can be written using Taylor series up to second-order as follows:
	\[
	y(t_n + h) = y(t_n) + h \left(\fone(y_n) + \ftwo(y_n)\right) + \frac{h^2}{2} \pfrac{(\fone(y) + \ftwo(y))}{y}\biggr\rvert_{y = y_n} \left(\fone(y_n) + \ftwo(y_n)\right) + \mathcal{O}(h^3). 
	\]
	Expand the mixed second order term in the Taylor series using the linear and non-linear parts of each component:	%
	\begin{equation*}
		\begin{split}
			\pfrac{\fone(y)}{y} \ftwo(y) &= \pfrac{\left(\Lone y + \None(y)\right)}{y} \left(\Ltwo y + \Ntwo(y)\right) \\
			&= \Lone \Ltwo y + \Lone \Ntwo(y) + \pfrac{\None(y)}{y} \Ltwo y+ \pfrac{\None(y)}{y} \Ntwo(y)
		\end{split}
	\end{equation*}
	So the true solution has the following terms in its expansion:
	\begin{subequations}
	\label{eqn:mixed-S-derivatives}
	\begin{equation}
	\bigg\{\Lone \Ltwo y, \Lone \Ntwo(y), \pfrac{\None(y)}{y} \Ltwo y, \pfrac{\None(y)}{y} \Ntwo(y)\bigg\}.
	\end{equation}
	The other second order term gives rise to the following additional terms in the expansion of the true solution:
	\begin{equation}
	\bigg\{\Ltwo \Lone y, \Ltwo \None(y), \pfrac{\Ntwo(y)}{y} \Lone y, \pfrac{\Ntwo(y)}{y} \None(y)\bigg\}.
	\end{equation}
	\end{subequations}
	The partitioned sEPIRK method of Formulation \ref{frm:gark_sepirk_method} applies functions of a component matrix $\Lone$ only to vectors corresponding to the same partition, and therefore it does not generate the following mixed terms in its solution expansion:
	\[
	\bigg\{\Lone \Ltwo y, \,\Lone \Ntwo(y),\, \Ltwo \Lone y,\, \Ltwo \None(y)\bigg\}.
	\]
	Consequently, Formulation \ref{frm:gark_sepirk_method} cannot be used to construct methods of order two or higher.
\end{proof}

%%%%%%%%%%%%%%%%%%%%%%%%%%%%%%%%%%%%%%%
\subsubsection{Partitioned EPIRK-W methods}

In order to obtain higher order methods we modify Formulation \ref{frm:gark_sepirk_method} as follows. First, in the spirit of EPIRK-W schemes \cite{Sandu_2019_EPIRKW}, we allow arbitrary matrices $\Wb^{\{m\}}$ for each partition in the formulation of matrix exponential functions. Second, we perform the forward difference operations on the component functions, rather than on the remainder terms. 

The resulting family of schemes, named PEPIRKW (partitioned EPIRK-W)
is summarized in Formulation \ref{frm:gark_type_epirkw_like_method}. The methods are of GAXP type, in that they build a separate set of stages for each partition, with coupling terms mixing in information from the other partitions.

\begin{formulation}
\begin{eqnarray*}
	{Y}_i^{\{q\}} &=& {y}_n + \sum_{m=1}^{\Pn}{a}_{i,1}^{\{q,m\}} {\Psi}_{i,1}^{\{q,m\}}\Bigl(h{g}_{i,1}^{\{q,m\}}{\Wb}^{\{m\}}\Bigr)\, h{f}^{\{m\}}({y}_n) \\
	&&\quad  + \sum_{m=1}^{\Pn}{\sum_{j=2}^{i_Y^{\{m\}}} {a}_{i,j}^{\{q,m\}} {\Psi}_{i,j}^{\{q,m\}}\Bigl(h{g}_{i,j}^{\{q,m\}}{\Wb}^{\{m\}}\Bigr)\,  h\Delta_{{Y}^{\{q\}}}^{(j - 1)}{f}^{\{m\}}\big({y}_n\big)}, \\
	&&\qquad\qquad\qquad i = 1, \hdots, s^{\{q\}} - 1, \; q =1,\hdots, \mathcal{P}, \nonumber\\
	{y}_{n + 1} &=& {y}_n + \sum_{m=1}^{\Pn}{b}_{1}^{\{m\}} {\Psi}_{s^{\{m\}},1}^{\{m, m\}}\Bigl(h{g}_{s^{\{m\}},1}^{\{m,m\}}{\Wb}^{\{m\}}\Bigr)\, h{f}^{\{m\}}({y}_n) \\
	&& \quad + \sum_{m=1}^{\Pn}{\sum_{j=2}^{s^{\{m\}}} {b}_{j}^{\{m\}} {\Psi}_{{s^{\{m\}}},j}^{\{m, m\}}\Bigl(h{g}_{{s^{\{m\}}},j}^{\{m,m\}}{\Wb}^{\{m\}}\Bigr)\,  h\Delta_{{Y}^{\{m\}}}^{(j - 1)}{f}^{\{m\}}\big({y}_n\big)}.
\end{eqnarray*}
\caption{Partitioned EPIRK-W method (PEPIRKW) of GAXP type.}
\label{frm:gark_type_epirkw_like_method}.
\end{formulation}

Unlike Formulation \ref{frm:gark_sepirk_method}, Formulation \ref{frm:gark_type_epirkw_like_method} can generate all mixed derivatives in a Taylor expansion of the numerical solution. This can be verified by noting that the power series representation of $\Psi$ functions has a scaled identity as the first term which in turn gives rise to all mixed derivatives of the partitioning functions as does the underlying GARK method. Consequently, higher order methods are achievable with Formulation \ref{frm:gark_type_epirkw_like_method}.

The methods proposed in this section are by no means an exhaustive list of  partitioned exponential methods of EXP and EPIRK type, but rather a sampling of different ways to structure computations for partitioned systems given the shortcomings in deriving one from the IVP, as established in Section \ref{sec:partitioned-problems}. We now turn our attention to the mechanics of actual method building of the types discussed.

%%%%%%%%%%%%%%%%%%%%%%%%%%%%%%%%%%%%%%%

%%%%%%%%%%%%%%%%%%%%%%%%%%%%%%%%%%%%%%%
% !TEX root = Nonstiff_pexpw.tex
%%%%%%%%%%%%%%%%%%%%%%%%%%%%%%%%%%%%%%%
\section{Order Conditions}
\label{sec:order}
%%%%%%%%%%%%%%%%%%%%%%%%%%%%%%%%%%%%%%%
To build a new time integration method in the formulations proposed in the earlier section, one must take successive derivatives and construct the Taylor series expansions of the exact and numerical solutions, and match them up to the desired order. Equating the coefficients of the two expansions gives rise to `classical' (nonstiff) order conditions that are solved to determine the coefficients of the scheme. When constructing methods of a high order, repeated differentiation of the expression for the numerical approximation can get cumbersome. We will employ Butcher's rooted tree and B-series formalism to ease this process \cite{Hairer_book_I,calvo1994,hairer2006,butcher2009,Butcher_2016_book,Butcher_2011_exp-order,luan2013,Rainwater_2014_semilinear,Sandu_2014_expK,Sandu_2019_EPIRKW}.

%%%%%%%%%%%%%%%%%%%%%%%%%%%%%%%%%%%%%%%
\subsection{TPS-trees and the corresponding B-series}
%%%%%%%%%%%%%%%%%%%%%%%%%%%%%%%%%%%%%%%

Rooted trees graphically represent the elementary differentials that arise during Taylor expansion of both the exact solution and the numerical approximation \cite{Hairer_book_I,Butcher_2016_book,butcher1972}. The formulation of the numerical approximation and the form of the right-hand side function  (such as non-split, linearized about the current state, and split using a linear operator) together determine the structure of the elementary differentials, and consequently the rooted trees. For example, the rooted trees of an unpartitioned Runge--Kutta method are T-trees \cite{Hairer_book_I,Butcher_2016_book,butcher2009}. An expansion defined using rooted trees as the basis is called B-series \cite{Hairer_book_I}.

\begin{definition}[B-series \cite{Hairer_book_I,hairer2006,chartier2005,butcher2009}] A B-series $\mathsf{a}: \mathcal{T} \cup \{\emptyset\} \mapsto \R$ is a mapping from the set of rooted trees $\mathcal{T}$ and the empty tree $(\emptyset)$, to real numbers:
	\begin{eqnarray*}
		B(\mathsf{a}, y) &=& \mathsf{a}(\emptyset)\, y + \displaystyle\sum_{\tau \in \mathcal{T}} \mathsf{a}(\tau)\, \frac{h^{\abs{\tau}}}{\sigma(\tau)}\, F(\tau)(y) {.}
	\end{eqnarray*}
Here $h$ is the timestep where the numerical method approximates the exact solution, $y_{n+1} \approx y(t_n + h)$.
For each tree $\tau \in \mathcal{T} \cup \{\emptyset\}$, $\abs{\tau}$ is the order of the tree, and corresponds to the number of nodes in the tree, and $\sigma(\tau)$ is the order of the symmetry group associated with the tree.
$F(\tau)(y)$ is the elementary differential corresponding to the tree, evaluated at $y$.
\end{definition}

We consider the initial value problem \eqref{eqn:functional_splitting_ode_system} with a two-way partitioned system, $P=2$. To track higher derivatives of  $\fone$ and $\ftwo$ one needs two different colors for the nodes. Further, to account for the linearized partition \eqref{eqn:linearized_rhs_funcs}, one needs two types of nodes for each component, one representing the linear parts and the other derivatives of the nonlinear parts. Consequently, the set of Butcher trees that represent the solutions of the partitioned system \eqref{eqn:functional_splitting_ode_system} with $P=2$ consists of four different node types.

\begin{definition}[TPS-trees] The set of TPS-trees is a generalized set of T-trees with four different types of nodes, chosen from the set $\{\CIRCLE, \Circle, \blacksquare, \square\}$. The tree structures are constrained as follows:
	\begin{itemize}
		\item[i.] Each node in the tree can be either a square or a round node, $\textnormal{node} \in \{\blacksquare, \square\, \CIRCLE, \Circle\}$, however 
		\item[ii.] Square nodes have only one child, i.e., $\textnormal{nodes} \in \{\blacksquare, \square\}$ are singly branched.
	\end{itemize}
	\qed
\end{definition}

The elementary differentials associated with each trees are constructed differently, depending on the type of method we consider, as discussed in Section \ref{sec:split-vs-W}:
\begin{itemize}
\item For partitioned-split methods, the square nodes nodes $\blacksquare$ and $\square$ represent the action of the linear parts $\Lone$ and $\Ltwo$, respectively. The round nodes $\CIRCLE$ and $\Circle$ represent the non-linear parts $\None$ and $\Ntwo$, respectively. We use this formalism for deriving partitioned methods built by averaging unpartitioned counterparts, as described in  \ref{sec:AVG-SEPIRK}.
\item In the case of partitioned W-methods, the square nodes $\blacksquare$ and $\square$ represent the action of the approximate Jacobians $\Wone$ and $\Wtwo$, respectively. The round nodes $\CIRCLE$ and $\Circle$ represent the component functions $\fone$ and $\ftwo$, respectively. We use this formalism to derive partitioned methods of PEXPW and PEPIRKW type, as discussed in sections  \ref{sec:PEXPW} and \ref{sec:PEPIRKW}, respectively.
\end{itemize}

Before continuing our discussion of order conditions for partitioned exponential methods, we revisit the definition of the $B^{\#}$ operator, which takes a B-series and returns the set of coefficients ordered sequentially over the rooted trees.  
\begin{definition}[The $B^{\#}$ operator \cite{Sandu_2014_expK}] Let $B(\mathsf{a}, y)$ be a B-series, then the operator $B^{\#}$ is defined as follows: 
	\begin{eqnarray*}
		B^{\#}\left(B(\mathsf{a}, y)\right) &=& \mathsf{a}.
	\end{eqnarray*}
	\qed
\end{definition}

The TPS trees up to order three are shown in 
\iflong
Tables \ref{table:tps_trees1}--\ref{table:tps_trees11} of \ref{sec:TPS}. 
\else
{\cite[Tables F1--F11]{narayanamurthi2019}.} 
\fi
The tables show the tree geometry, and provide the following information for each tree $\tau$:
\begin{itemize}
\item $F(\tau)(y)$: the elementary differential corresponding to the tree, evaluated at $y$ in tensor notation (similarly structured for split and W-methods);
% \item coefficients of an arbitrary B-series as a function of the tree, $\mathsf{a}(\tau)$;
\item $\mathsf{a}(\tau)$: the coefficient of B-series $B(\mathsf{a}, y)$ corresponding to the tree $\tau$;
% \item $B^{\#}$ applied to the multiplication of the arbitrary B-series, $B(\mathsf{a}(\tau), y)$, by linear operators corresponding to the nodes $\blacksquare$ and $\square$;
\item $B^{\#}\left(h\, \Lb\, B(\mathsf{a}(\tau), y)\right)$ and $B^{\#}\left(h\, \Mb\, B(\mathsf{a}(\tau), y)\right)$:  coefficients corresponding to the tree $\tau$, of the result of  multiplying the arbitrary B-series $B\left(\mathsf{a}(\tau), y\right)$ by linear operators represented by the nodes $\blacksquare$ and $\square$, respectively;
% \item $B^{\#}$ applied to the evaluation of functions corresponding to the nodes $\CIRCLE$ and $\Circle$ on the arbitrary B-series $B(\mathsf{a}(\tau), y)$; 
\item $B^{\#}\left(h\, \N(B\left(\mathsf{a}(\tau), y\right)\right)$ and $B^{\#}\left(h\, \P(B\left(\mathsf{a}(\tau), y\right)\right)$: coefficients corresponding to the tree $\tau$ of the result of evaluating functions represented by nodes $\CIRCLE$ and $\Circle$ on the arbitrary B-series $B(\mathsf{a}(\tau), y)$; 
\item $1/\gamma_{\textsc{s}}(\tau)$ and $1/\gamma_{\textsc{w}}(\tau)$: the coefficients of the TPS-tree $\tau$ in the B-series expansion of the exact solution for split and W-methods, respectively.
\end{itemize}

\begin{remark}[Interpretation of the trees]
Although the interpretation of the nodes $\{\CIRCLE, \Circle,  \blacksquare, \square\}$ differs between split methods and W-methods, the definition of the TPS-tree operations in 
\iflong
Tables \ref{table:tps_trees1}--\ref{table:tps_trees11} of \ref{sec:TPS} 
\else
{\cite[Tables F1--F11]{narayanamurthi2019} }
\fi
remains the same. 
%The interpretation of the elementary differentials and the coefficients of the exact solution together account for this difference.
In
\iflong
Tables \ref{table:tps_trees1}--\ref{table:tps_trees11} of \ref{sec:TPS}, 
\else
{\cite[Tables F1--F11]{narayanamurthi2019}, }
\fi
we use the symbols $\Lb$ and $\Mb$ to represent the linear parts of the two components and the corresponding nodes are $\{\blacksquare, \square\}$, respectively. The symbols $\N$ and $\P$ represent the nonlinear parts of the two components and the corresponding nodes  are $\{\CIRCLE, \Circle\}$. The corresponding quantities for each case are shown in Tables \ref{table:equivalence_split_methods} and \ref{table:equivalence_W_methods}.

\begin{table}[tbh!]
	\parbox{.45\linewidth}{
		\centering
		\begin{tabular}{|c|c|c|}
			\hline
			Node & %
			\iflong
			Tables \ref{table:tps_trees1}--\ref{table:tps_trees11} of \ref{sec:TPS} 
			\else
			{\cite[Tables F1--F11]{narayanamurthi2019} }
			\fi
			& The quantity \\
			& notation & it represents \\
			\hline
			\hline
			$\blacksquare$ & $\Lb$ & $\Lone$ \\
			\hline
			$\square$ & $\Mb$ & $\Ltwo$ \\
			\hline
			$\CIRCLE$ & $\N$ & $\None$ \\
			\hline
			$\Circle$ & $\P$ & $\Ntwo$ \\
			\hline
		\end{tabular}
		\caption{Notation for Split-RHS methods\label{table:equivalence_split_methods}}
	}
	\hfill
	\parbox{.45\linewidth}{
		\centering
		\begin{tabular}{|c|c|c|}
			\hline
			Node & 
			\iflong
			Tables \ref{table:tps_trees1}--\ref{table:tps_trees11} of \ref{sec:TPS} 
			\else
			{\cite[Tables F1--F11]{narayanamurthi2019} }
			\fi
			& The quantity \\
			& notation & it represents \\
			\hline
			\hline
			$\blacksquare$ & $\Lb$ & $\Wone$ \\
			\hline
			$\square$ & $\Mb$ & $\Wtwo$ \\
			\hline
			$\CIRCLE$ & $\N$ & $\fone$ \\
			\hline
			$\Circle$ & $\P$ & $\ftwo$ \\
			\hline
		\end{tabular}
		\caption{Notation for W methods\label{table:equivalence_W_methods}}
	}
\end{table}
\end{remark}

\begin{remark}[Trees ending in a square node]
The interpretation of TPS-trees ending in a square node is as follows:
\begin{itemize}
\item for a split-RHS method, if a TPS-tree ends in a square node,  $\blacksquare$ (or $\square$), it corresponds to the appearance of an $\Lone y$ (or $\Ltwo y$) term at the appropriate location in the elementary differential. 
\item for a W method, if a TPS-tree ends in a square node, $\blacksquare$ (or $\square$), it corresponds to the appearance of a $\Wone y$ (or $\Wtwo y$) term at the appropriate location in the  elementary differential. Terms containing $\Wone y$ (or $\Wtwo y$) do not appear in the Taylor expansion of the exact or the numerical solution. Consequentially, trees ending in a square node carry a zero B-series coefficient in both exact and numerical solutions.
\end{itemize}
\end{remark}

%%%%%%%%%%%%%%%%%%%%%%%
\subsection{B-series of the exact solution} 
%%%%%%%%%%%%%%%%%%%%%%%

The last two rows of 
\iflong
Tables \ref{table:tps_trees1}--\ref{table:tps_trees11} of \ref{sec:TPS} 
\else
{\cite[Tables F1--F11]{narayanamurthi2019} }
\fi
give the B-series of the exact solution expressed using TPS-trees ($1/\gamma_{\textsc{s}}(\tau)$ or $1/\gamma_{\textsc{w}}(\tau)$, depending on the type of method, split-RHS or W, respectively). These coefficients are derived by mapping T-trees to the corresponding TPS-trees, together with the coefficients of the B-series expansion of the exact solution. 

For instance, the elementary differential $F_y\, F$, corresponding to a single T-tree, maps to sixteen elementary differentials of a partitioned split-RHS method of the form \eqref{eqn:mixed-S-derivatives}. Each of these elementary differentials corresponds to a TPS-tree, and the B-series coefficient of the exact solution for each of the sixteen TPS-trees will be the same coefficient as on the T-tree corresponding to $F_y\, F$. These mappings can be derived by Taylor expanding the true solution in elementary differentials corresponding to T-trees and replacing the derivatives with those when the function is split as is done in Lemma \ref{lemma:drawback_of_sepirk_method} for the mixed derivatives case.

For W-methods, the TPS-trees are in fact a superset of the P-trees \cite[Section II.15]{Hairer_book_I}. The B-series coefficients of the exact solution for trees that do not contain a $\Wb$ node are the same as the coefficients of the equivalent P-trees. The B-series coefficients of the exact solution for trees that  contain a $\Wb$ node are zero, as these trees do not show up in the expansion of the true solution.

%%%%%%%%%%%%%%%%%%%%%%%%%%%%%%%%%%%%%%%
\subsection{Operations with TPS-trees and the corresponding B-series}
%%%%%%%%%%%%%%%%%%%%%%%%%%%%%%%%%%%%%%%

We use the linearity property of B-series and the $B^{\#}$ operator, and the TPS operations defined below, 
%in the Tables \ref{table:tps_trees1}--\ref{table:tps_trees11} 
to derive the B-series expansion of the numerical solutions via an algorithmic procedure similar to that outlined in \cite{Sandu_2014_expK,Sandu_2019_EPIRKW}.

\begin{definition}
The operator $\Tree$ maps an elementary differential to the corresponding TPS-tree:

\[
	\Tree: \;\textnormal{elementary differential}\; \mapsto \;\textnormal{TPS-tree}
\]	
\end{definition}

\begin{example}
	\label{example:elemetary_differential_to_tree_function}
	The application of operator $\Tree$ on elementary differentials for split-RHS and W-methods is illustrated below:
	\begin{table}[h!]
		\centering
		\begin{tabular}{ll|ll}
		\multicolumn{2}{c|}{Split-RHS} & \multicolumn{2}{c}{W}\\
		\hline
		$\Tree\bigl(\Lone y\bigr)$ & $\leftrightarrow$ \quad \raisebox{-4pt}{$\begin{tikzpicture}[scale=.5,
				meagrecirc/.style={circle,draw, fill=black!100,thick},
				fatcirc/.style={circle,draw,thick},
				fatrect/.style = {rectangle,draw,thick},
				meagrerect/.style = {rectangle,fill=black!100,draw,thick}]
			\node[meagrerect] (j) at (0,0) [label=right:$j$] {};
		\end{tikzpicture}$} & 
		$\Tree\bigl(\Wone y\bigr)$ & $\leftrightarrow$ \quad  \raisebox{-4pt}{$\begin{tikzpicture}[scale=.5,
			meagrecirc/.style={circle,draw, fill=black!100,thick},
			fatcirc/.style={circle,draw,thick},
			fatrect/.style = {rectangle,draw,thick},
			meagrerect/.style = {rectangle,fill=black!100,draw,thick}]
		\node[meagrerect] (j) at (0,0) [label=right:$j$] {};
		\end{tikzpicture}$} \\
		$\Tree\bigl(\None(y)\bigr)$ & $\leftrightarrow$ \quad \raisebox{-4pt}{$\begin{tikzpicture}[scale=.5,
				meagrecirc/.style={circle,draw, fill=black!100,thick},
				fatcirc/.style={circle,draw,thick},
				fatrect/.style = {rectangle,draw,thick},
				meagrerect/.style = {rectangle,fill=black!100,draw,thick}]
			\node[meagrecirc] (j) at (0,0) [label=right:$j$] {};
		\end{tikzpicture}$} &
		$\Tree\bigl(\fone(y)\bigr)$ & $\leftrightarrow$ \quad \raisebox{-4pt}{$\begin{tikzpicture}[scale=.5,
			meagrecirc/.style={circle,draw, fill=black!100,thick},
			fatcirc/.style={circle,draw,thick},
			fatrect/.style = {rectangle,draw,thick},
			meagrerect/.style = {rectangle,fill=black!100,draw,thick}]
		\node[meagrecirc] (j) at (0,0) [label=right:$j$] {};
		\end{tikzpicture}$} \\
		$\Tree\bigl(\Ltwo y\bigr)$ & $\leftrightarrow$ \quad \raisebox{-4pt}{$\begin{tikzpicture}[scale=.5,
			meagrecirc/.style={circle,draw, fill=black!100,thick},
			fatcirc/.style={circle,draw,thick},
			fatrect/.style = {rectangle,draw,thick},
			meagrerect/.style = {rectangle,fill=black!100,draw,thick}]
		\node[fatrect] (j) at (0,0) [label=right:$j$] {};
		\end{tikzpicture}$} &
		$\Tree\bigl(\Wtwo y\bigr)$ & $\leftrightarrow$ \quad \raisebox{-4pt}{$\begin{tikzpicture}[scale=.5,
			meagrecirc/.style={circle,draw, fill=black!100,thick},
			fatcirc/.style={circle,draw,thick},
			fatrect/.style = {rectangle,draw,thick},
			meagrerect/.style = {rectangle,fill=black!100,draw,thick}]
		\node[fatrect] (j) at (0,0) [label=right:$j$] {};
		\end{tikzpicture}$} \\
		$\Tree\bigl(\Ntwo(y)\bigr)$ & $\leftrightarrow$ \quad \raisebox{-4pt}{$\begin{tikzpicture}[scale=.5,
			meagrecirc/.style={circle,draw, fill=black!100,thick},
			fatcirc/.style={circle,draw,thick},
			fatrect/.style = {rectangle,draw,thick},
			meagrerect/.style = {rectangle,fill=black!100,draw,thick}]
		\node[fatcirc] (j) at (0,0) [label=right:$j$] {};
		\end{tikzpicture}$} &
		$\Tree\bigl(\ftwo(y)\bigr)$ & $\leftrightarrow$ \quad \raisebox{-4pt}{$\begin{tikzpicture}[scale=.5,
			meagrecirc/.style={circle,draw, fill=black!100,thick},
			fatcirc/.style={circle,draw,thick},
			fatrect/.style = {rectangle,draw,thick},
			meagrerect/.style = {rectangle,fill=black!100,draw,thick}]
		\node[fatcirc] (j) at (0,0) [label=right:$j$] {};
		\end{tikzpicture}$} \\
		$\Tree\bigl(\None_y(y) \Ltwo y\bigr)$ & $\leftrightarrow$ \quad \raisebox{-4pt}{$\begin{tikzpicture}[scale=.5,
			meagrecirc/.style={circle,draw, fill=black!100,thick},
			fatcirc/.style={circle,draw,thick},
			fatrect/.style = {rectangle,draw,thick},
			meagrerect/.style = {rectangle,fill=black!100,draw,thick}]
		\node[meagrecirc] (j) at (0,0) [label=right:$j$] {};
		\node[fatrect] (k) at (1,1) [label=right:$k$] {};
		\draw[-] (j) -- (k);
		\end{tikzpicture}$} & 
		$\Tree\bigl(\fone_y(y) \Wtwo y\bigr)$ & $\leftrightarrow$ \quad \raisebox{-4pt}{$\begin{tikzpicture}[scale=.5,
			meagrecirc/.style={circle,draw, fill=black!100,thick},
			fatcirc/.style={circle,draw,thick},
			fatrect/.style = {rectangle,draw,thick},
			meagrerect/.style = {rectangle,fill=black!100,draw,thick}]
		\node[meagrecirc] (j) at (0,0) [label=right:$j$] {};
		\node[fatrect] (k) at (1,1) [label=right:$k$] {};
		\draw[-] (j) -- (k);
		\end{tikzpicture}$}
	\end{tabular}
	\end{table}
\end{example}

\FloatBarrier

\begin{definition}
The operator $\RemoveRoot^{[\ell]}_{\zeta}(\tau)$ removes $\ell$ times the node $\zeta$ from the root of tree $\tau$, to produce a resultant tree. The operation is only defined on trees where such removal is possible.
\end{definition}

\begin{example}
	\label{example:removeroot_function}
	The application of operator $\RemoveRoot^{\{\ell\}}_{\zeta}$ to TPS trees for split-RHS and W-methods is illustrated below:
	\begin{table}[h!]
		\centering
		\begin{tabular}{ll|ll}
		\multicolumn{2}{c}{Split-RHS} & \multicolumn{2}{c}{W}\\
		$\RemoveRoot^{\{1\}}_{\Lone}\left(\Tree\bigl(\Lone y\bigr)\right)$ & $\leftrightarrow$ \quad \raisebox{-4pt}{$\begin{tikzpicture}[scale=.5,
				meagrecirc/.style={circle,draw, fill=black!100,thick},
				fatcirc/.style={circle,draw,thick},
				fatrect/.style = {rectangle,draw,thick},
				meagrerect/.style = {rectangle,fill=black!100,draw,thick}]
			$\emptyset$ (j) at (0,0) [label=right:$$] {};
		\end{tikzpicture}$} & 
		$\RemoveRoot^{\{1\}}_{\Wone}\left(\Tree\bigl(\Wone y\bigr)\right)$ & $\leftrightarrow$ \quad \raisebox{-4pt}{$\begin{tikzpicture}[scale=.5,
			meagrecirc/.style={circle,draw, fill=black!100,thick},
			fatcirc/.style={circle,draw,thick},
			fatrect/.style = {rectangle,draw,thick},
			meagrerect/.style = {rectangle,fill=black!100,draw,thick}]
		$\emptyset$ (j) at (0,0) [label=right:$$] {};
		\end{tikzpicture}$} \\
		$\RemoveRoot^{\{1\}}_{\Lone}\left(\Tree\bigl(\Lone\Lone(y)\bigr)\right)$ & $\leftrightarrow$ \quad \raisebox{-4pt}{$\begin{tikzpicture}[scale=.5,
				meagrecirc/.style={circle,draw, fill=black!100,thick},
				fatcirc/.style={circle,draw,thick},
				fatrect/.style = {rectangle,draw,thick},
				meagrerect/.style = {rectangle,fill=black!100,draw,thick}]
			\node[meagrerect] (j) at (0,0) [label=right:$j$] {};
		\end{tikzpicture}$} &
		$\RemoveRoot^{\{1\}}_{\Wone}\left(\Tree\bigl(\Wone\Wone(y)\bigr)\right)$ & $\leftrightarrow$ \quad \raisebox{-4pt}{$\begin{tikzpicture}[scale=.5,
			meagrecirc/.style={circle,draw, fill=black!100,thick},
			fatcirc/.style={circle,draw,thick},
			fatrect/.style = {rectangle,draw,thick},
			meagrerect/.style = {rectangle,fill=black!100,draw,thick}]
		\node[meagrerect] (j) at (0,0) [label=right:$j$] {};
		\end{tikzpicture}$} \\
		$\RemoveRoot^{\{1\}}_{\Ltwo}(\Tree\bigl(\Ltwo y\bigr))$ & $\leftrightarrow$ \quad \raisebox{-4pt}{$\begin{tikzpicture}[scale=.5,
			meagrecirc/.style={circle,draw, fill=black!100,thick},
			fatcirc/.style={circle,draw,thick},
			fatrect/.style = {rectangle,draw,thick},
			meagrerect/.style = {rectangle,fill=black!100,draw,thick}]
		$\emptyset$ (j) at (0,0) [label=right:$$] {};
		\end{tikzpicture}$} &
		$\RemoveRoot^{\{1\}}_{\Wtwo}(\Tree\bigl(\Wtwo y\bigr))$ & $\leftrightarrow$ \quad \raisebox{-4pt}{$\begin{tikzpicture}[scale=.5,
			meagrecirc/.style={circle,draw, fill=black!100,thick},
			fatcirc/.style={circle,draw,thick},
			fatrect/.style = {rectangle,draw,thick},
			meagrerect/.style = {rectangle,fill=black!100,draw,thick}]
		$\emptyset$ (j) at (0,0) [label=right:$$] {};
		\end{tikzpicture}$} \\
		$\RemoveRoot^{\{1\}}_{\Ltwo}(\Tree\bigl(\Ltwo\Ltwo(y)\bigr))$ & $\leftrightarrow$ \quad \raisebox{-4pt}{$\begin{tikzpicture}[scale=.5,
			meagrecirc/.style={circle,draw, fill=black!100,thick},
			fatcirc/.style={circle,draw,thick},
			fatrect/.style = {rectangle,draw,thick},
			meagrerect/.style = {rectangle,fill=black!100,draw,thick}]
		\node[fatrect] (j) at (0,0) [label=right:$j$] {};
		\end{tikzpicture}$} &
		$\RemoveRoot^{\{1\}}_{\Wtwo}(\Tree\bigl(\Wtwo\Wtwo(y)\bigr))$ & $\leftrightarrow$ \quad \raisebox{-4pt}{$\begin{tikzpicture}[scale=.5,
			meagrecirc/.style={circle,draw, fill=black!100,thick},
			fatcirc/.style={circle,draw,thick},
			fatrect/.style = {rectangle,draw,thick},
			meagrerect/.style = {rectangle,fill=black!100,draw,thick}]
		\node[fatrect] (j) at (0,0) [label=right:$j$] {};
		\end{tikzpicture}$} \\
		$\RemoveRoot^{\{1\}}_{\Ltwo}(\Tree\bigl(\Ltwo \None(y)\bigr))$ & $\leftrightarrow$ \quad \raisebox{-4pt}{$\begin{tikzpicture}[scale=.5,
			meagrecirc/.style={circle,draw, fill=black!100,thick},
			fatcirc/.style={circle,draw,thick},
			fatrect/.style = {rectangle,draw,thick},
			meagrerect/.style = {rectangle,fill=black!100,draw,thick}]
		\node[meagrecirc] (j) at (0,0) [label=right:$j$] {};
		\end{tikzpicture}$} & 
		$\RemoveRoot^{\{1\}}_{\Wtwo}(\Tree\bigl(\Wtwo \fone(y)\bigr))$ & $\leftrightarrow$ \quad \raisebox{-4pt}{$\begin{tikzpicture}[scale=.5,
			meagrecirc/.style={circle,draw, fill=black!100,thick},
			fatcirc/.style={circle,draw,thick},
			fatrect/.style = {rectangle,draw,thick},
			meagrerect/.style = {rectangle,fill=black!100,draw,thick}]
		\node[meagrecirc] (j) at (0,0) [label=right:$j$] {};
		\end{tikzpicture}$} \\
		$\RemoveRoot^{\{1\}}_{\Ltwo}(\Tree\bigl(\Ltwo \Ltwo \None(y)\bigr))$ & $\leftrightarrow$ \quad \raisebox{-4pt}{$\begin{tikzpicture}[scale=.5,
			meagrecirc/.style={circle,draw, fill=black!100,thick},
			fatcirc/.style={circle,draw,thick},
			fatrect/.style = {rectangle,draw,thick},
			meagrerect/.style = {rectangle,fill=black!100,draw,thick}]
		\node[fatrect] (k) at (1,1) [label=right:$j$] {};
		\node[meagrecirc] (l) at (2,2) [label=right:$k$] {};
		\draw[-] (k) -- (l);
		\end{tikzpicture}$} & 
		$\RemoveRoot^{\{1\}}_{\Wtwo}(\Tree\bigl(\Wtwo \Wtwo \fone(y)\bigr))$ & $\leftrightarrow$ \quad \raisebox{-4pt}{$\begin{tikzpicture}[scale=.5,
			meagrecirc/.style={circle,draw, fill=black!100,thick},
			fatcirc/.style={circle,draw,thick},
			fatrect/.style = {rectangle,draw,thick},
			meagrerect/.style = {rectangle,fill=black!100,draw,thick}]
		\node[fatrect] (k) at (1,1) [label=right:$j$] {};
		\node[meagrecirc] (l) at (2,2) [label=right:$k$] {};
		\draw[-] (k) -- (l);
		\end{tikzpicture}$} \\
		$\RemoveRoot^{\{2\}}_{\Ltwo}(\Tree\bigl(\Ltwo \Ltwo \None(y)\bigr))$ & $\leftrightarrow$ \quad \raisebox{-4pt}{$\begin{tikzpicture}[scale=.5,
			meagrecirc/.style={circle,draw, fill=black!100,thick},
			fatcirc/.style={circle,draw,thick},
			fatrect/.style = {rectangle,draw,thick},
			meagrerect/.style = {rectangle,fill=black!100,draw,thick}]
		\node[meagrecirc] (l) at (2,2) [label=right:$j$] {};
		\end{tikzpicture}$} & 
		$\RemoveRoot^{\{2\}}_{\Wtwo}(\Tree\bigl(\Wtwo \Wtwo \fone(y)\bigr))$ & $\leftrightarrow$ \quad \raisebox{-4pt}{$\begin{tikzpicture}[scale=.5,
			meagrecirc/.style={circle,draw, fill=black!100,thick},
			fatcirc/.style={circle,draw,thick},
			fatrect/.style = {rectangle,draw,thick},
			meagrerect/.style = {rectangle,fill=black!100,draw,thick}]
		\node[meagrecirc] (l) at (2,2) [label=right:$j$] {};
		\end{tikzpicture}$} \\
		$\RemoveRoot^{\{1\}}_{\Lone}(\Tree\bigl(\Ltwo \Ltwo \None(y)\bigr))$ & $\leftrightarrow$ \quad \raisebox{-4pt}{$\textnormal{undefined}$} & 
		$\RemoveRoot^{\{1\}}_{\Wone}(\Tree\bigl(\Wtwo \Wtwo \fone(y)\bigr))$ & $\leftrightarrow$ \quad \raisebox{-4pt}{$\textnormal{undefined}$} 
	\end{tabular}
	\end{table}
\end{example}

\FloatBarrier

\begin{definition}
	The operator $\varrho(\tau)$ returns the root node of the tree $\tau$. 
\end{definition}

%\begin{remark}
%	In the examples above, the root is labeled $j$. As a convention, tree roots are labeled $j$, subsequent children take labels $k, l, \hdots$ \cite{Hairer_book_I}. One may also think of $\varrho$ as returning a single node tree corresponding to the root of the argument tree.
%\end{remark}

\begin{example}
	\label{example:elementary_differential_to_tree_function}
	The application of operator $\varrho$ to TPS trees for split-RHS and W-methods is illustrated below:
	\begin{table}[h!]
		\centering
		\begin{tabular}{ll|ll}
		\multicolumn{2}{c}{Split-RHS} & \multicolumn{2}{c}{W}\\
		$\varrho(\Tree\bigl(\Lone y\bigr))$ & $\leftrightarrow$ \quad \raisebox{-4pt}{$\begin{tikzpicture}[scale=.5,
				meagrecirc/.style={circle,draw, fill=black!100,thick},
				fatcirc/.style={circle,draw,thick},
				fatrect/.style = {rectangle,draw,thick},
				meagrerect/.style = {rectangle,fill=black!100,draw,thick}]
			\node[meagrerect] (j) at (0,0) [label=right:$$] {};
		\end{tikzpicture}$} & 
		$\varrho(\Tree\bigl(\Wone y\bigr))$ & $\leftrightarrow$ \quad \raisebox{-4pt}{$\begin{tikzpicture}[scale=.5,
			meagrecirc/.style={circle,draw, fill=black!100,thick},
			fatcirc/.style={circle,draw,thick},
			fatrect/.style = {rectangle,draw,thick},
			meagrerect/.style = {rectangle,fill=black!100,draw,thick}]
		\node[meagrerect] (j) at (0,0) [label=right:$$] {};
		\end{tikzpicture}$} \\
		$\varrho(\Tree\bigl(\None(y)\bigr))$ & $\leftrightarrow$ \quad \raisebox{-4pt}{$\begin{tikzpicture}[scale=.5,
				meagrecirc/.style={circle,draw, fill=black!100,thick},
				fatcirc/.style={circle,draw,thick},
				fatrect/.style = {rectangle,draw,thick},
				meagrerect/.style = {rectangle,fill=black!100,draw,thick}]
			\node[meagrecirc] (j) at (0,0) [label=right:$$] {};
		\end{tikzpicture}$} &
		$\varrho(\Tree\bigl(\fone(y)\bigr))$ & $\leftrightarrow$ \quad \raisebox{-4pt}{$\begin{tikzpicture}[scale=.5,
			meagrecirc/.style={circle,draw, fill=black!100,thick},
			fatcirc/.style={circle,draw,thick},
			fatrect/.style = {rectangle,draw,thick},
			meagrerect/.style = {rectangle,fill=black!100,draw,thick}]
		\node[meagrecirc] (j) at (0,0) [label=right:$$] {};
		\end{tikzpicture}$} \\
		$\varrho(\Tree\bigl(\Ltwo y\bigr))$ & $\leftrightarrow$ \quad \raisebox{-4pt}{$\begin{tikzpicture}[scale=.5,
			meagrecirc/.style={circle,draw, fill=black!100,thick},
			fatcirc/.style={circle,draw,thick},
			fatrect/.style = {rectangle,draw,thick},
			meagrerect/.style = {rectangle,fill=black!100,draw,thick}]
		\node[fatrect] (j) at (0,0) [label=right:$$] {};
		\end{tikzpicture}$} &
		$\varrho(\Tree\bigl(\Wtwo y\bigr))$ & $\leftrightarrow$ \quad \raisebox{-4pt}{$\begin{tikzpicture}[scale=.5,
			meagrecirc/.style={circle,draw, fill=black!100,thick},
			fatcirc/.style={circle,draw,thick},
			fatrect/.style = {rectangle,draw,thick},
			meagrerect/.style = {rectangle,fill=black!100,draw,thick}]
		\node[fatrect] (j) at (0,0) [label=right:$$] {};
		\end{tikzpicture}$} \\
		$\varrho(\Tree\bigl(\Ntwo(y)\bigr))$ & $\leftrightarrow$ \quad \raisebox{-4pt}{$\begin{tikzpicture}[scale=.5,
			meagrecirc/.style={circle,draw, fill=black!100,thick},
			fatcirc/.style={circle,draw,thick},
			fatrect/.style = {rectangle,draw,thick},
			meagrerect/.style = {rectangle,fill=black!100,draw,thick}]
		\node[fatcirc] (j) at (0,0) [label=right:$$] {};
		\end{tikzpicture}$} &
		$\varrho(\Tree\bigl(\ftwo(y)\bigr))$ & $\leftrightarrow$ \quad \raisebox{-4pt}{$\begin{tikzpicture}[scale=.5,
			meagrecirc/.style={circle,draw, fill=black!100,thick},
			fatcirc/.style={circle,draw,thick},
			fatrect/.style = {rectangle,draw,thick},
			meagrerect/.style = {rectangle,fill=black!100,draw,thick}]
		\node[fatcirc] (j) at (0,0) [label=right:$$] {};
		\end{tikzpicture}$} \\
		$\varrho(\Tree\bigl(\None_y(y) \Ltwo y\bigr))$ & $\leftrightarrow$ \quad \raisebox{-4pt}{$\begin{tikzpicture}[scale=.5,
			meagrecirc/.style={circle,draw, fill=black!100,thick},
			fatcirc/.style={circle,draw,thick},
			fatrect/.style = {rectangle,draw,thick},
			meagrerect/.style = {rectangle,fill=black!100,draw,thick}]
		\node[meagrecirc] (j) at (0,0) [label=right:$$] {};
		\end{tikzpicture}$} & 
		$\varrho(\Tree\bigl(\fone_y(y) \Wtwo y\bigr))$ & $\leftrightarrow$ \quad \raisebox{-4pt}{$\begin{tikzpicture}[scale=.5,
			meagrecirc/.style={circle,draw, fill=black!100,thick},
			fatcirc/.style={circle,draw,thick},
			fatrect/.style = {rectangle,draw,thick},
			meagrerect/.style = {rectangle,fill=black!100,draw,thick}]
		\node[meagrecirc] (j) at (0,0) [label=right:$$] {};
		\end{tikzpicture}$} \\
		$\varrho(\Tree\bigl(\Ntwo_y(y) \Ltwo y\bigr))$ & $\leftrightarrow$ \quad \raisebox{-4pt}{$\begin{tikzpicture}[scale=.5,
			meagrecirc/.style={circle,draw, fill=black!100,thick},
			fatcirc/.style={circle,draw,thick},
			fatrect/.style = {rectangle,draw,thick},
			meagrerect/.style = {rectangle,fill=black!100,draw,thick}]
		\node[fatcirc] (j) at (0,0) [label=right:$$] {};
		\end{tikzpicture}$} & 
		$\varrho(\Tree\bigl(\ftwo_y(y) \Wtwo y\bigr))$ & $\leftrightarrow$ \quad \raisebox{-4pt}{$\begin{tikzpicture}[scale=.5,
			meagrecirc/.style={circle,draw, fill=black!100,thick},
			fatcirc/.style={circle,draw,thick},
			fatrect/.style = {rectangle,draw,thick},
			meagrerect/.style = {rectangle,fill=black!100,draw,thick}]
		\node[fatcirc] (j) at (0,0) [label=right:$$] {};
		\end{tikzpicture}$} 
	\end{tabular}
	\end{table}
\end{example}

\FloatBarrier

Lastly, we denote by $\theta^{[\tau_1, \tau_2, \hdots \tau_k]}$ the TPS-tree with root $\theta$ and the subtrees $\tau_1, \tau_2, \hdots, \tau_k$ as the children of the root. Pictorially the tree is represented as follows:

\begin{figure}[h!]
	\centering
	\begin{tikzpicture}[scale=.5,
		meagrecirc/.style={circle,draw, fill=black!100,thick},
		fatcirc/.style={circle,draw,thick},
		fatrect/.style = {rectangle,draw,thick},
		meagrerect/.style = {rectangle,fill=black!100,draw,thick}]
	\node (j) at (0,0) [label=below:$\theta$] {};
	\node (k) at (-3,2) [label=above:$\tau_1$] {};
	\node (l) at (-2,2) [label=above:$\tau_2$] {};
	\node (m) at (0,2) [label=above:$\hdots$] {};
	\node (n) at (2,2) [label=above:$\tau_k$] {};
	\draw[-] (j) -- (k);
	\draw[-] (j) -- (l);
	\draw[-] (j) -- (n);
	\end{tikzpicture}
\end{figure}

\FloatBarrier

\begin{lemma}[Application of a function to a B-series]
	\label{lemma:application_of_function_to_bseries}
The application of a component function to a B-series is another B-series,
\begin{equation*}
	\begin{split}
h\, f^{\{m\}}\bigl(B(\mathsf{a}, y)\bigr) = B\bigl(\mathsf{b}, y\bigr),
	\end{split}
\end{equation*}
whose coefficients are formally denoted by $\mathsf{b} = f^{\{m\}}(\mathsf{a})$ and are computed as follows: 
\begin{equation*}
	f^{\{m\}}(\mathsf{a})(\tau) =  \begin{cases}
		0 &  \tau = \emptyset,\\
		1 &  \tau = \Tree\bigl(f^{\{m\}}\bigr),\\
		\mathsf{a}(\tau_1) \cdot \mathsf{a}(\tau_2) \hdots \mathsf{a}(\tau_k) & \tau = \Tree\bigl(f^{\{m\}}\bigr)^{[\tau_1, \tau_2, \hdots \tau_k]},\\
		0 & \varrho(\tau) \ne \Tree\bigl(f^{\{m\}}\bigr).
		\end{cases}
\end{equation*}
\end{lemma}

\begin{proof}
Similar to \cite[Corollary 2]{Rainwater_2014_semilinear}.
\end{proof}

\begin{lemma}[Application of a matrix to a B-series]
\label{lemma:application_of_matrix_to_bseries}
The application of a matrix to a B-series is another B-series,
\[
h\, \Lb^{\{m\}}\bigl(B(\mathsf{a}, y)\bigr) = B(\mathsf{c}, y),
\]
whose coefficients are formally denoted $\mathsf{c} = \Lb^{\{m\}}\,\mathsf{a}$ and are computed as follows: 
\begin{equation*}
	(\Lb^{\{m\}}\mathsf{a})(\tau) =  \begin{cases}
		\mathsf{a}(\tau_1) &  \tau = \Tree\bigl(\Lb^{\{m\}}\,y\bigr)^{[\tau_1]},\\
		0 & \textnormal{otherwise}.
		\end{cases}
\end{equation*}
\end{lemma}
\begin{proof}
Observe that multiplying by $\Lb^{\{m\}}$ from the left will shift the coefficients from a tree $\tau$ to a tree $\Tree\bigl(\Lb^{\{m\}}\bigr)^{[\tau]}$. Also refer to \cite{Butcher_2011_exp-order,hairer2006}.
\end{proof}

\begin{lemma}[Application of a matrix function to a B-series]
	\label{lemma:application_of_matrixfunction_to_bseries}
Let $\phi$ be an analytical function with a power series expansion:
\[
	\phi(z) = \sum_{i = 0}^{\infty} c_i\, z^i.
\]
The application of the analytical matrix function to a B-series is another B-series,
\[
\phi(h\Lb^{\{m\}})\bigl(B(\mathsf{a}, y)\bigr) = B(\mathsf{c}, y),
\]
whose coefficients are formally denoted $\mathsf{c} = \phi(h\Lb^{\{m\}})\,\mathsf{a}$ and are computed as follows: 
\begin{equation*}
	\bigl(\phi(h\Lb^{\{m\}})\mathsf{a}\bigr)(\tau) =  \begin{cases}
		\sum_{i \ge 0}c_i \cdot \mathsf{a}(\tau_{m_i}) & \tau_{m_i} = \RemoveRoot^{\{i\}}_{\Lb^{\{m\}}}(\tau),\\
		0 & \varrho(\tau) \ne \Tree\bigl(\Lb^{\{m\}}\bigr),
	\end{cases}
\end{equation*}
where $\tau_{m_i}$ is the tree obtained by removing $m$ times $\Tree\bigl(\Lb^{\{m\}}y\bigr)$ from the root position of $\tau$. The coefficients, $c_i$, come from the power series expansion of the matrix function $\phi$. For instance, if $\phi = \varphi_k$, then the coefficient $c_i = \cfrac{1}{(k+i)!}$. The summation is only over those trees $\tau_{m_i}$ where the $\RemoveRoot^{\{i\}}_{\Lb^{\{m\}}}$ function is defined. 
\end{lemma}
\begin{proof}
By applying \ref{lemma:application_of_matrix_to_bseries}, and using the linearity of B-series and the theory from \cite[\S 4]{butcher2009}.
\end{proof}

%%%%%%%%%%%%%%%%%%%%%%%%%
\subsection{B-series of the numerical solution}
%%%%%%%%%%%%%%%%%%%%%%%%%

B-series of the numerical solution can be obtained by starting with the B-series expansion for the true solution at the current time instant, $y_n$, and stepping through an algorithmic procedure that mimics the numerical method for which the B-series is being constructed. Earlier works \cite{Sandu_2014_expK,Sandu_2019_EPIRKW} by the authors demonstrate this procedure. In Algorithm \ref{alg:PEPIRKW_b_series} we show how this is done for PEPIRKW. Similarly, we derive the B-series of the numerical solution for each of the other methods discussed.

\begin{remark}
	We truncate each of the B-series and its operations up to the order of the method that we are interested in building. In this work, we only consider truncated operations for up to order three corresponding. The corresponding TPS-trees are shown in %
	\iflong
	Tables \ref{table:tps_trees1}--\ref{table:tps_trees11} of \ref{sec:TPS}. 
	\else
	{\cite[Tables F1--F11]{narayanamurthi2019}. }
	\fi
\end{remark}

\begin{remark}
	The B-series coefficient of the true solution at the current time instant, $y_n$, has zeros in front of each TPS-tree, and a one in front of the empty tree, $\emptyset$.  
\end{remark}

\begin{algorithm}[tbh!]
	\caption{Computation of the B-series of PEPIRKW numerical solution}
	\label{alg:PEPIRKW_b_series}
	\begin{algorithmic}[1]
	\State  {\bf Input:} $\mathsf{y_{n}}$\Comment{B-series coefficient of the current numerical solution.}
	\For{$i=1:\max[s^{\{.\}}] - 1$}{}{}\Comment{Stage Index: 
	Do for the largest number of stages}
	\For{$q=1:P$}{}{}\Comment{Partition Index: 
	For each stage and partition do the following}
	\State $\mathsf{u} = 0$
	\For{$m=1:P$}
	\State $ \mathsf{v} = B^{\#}(h\,\mbf{f}^{\{m\}}(B(\mathsf{y_{n}},y)))$\Comment{Composition of $f$ with B-series of the current solution.}
	\State $ \mathsf{v} = B^{\#}(\boldsymbol{\psi}_{i,1}^{\{q, m\}}(h\,g_{i,1}^{\{q, m\}}\,\,\Wb^{\{m\}})\,\cdot\,B(\mathsf{v}, y))$\Comment{Multiplication by $\boldsymbol{\psi}$ function}
	\State $\mathsf{u} = \mathsf{u} + a_{i,1}^{\{q, m\}} * \mathsf{v}$\Comment{Scaling by a constant and add}
	\EndFor
	\For{$m=1:P$}
	\For{$j=2:i_Y^{\{m\}}$}{}{}
	\State $ \mathsf{v} = B^{\#}(h\, \Delta^{(j-1)}_{Y^{\{q\}}}\mbf{f}^{\{m\}}(\mathsf{y_{n}}))$\Comment{Recursive forward-difference}
	\State $ \mathsf{v} = B^{\#}(\boldsymbol{\psi}_{i,j}^{\{q, m\}}(h\,g_{i,j}^{\{q, m\}}\,\Wb^{\{m\}})\,\cdot\,B(\mathsf{v}, y))$
	\State $ \mathsf{u} = \mathsf{u} + a_{i,j}^{\{q, m\}}\, * \mathsf{v}$
	\EndFor
	\EndFor
	\State $ \mathsf{\mbf{Y}_i^{\{q\}}} = \mathsf{y_{n}} + \mathsf{u}$\Comment{Addition of two B-series}
	\EndFor
	\EndFor
	\State $\mathsf{u} = 0$
	\For{$m=1:P$}
	\State $ \mathsf{v} = B^{\#}(h\,\mbf{f}^{\{m\}}(B(\mathsf{y_{n}},y)))$
	\State $ \mathsf{v} = B^{\#}(\boldsymbol{\psi}_{s^{\{m\}},1}^{\{m,m\}}(h\,g_{s^{\{m\}},1}^{\{m,m\}}\,\Wb^{\{m\}})\,\cdot\,B(\mathsf{v}, y))$
	\State $ \mathsf{u} = \mathsf{u} + b_{1}^{\{m\}}\, * \mathsf{v}$
	\EndFor
	\For{$m=1:P$}
	\For{$j=2:s^{\{m\}}$}
	\State $ \mathsf{v} = B^{\#}(h\, \Delta^{(j-1)}_{Y^{\{m\}}}\mbf{f}^{\{m\}}(\mathsf{y_{n}}))$
	\State $ \mathsf{v} = B^{\#}(\boldsymbol{\psi}_{s^{\{m\}},j}^{\{m,m\}}(h\,g_{s^{\{m\}},j}^{\{m,m\}}\,\Wb^{\{m\}})\,\cdot\,B(\mathsf{v}, y))$
	\State $ \mathsf{u} = \mathsf{u} + b_{j}^{\{m\}}\, * \mathsf{v}$
	\EndFor
	\EndFor
	\State $ \mathsf{y_{n+1}} = \mathsf{y_{n}} + \mathsf{u}$
	\State  {\bf Output:} $\mathsf{y_{n+1}}$\Comment{B-series coefficient of the next step numerical solution.}
	\end{algorithmic}
\end{algorithm}

%%%%%%%%%%%%%%%%%%%%%%%%%
\subsection{Order conditions for partitioned exponential methods}
%%%%%%%%%%%%%%%%%%%%%%%%%

Order conditions are constructed for each method by equating the coefficients of the B-series expansion $\mathsf{y}_{n+1}$ of the numerical solution $y_{n + 1}$ to thoset of the exact solution $y(t_n + h)$ up to the desired order of accuracy. 
\begin{theorem}[Order conditions]
	\label{thm:order_conditions}
A partitioned exponential method of W type has order $p$ only if it satisfies the conditions:
\[
	\mathsf{y}_{n+1}(\tau) = 1/\gamma_{\textsc{w}}(\tau)\quad \forall \tau \in \;\textnormal{TPS-trees with}\; \abs{\tau} \leq p.
\]
A partitioned exponential method of split-RHS type has order $p$ only if it satisfies the conditions:
\[
	\mathsf{y}_{n+1}(\tau) = 1/\gamma_{\textsc{s}}(\tau)\quad \forall \tau \in \;\textnormal{TPS-trees with}\; \abs{\tau} \leq p.
\]
\end{theorem}

The order conditions for a three-stage formulation of each type of method  discussed herein are given in
\iflong
\ref{sec:order-conditions-3stage}. 
\else
{\cite[Appendix E]{narayanamurthi2019}. }
\fi

\begin{remark}
%While the proof of correctness of the B-series operations for these methods is not discussed at length, the readers can consult 
A discussion of closely related B-series concepts can be found in \cite{Butcher_2011_exp-order,butcher2009,Rainwater_2014_semilinear,Tokman_2011_EPIRK}. 
\end{remark}

%%%%%%%%%%%%%%%%%%%%%%%%%%%%%%%%%%%%%%%

%%%%%%%%%%%%%%%%%%%%%%%%%%%%%%%%%%%%%%%
% !TEX root = Nonstiff_pexpw.tex
%%%%%%%%%%%%%%%%%%%%%%%%%%%%%%%%%%%%%%%
\section{Construction  of Third Order Schemes}
\label{sec:construction}
%%%%%%%%%%%%%%%%%%%%%%%%%%%%%%%%%%%%%%%

We build three-stage third-order PEXPW and PEPIRKW methods, whose formulations were discussed in Sections \ref{sec:PEXPW} and \ref{sec:PEPIRKW}, respectively. PEPIRKW methods with three stages of higher than third-order cannot be constructed, as the corresponding unpartitioned methods EPIRKW \cite[Sec. 3.2]{Sandu_2019_EPIRKW}, themselves cannot achieve it. 

Construction of methods from the two families begins with first formulating the algebraic equations corresponding to order conditions of Theorem \ref{thm:order_conditions}. This is done automatically using Mathematica and running the Algorithm \ref{alg:PEPIRKW_b_series} for PEPIRKW methods, and its exponential-Rosenbrock counterpart for PEXPW methods, to compute the B-series coefficients of the numerical solutions, and then equating them to the B-series coefficients of the exact solution. Next, we use Mathematica to solve the nonlinear equations for numerical values of the method coefficients.

Third-order coefficients for two PEXPW methods with second-order embedded schemes are given in \ref{sec:coefficients-PEXPW} as Method \ref{method:pexpw3a_coefficients} and \ref{method:pexpw3b_coefficients}, respectively. Although one can obtain third-order methods with three stages in each partition, we increased the stage count of the second partition to four as it was not possible to build an embedded method without increasing the degrees of freedom. Also notice that the coupling between the stages in the two methods happens via the $u_i^{\{j\}}$ stages.

The coefficients of two third order PEPIRKW methods are given in \ref{sec:coefficients-PEPIRKW} as Method \ref{method:pepirkw3a_coefficients} and Method \ref{method:pepirkw3b_coefficients}, respectively. Each of the methods has an embedded second order scheme for error estimation and step size control.

Finally, coefficients of a third order partitioned sEPIRK methods based on averaging are given in
\iflong
\ref{sec:coefficients-PSEPIRK} as Method \ref{method:psepirkb_coefficients}. 
\else
{\cite[Appendix D]{narayanamurthi2019}.} 
\fi
While of theoretical interest, methods of this type perform poorly in practice for stiff systems, and we will not pursue them further.

%%%%%%%%%%%%%%%%%%%%%%%%%%%%%%%%%%%%%%%

%%%%%%%%%%%%%%%%%%%%%%%%%%%%%%%%%%%%%%%
% !TEX root = Nonstiff_pexpw.tex
%%%%%%%%%%%%%%%%%%%%%%%%%%%%%%%%%%%%%%%
\section{Implementation Considerations}
\label{sec:implementation}
%%%%%%%%%%%%%%%%%%%%%%%%%%%%%%%%%%%%%%%

Computing the action of matrix-exponential like functions on vectors constitutes the bulk of the computational cost of exponential time integrators \cite{Tokman_2011_EPIRK,Rainwater_2014_semilinear}. In the case of large systems of ODEs that arise naturally from the semi-discretization of PDEs, Krylov-subspace based methods are the de-facto choice for efficiently computing these products \cite{Hochbruck_1997_exp,niesen2012,Sidje_1998}. To evaluate $\varphi(h\gamma A) b$ or $\Psi\bigl(h\gamma A\bigr) b$, an $M$-dimensional Krylov-subspace,  ${\mathcal{K}_M = \textnormal{span}\{b, A b, \dots, A^{M - 1} b\}}$, is built using Arnoldi/Lanczos iterations. The by-product of  Arnoldi/Lanczos is an orthonormal matrix, $V_M$, which spans the Krylov-subspace, $\mathcal{K}_M$, and an upper Hessenberg/tridiagonal matrix, $H_M$. Then, $\varphi_k(h\gamma A) b$ can be approximated as $\varphi_k(h\gamma A) b \approx \norm{b} V_M \varphi_{k}(h \gamma H_M) e_1$ \cite{Sidje_1998,Sandu_2014_expK,Sandu_2019_EPIRKW}.

The $\Psi$ matrix-vector products can be obtained by computing the individual $\varphi$ products and taking a linear combination of them. We, alternatively, follow the `Expokit' strategy \cite{Sidje_1998} of constructing an augmented matrix around $H_M$, and exponentiating it to get approximations of $\varphi_k(h \gamma H_M) e_1$ products for a range of $k$ values as columns of the resultant matrix. Multiplying the columns of this matrix by $V_M$, rescaling by $\norm{b}$, and taking linear combinations using $p_{j,k}$ coefficients as weights \eqref{eqn:psi_definition} yields an approximation for $\Psi\bigl(h \gamma A\bigr) b$. We briefly discuss a number of factors that we have taken into consideration in the computation of these matrix-exponential like functions.

\paragraph{Adaptivity in Arnoldi/Lanczos iterations} 

%\textcolor{red}{You check accuracy of $\varphi_1(\Lone) u$ or $\varphi_1(\Ltwo) v$? What vectors are these matrices applied to when you check residuals? Please clarify.}
%\mahesh{I have not implemented it the way you are thinking (specific to each partition). Everytime we want to compute $\Psi \cdot v$, I compute the subspace for $\varphi_1$ only. Higher order $\varphi$ functions converge faster than $\varphi_1$ and hopefully are computable with the same subspace as for $\varphi$ (as the weights in the denominator of the series expansion become progressively larger \eqref{eqn:phi-series-expansion}). I have clarified this information better below.}

We build the Krylov-subspace adaptively by bounding the error between the quantity $\varphi_1(h \gamma A) v$ and its approximation using the computed Krylov-subspace, $V_M \varphi_1(h \gamma H_M) e_1$, i.e., ${s_M = \varphi_1(h \gamma A) v - V_M \varphi_1(h \gamma H_M) e_1}$ (see \cite{saad1992}). The size of the Krylov-subspace, $M$, is chosen so that the error, $s_M$, is under a certain tolerance. Since higher order $\varphi_k$ functions must converge faster than $\varphi_1$, because the weights in the denominator of the series expansion of $\varphi_k$ become progressively larger, according to equation \eqref{eqn:phi-series-expansion}, with increase in order $k$, we use the same subspace as obtained from bounding the error $s_M$ to approximate $\varphi_k(h \gamma A) v$. When a $\Psi$ function product needs to be computed, we build the Krylov-subspace by bounding $s_M$ and use the resultant upper-Hessenberg/tridiagonal matrix, $H_M$, to construct an augmented matrix in line with Sidje's `Expokit' strategy \cite{Sidje_1998} and evaluate all the $\varphi_k$ products at once. 

The dimension $M$ of the Krylov-subspace is chosen by ensuring that the error, $s_M$, is below $10^{-12}$ for fixed timestep experiments, and the solution tolerance for adaptive timestep experiments. Since computing the error, $s_M$, involves a matrix-exponential operation, we only compute it at certain pre-determined indices (of which the first sixteen are $[1,\allowbreak 2,\allowbreak 3,\allowbreak 4,\allowbreak 6,\allowbreak 8,\allowbreak 11,\allowbreak 15,\allowbreak 20,\allowbreak 27,\allowbreak 36,\allowbreak 46,\allowbreak 57,\allowbreak 70,\allowbreak 85,\allowbreak 100]$) such that the cost of computing the error equals the cost of all previous computations in the same projection \cite[Section 6.4]{Hochbruck_1998_exp}. 

We can speedup the adaptive Arnoldi/Lanczos process by passing in the size of the Krylov-subspace from the previous timestep and using this information to reduce the total number of error computations per projection. We compute the errors only from the index that just precedes the subspace size passed in from the previous timestep allowing us to shrink the subspace size if needed while simultaneously keeping the total number of computations low. Initial experiments reveal that a $\sim 20\%$ cost savings in overall time-integration process is possible. We can extend this idea to subsequent projections between stages or within one stage of the same timestep.  

\paragraph{Minimizing the $g$ coefficients} 

The $g$ coefficients that scale the matrix argument of $\Psi$ in the EPIRK family of methods has to be chosen as small as possible for efficiency reasons \cite{Tokman_2011_EPIRK,Rainwater_2014_semilinear}. To this end, we optimized the undetermined $g$ coefficients after obtaining a family of solutions for the PEPIRKW3 methods (both A and B). We did not optimize all the coefficients of PEXPW3 methods as the order conditions were significantly harder to solve with $\gamma$ values not assumed. 

\paragraph{Computational optimizations for reaction-diffusion systems}

The computation of matrix-exponential-like functions can be further optimized for reaction-diffusion PDE systems. Let $\Jb$ be the full-Jacobian; $\Jb_D$, the Jacobian of the diffusion part; $\Jb_R$, the Jacobian of the reaction part; $\Jb_R^P$, the Jacobian of the reaction part with the variables permuted to give the reaction Jacobian a block diagonal structure. For a two species reaction-diffusion system, with the state variables ordered first by grid location and second by species, the structure of the Jacobians is pictured in Figure~\ref{fig:Jacobian_operators_structure}. The permuted Jacobian $\Jb_R^P$, where state variables are are ordered first by  species and second by locations, has a block diagonal structure.

\begin{figure*}[htb!]
	\centering
	\begin{subfigure}[t]{0.49\textwidth}
		\centering
		\frame{\includegraphics[scale=0.55]{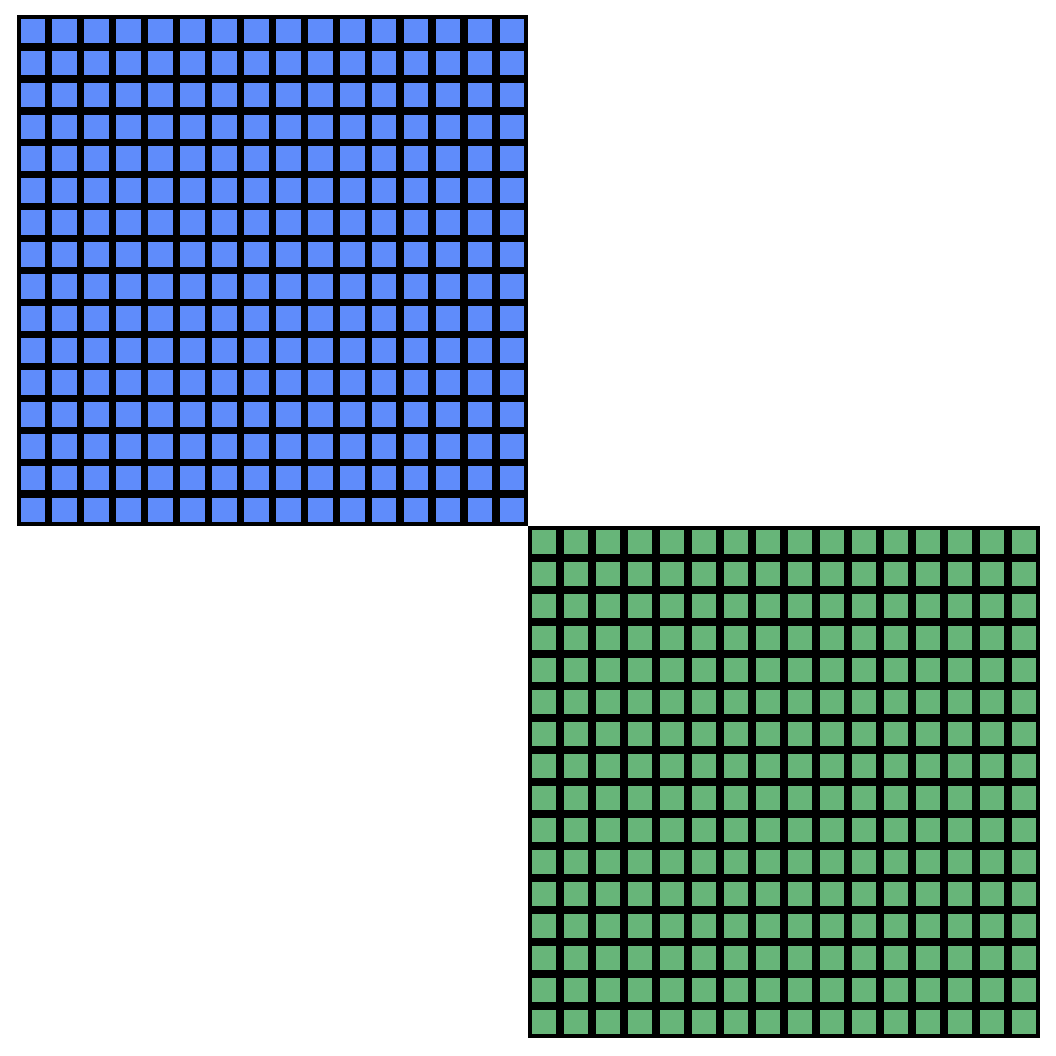}}
		\caption{Diffusion Jacobian ($\Jb_D$)}
	\end{subfigure}
	~
	\begin{subfigure}[t]{0.49\textwidth}
		\centering
		\frame{\includegraphics[scale=0.55]{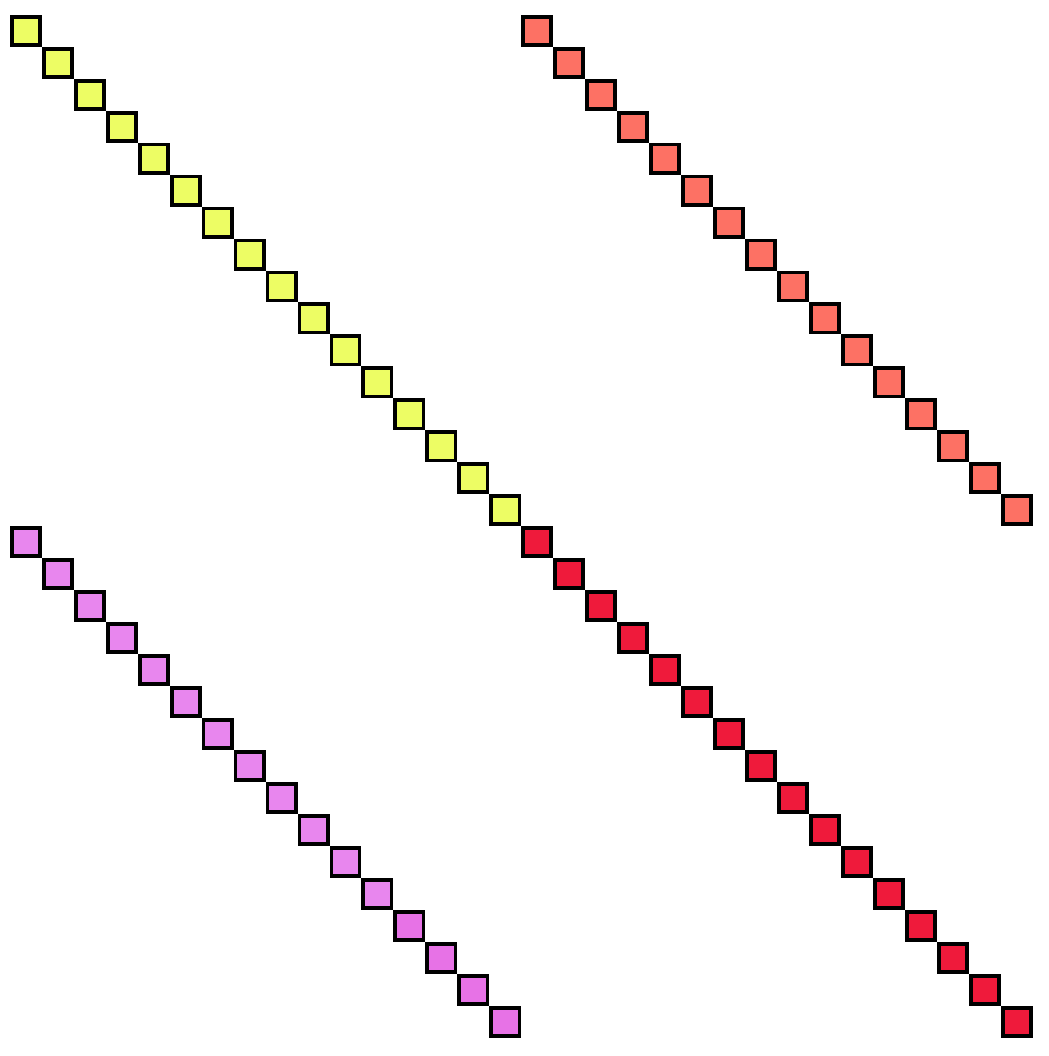}}
		\caption{Reaction Jacobian ($\Jb_R$)}
	\end{subfigure}
	\vskip\baselineskip
	\begin{subfigure}[t]{0.49\textwidth}
		\centering
		\frame{\includegraphics[scale=0.55]{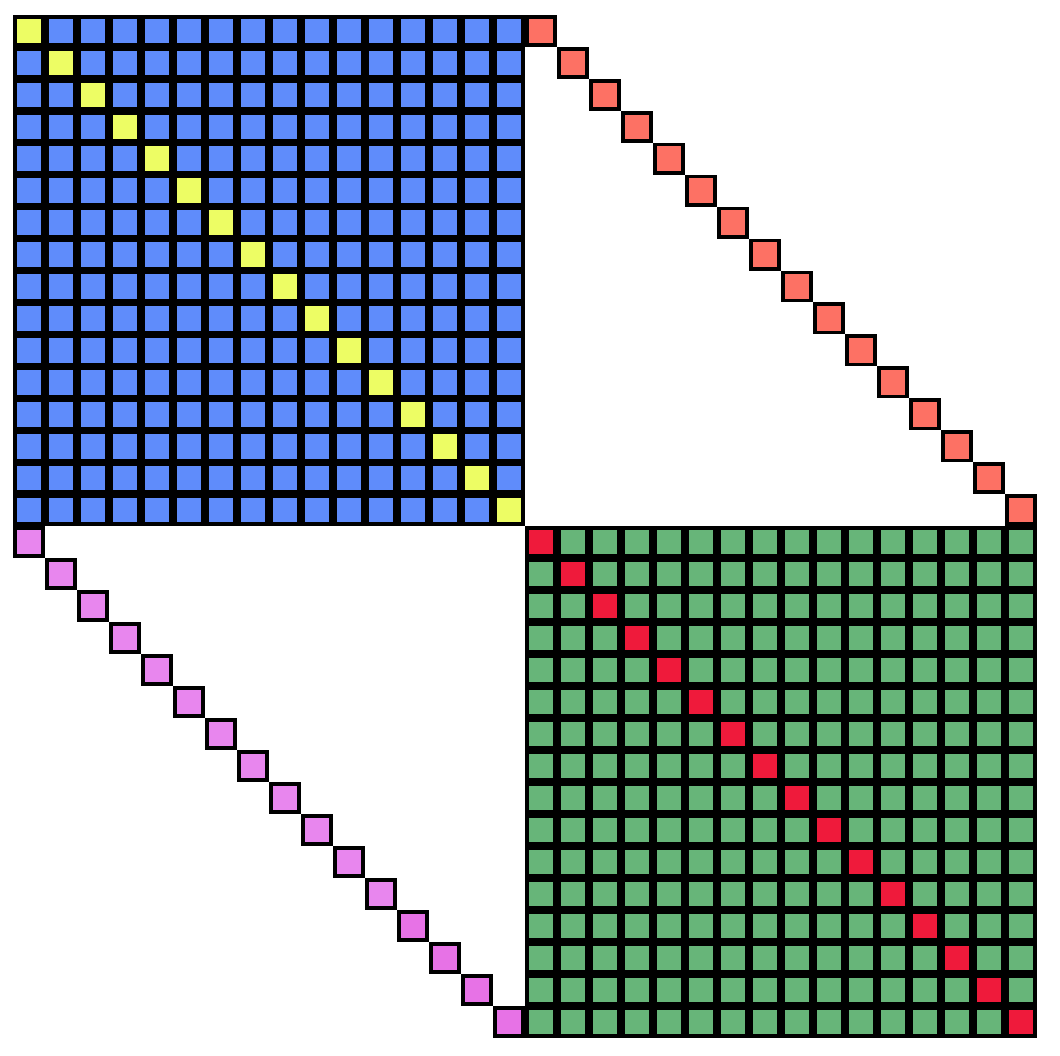}}
		\caption{Diffusion + Reaction Jacobian ($\Jb$)}
	\end{subfigure}
	~
	\begin{subfigure}[t]{0.49\textwidth}
		\centering
		\frame{\includegraphics[scale=0.55]{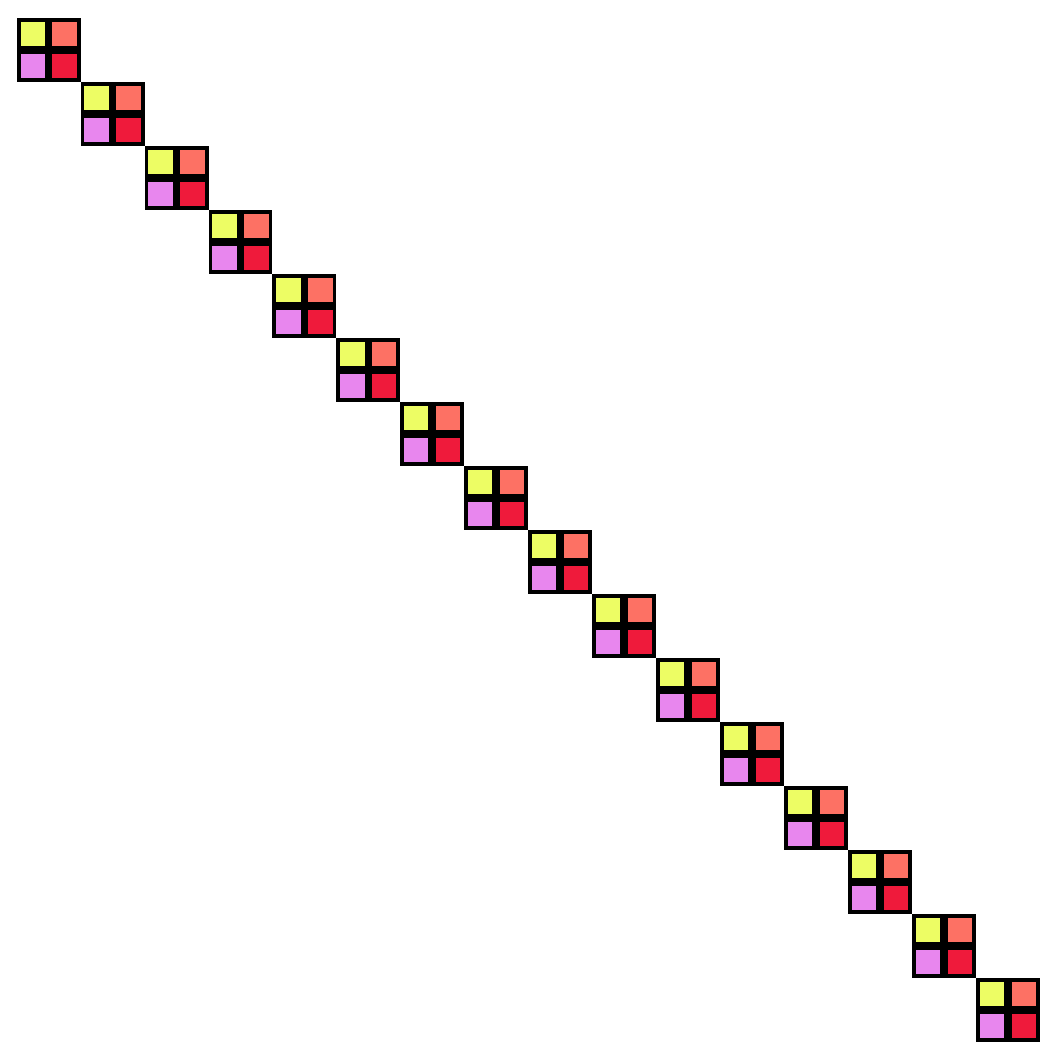}}
		\caption{Permuted Reaction Jacobian ($\Jb_R^P$)}
	\end{subfigure}
	\caption{Structure of the Jacobians of the different operators in a two-species reaction-diffusion system. Note that the each of the blocks of the diffusion Jacobian has (an almost) sparse diagonal structure with the number of diagonals dependent on the spatial discretization.\label{fig:Jacobian_operators_structure}}
\end{figure*}

We can then approximate the Jacobian of reaction-diffusion systems by choosing from any of the four components in Figure \ref{fig:Jacobian_operators_structure}. There are benefits to choosing approximations which have a block diagonal structure -- i) evaluating matrix functions on them will not result in fill-ins, ii) blocks can be separated into groups, with each group having one or more blocks, and the matrix functions can be computed on groups in parallel. These properties do not hols when the matrix functions are applied to the full Jacobian.

For partitioned-exponential methods, we partition the right-hand side into the diffusion part and all of the rest, which may include forcing terms along with the reaction part. We use the Jacobian of each part as the argument of $\varphi$ and $\Psi$ functions for split and W-methods in their respective partitions. Later, we demonstrate using numerical experiments that there are computational advantages to such a partitioning. 

\paragraph{Reordered Jacobians} Renumbering the variables via a permutation matrix $P$ changes the Jacobian (here, the reaction Jacobian) as follows:
\begin{equation}
	\label{eqn:relationship_between_permuted_reaction_jacobian_and_unpermuted}
	\Jb_R^P = P^{T} \,\Jb_R\, P.
\end{equation} 
%
%One can parallelize the computations of $\varphi$ on $J_R^P$ by grouping blocks together. 

In many applications it is computationally convenient to evaluate $\varphi$ and $\Psi$ functions on a component Jacobian using one ordering of variables, and to evaluate them on another component Jacobian using a different ordering. For example, for reaction diffusion systems, the matrix functions work efficiently on the permuted reaction Jacobian $\Jb_R^P$, but on the diffusion Jacobian $\Jb_D$ with the original ordering. The following lemma provides an elegant way to accommodate different permutations.

\begin{lemma}
\label{lemma:varphi_k_on_reaction_jacobian}
The result of applying the matrix functions $\varphi_k$ to the original Jacobian can be obtained by applying the function to the permuted Jacobian, and permuting back the rows and columns of the result:
\[
\varphi_k(h \gamma \Jb_R) = P \varphi_k(h \gamma \Jb_R^P) P^T, \qquad \forall k \geq 0.
\]
\end{lemma}
\begin{proof}
We know that $P$ is a permutation matrix. Therefore, $P P^T = P^T P = I$. We also know that the $\varphi_k$ function has a power series expansion: 
\[\varphi_k(z) = \sum_{l = 0}^{\infty} \frac{z^l}{(k + l)!}. \]
Since $\Jb_R^P = P^T \Jb_R P$, pre- and post-multiplying by $P$ and $P^T$ respectively, we get $\Jb_R = P \Jb_R^P P^T$. Plugging in the expression for $\Jb_R$ in the power series expansion of $\varphi_k(h \gamma \Jb_R)$ and using the fact that $P P^T = P^T P = I$, we get the desired result.
\end{proof}

\begin{corollary}
\label{corollary:Psi_on_reaction_jacobian}
The same permute, evaluate, and permute back sequence of operations holds for the $\Psi$ functions:
$\Psi\bigl(h \gamma \Jb_R\bigr) = P \Psi\bigl(h \gamma \Jb_R^P\bigr) P^T$
\end{corollary}
\begin{proof}
Result follows from Lemma \ref{lemma:varphi_k_on_reaction_jacobian} and noting that $\Psi$ functions are linear combinations of $\varphi_k$ functions.
\end{proof}

To avoid repeated multiplications by permutation matrices and their transposes, one constructs the permuted reaction Jacobian $\Jb_R^P$ directly. To compute the action of $\varphi_k(h \gamma \Jb_R)$ or $\Psi\bigl(h \gamma \Jb_R\bigr)$ on a vector $v$, one only permutes vectors, which is significantly more efficient than permuting both the rows and columns of matrices:
\begin{equation}
	\label{eqn:arrangement_of_multiplications_by_permutation_matrices}
	\varphi_k(h \gamma \Jb_R) v =  \underbrace{P \cdot \bigg(\varphi_k(h \gamma \underbrace{\Jb_R^P}_{\arrd{\textnormal{Computed directly}}}) \cdot \underbrace{\left(P^T v\right)}_{\arru{\textnormal{Permute vector}}}\bigg)}_{\textnormal{Permute back to the desired ordering of variables.}}
\end{equation}

\paragraph{Benefits of using block diagonal Jacobians} Recall that we use Krylov-subspaces to compute matrix-exponential like functions. The benefit of using block diagonal structures is the savings resulting from computing subspaces of smaller blocks compared to the full Jacobian matrix. For instance, computing an $M$ dimensional subspace for a matrix of size $N$ requires $\sim M N^2 +  M^2 N +  M N$ operations. If the matrix is block diagonal with two blocks, = computing an $M$-dimensional subspace on two blocks of size $N/2$ using the Arnoldi process requires  $\sim M \cdot \frac{N^2}{2} +  M^2 N +  M  N$ operations. When $M \ll N$ this is about half the computational cost of using the unstructured matrix. Moreover, if one uses an adaptive Arnoldi process like we do in this work, additional savings are obtained when one block converges to the desired tolerance faster than the other (meaning that fewer vectors are needed in the subspace for one block when compared to the other block). Computing matrix-exponential like functions on individual blocks can be performed in parallel.

In our numerical tests, we experimented with parallelism of block based matrix-exponential operations on the diffusion matrix alone. We saw noticeable improvements in performance as discussed later in the numerical results. Serial implementation of matrix-exponential operations on the rearranged reaction Jacobian turned out to be slightly more expensive than for the original reaction Jacobian due to the two additional vector permutations (data movement is expensive).  We did not experiment with parallelism on the rearranged reaction Jacobian, but  significant benefits are possible for large reaction-diffusion systems with many chemical species.

%%%%%%%%%%%%%%%%%%%%%%%%%%%%%%%%%%%%%%%

%%%%%%%%%%%%%%%%%%%%%%%%%%%%%%%%%%%%%%%
% !TEX root = Nonstiff_pexpw.tex
%%%%%%%%%%%%%%%%%%%%%%%%%%%%%%%%%%%%%%%
\section{Numerical Experiments}
\label{sec:numerics}
%%%%%%%%%%%%%%%%%%%%%%%%%%%%%%%%%%%%%%%

%%%%%%%%%%%%%%%%%%%%%%%%%%%%%%%%%%%%%%%
\subsection{Test problems}
%%%%%%%%%%%%%%%%%%%%%%%%%%%%%%%%%%%%%%%

\paragraph{Lorenz-96} 
The dynamical system described by the system of equations,
\begin{equation}
	\label{eqn:Lorenz}
	\frac{dy_j}{dt} = -y_{j-1}(y_{j-2}-y_{j+1})-y_j+F,\quad j = 1\dots N, \quad
	y_{0} = y_{N},
\end{equation}
was introduced by Edward Lorenz in \cite{lorenz1996}. It describes the evolution of an arbitrary component of the atmosphere at $N$ equally spaced points along a latitudinal circle. In our fixed-step experiment, we take $N = 40$ and $f = 8$. The initial condition, $y_0$, is obtained by integrating a random initialization of the state vector from $0$ to $0.3$ time units. 

\paragraph{1-D Semi-linear Parabolic Problem} The semilinear parabolic PDE \cite[Example 6.1]{luan2014a}
\begin{equation}
	\label{eqn:Semilinear}
	\pfrac{y}{t} = \pfrac{^2 y}{x^2} + \frac{1}{1 + y^2} + \Phi(x, t), \qquad (x, t) \in [0, 1] \times [0, 1],
\end{equation}
is solved numerically for $y(x,t)$ assuming homogenous Dirichlet boundary conditions. $\Phi(x, t)$ is chosen such that the true solution of equation \eqref{eqn:Semilinear} is $y(x, t) = x(1-x) e^t$. The PDE is discretized in space using second-order finite differences over 500 equidistant grid points. 

\paragraph{Allen-Cahn}

The Allen-Cahn reaction-diffusion PDE \cite{allen1979}
\begin{equation}
	\label{eqn:Allen_Cahn}
	\pfrac{u}{t} = \alpha \cdot \Delta u + \gamma \cdot (u - u^3),\qquad (x,y) \in [0, 1]^2,\quad t \in [0,0.5],
\end{equation}
describes the process of phase separation. We numerically solve a two dimensional version of the PDE semi-discretized using standard second-order finite differences with selected diffusion coefficient  $\alpha$, and reaction coefficient $\gamma$. The boundary conditions were assumed to be periodic. Initial conditions were set to $u_0(x, y, 0) = 0.4 + 0.1\,(x + y) + 0.1\,\sin(10x) \, \sin(20y)$. The ODE system obtained after semi-discretization is stiff and integrated in time using our methods. 

\paragraph{Reversible Gray-Scott}

The following system of PDEs \cite{mahara2004}:
\begin{equation}
	\begin{split}
	\pfrac{U}{t} &= D_U \nabla^2 U  - k_1 U V^2 + f(1 - U) + k_{-1} V^3 ,\\
	\pfrac{V}{t} &= D_V \nabla^2 V  + k_1 U V^2 - (f + k_2)V - k_{-1} V^3  + k_{-2} P,\\
	\pfrac{P}{t} &= D_P \nabla^2 P  - k_{-2} P - f P,
	\end{split}
	\label{eqn:Reversible_Gray_Scott}
\end{equation}
extends the traditional Gray-Scott model \cite{pearson1993} by making the reactions reversible. $U$, $V$ and $P$ are the concentrations of the reactants, and $D_U$, $D_V$ and $D_P$ are the corresponding diffusion coefficients. $f$ is the rate of flow, $k_1$ and $k_2$ are the rate constants of the forward reactions and $k_{-1}$ and $k_{-2}$ the rate constants of the backward reactions. We discretize the PDE system in space using second-order finite differences over a grid of size $100 \times 100$. The parameter settings for the experiment are $D_U = 2$, $D_V = 1$, $D_P = 0.1$, $k_1 = 1$, $k_2 = 0.055$, $k_{-1} = 0.001$, $k_{-2} = 0.001$, $f = 0.028$. The initial conditions are $(U, V, P) = (1, 0, 0)$ with a small square region perturbed as $(U, V, P) = (0.5, 0.25, 0) + 10\%$ random noise \cite{mahara2005}. The boundary conditions are assumed to be periodic and the interval of integration is $[0, 5]$.

%%%%%%%%%%%%%%%%%%%%%%%%%%%%%%%%%%%%%%%
\subsection{Tested schemes}
%%%%%%%%%%%%%%%%%%%%%%%%%%%%%%%%%%%%%%%

We perform fixed and adaptive time-stepping experiments with the methods derived in this paper. We selected a set of third-order exponential methods, similar to those derived herein, from the literature to provide performance benchmarks:
\begin{itemize}
	\item[a.] sEPIRK3, a third-order split EPIRK method introduced and studied in \cite{Rainwater_2014_semilinear}; 
	\item[b.] EPIRK3s3, a third-order non-split EPIRK method which shares the coefficients with sEPIRK3 \cite{Rainwater_2014_semilinear}; 
	\item[c.] EPIRKW3 and EPIRKW3c, two third-order W-methods of EPIRK type that the authors introduced in \cite{Sandu_2019_EPIRKW}, with $\Wb = \Jb$; 
	\item[d.] EPIRKW3-D, same as EPIRKW3 but $\Wb = \Jb_D$; 
	\item[e.] EPIRKW3-R, same as EPIRKW3 but $\Wb = \Jb_R$;
\end{itemize}
Since the original publication does not provide embedded coefficients for sEPIRK (and EPIRK3s3) methods \cite{Rainwater_2014_semilinear}, we use sEPIRK (and EPIRK3s3) primarily in fixed timestep experiments.
The partitioned sEPIRK methods based on averaging (PSEPIRK) perform poorly in practice for stiff systems, and are not included in the numerical tests.

The convergence orders for select methods are summarized in Table \ref{table:order_of_convergence}. Table \ref{table:splittings_used_in_partitioned_methods} details the partitioning used for the right-hand side function of each problem, and their corresponding  Jacobians.

%%%%%%%%%%%%%%%%%%%%%%%%%%%%%%%%%%%%%%%
\subsection{Fixed timestep experiments}
%%%%%%%%%%%%%%%%%%%%%%%%%%%%%%%%%%%%%%%

For fixed timestep experiments, we ran the integrators on a set of test problems for a range of fixed step sizes as shown below for each problem:
	\begin{itemize}
		\item Lorenz-96, $h \in \{\num{0.06}, \num{0.03}, \hdots ,\num{9.375e-4}\}$,
		\item 1-D Semi-linear Parabolic Problem, $h \in \{\num{1.6e-2}, \num{0.8e-2}, \hdots ,\num{1.5625e-05}\}$,
		\item Allen-Cahn,  $150 \times 150$ grid, $h \in \{\num{1.28e-1}, \num{0.64e-02}, \hdots ,\num{1.5625e-05}\}$.
	\end{itemize}
The problem setup and the findings for each problem are discussed below.

\paragraph{Lorenz-96 results} 
The ODE system \eqref{eqn:Lorenz} is partitioned for the new methods into $\fone(y) = A_{N \times N}~y$, where $A_{N \times N}$ is a random matrix, and $\ftwo(y) = F(y) - A_{N \times N}~y$, where $F(y)$ is the right-hand side of the original ODE system. Component Jacobians are approximated by the diagonal of the Jacobian of the corresponding part. We  integrate equation \eqref{eqn:Lorenz} with each method by taking fixed timesteps over the time-span $[0, 0.3]$ time units and compute the relative error against a reference solution obtained by integrating the ODE for the same settings using \texttt{ode45}, with  absolute and relative tolerance set to $10^{-12}$. Results are shown in Figure \ref{fig:experiment-fixed-timestep-Lorenz-stepsvserror}. All integrators show their theoretical orders of convergence (see Table \ref{table:order_of_convergence}).

\begin{table}[htb!]
\centering
	\begin{tabular}{ | p{4cm} | c| c| c| c| c| }
		\hline
		\multicolumn{6}{|c|}{\textit{Fixed Timestep Experiments}} \\ \hline\hline
		\textbf{Experiment} & \textbf{pexpw3A} & \textbf{pexpw3B} & \textbf{epirkw3s3} & \textbf{sepirk3} & \textbf{expirkw3}\\
		\hline
		\textbf{Lorenz-96} & 2.99 & 2.99 & 3.03 & 2.99 & 3.00 \\ \hline
		\textbf{Semilinear Parabolic} & 1.31 & 1.32 & 2.69 & 3.33 &  2.89\\ \hline
		\textbf{Allen-Cahn,  150x150 Grid} & 1.77 & 1.88 & -- & 5.45 & 3.05 \\ \hline\hline
		\multicolumn{6}{|c|}{\textit{Adaptive Timestep Experiments}} \\ \hline
		\textbf{Experiment} & \multicolumn{2}{|c|}{\textbf{pexpw3A}} & \textbf{pexpw3B} & \textbf{expirkw3} & \textbf{expirkw3-D}\\ \hline
		\textbf{Allen-Cahn,  300x300 Grid (I)} & \multicolumn{2}{|c|}{2.05} &  1.60 &  2.43 &  - \\ \hline
		\textbf{Allen-Cahn,  300x300 Grid (II)} & \multicolumn{2}{|c|}{4.01} &  4.79 &  2.08 &  3.33 \\ \hline\hline
		\textbf{Experiment} & \multicolumn{2}{|c|}{\textbf{pexpw3A}} & \textbf{pexpw3A-RD} & \textbf{expirkw3} & \textbf{expirkw3c} \\ \hline
		\textbf{Reversible Gray-Scott, 100x100 Grid} & \multicolumn{2}{|c|}{2.99} & 3.34 & 1.80  & 3.01 \\ \hline
	\end{tabular}
	\caption{Experimental convergence orders of select methods computed using (at least) four consecutive points lying along a straight line. Partitioned EPIRK methods only gave a solution to the Lorenz-96 problem and they converged as follows: pepirkw3A (3.00), pepirkw3B (2.98).}
	 \label{table:order_of_convergence}
\end{table}

\begin{figure}[thb!]
	\centering
	\includegraphics[scale=0.35]{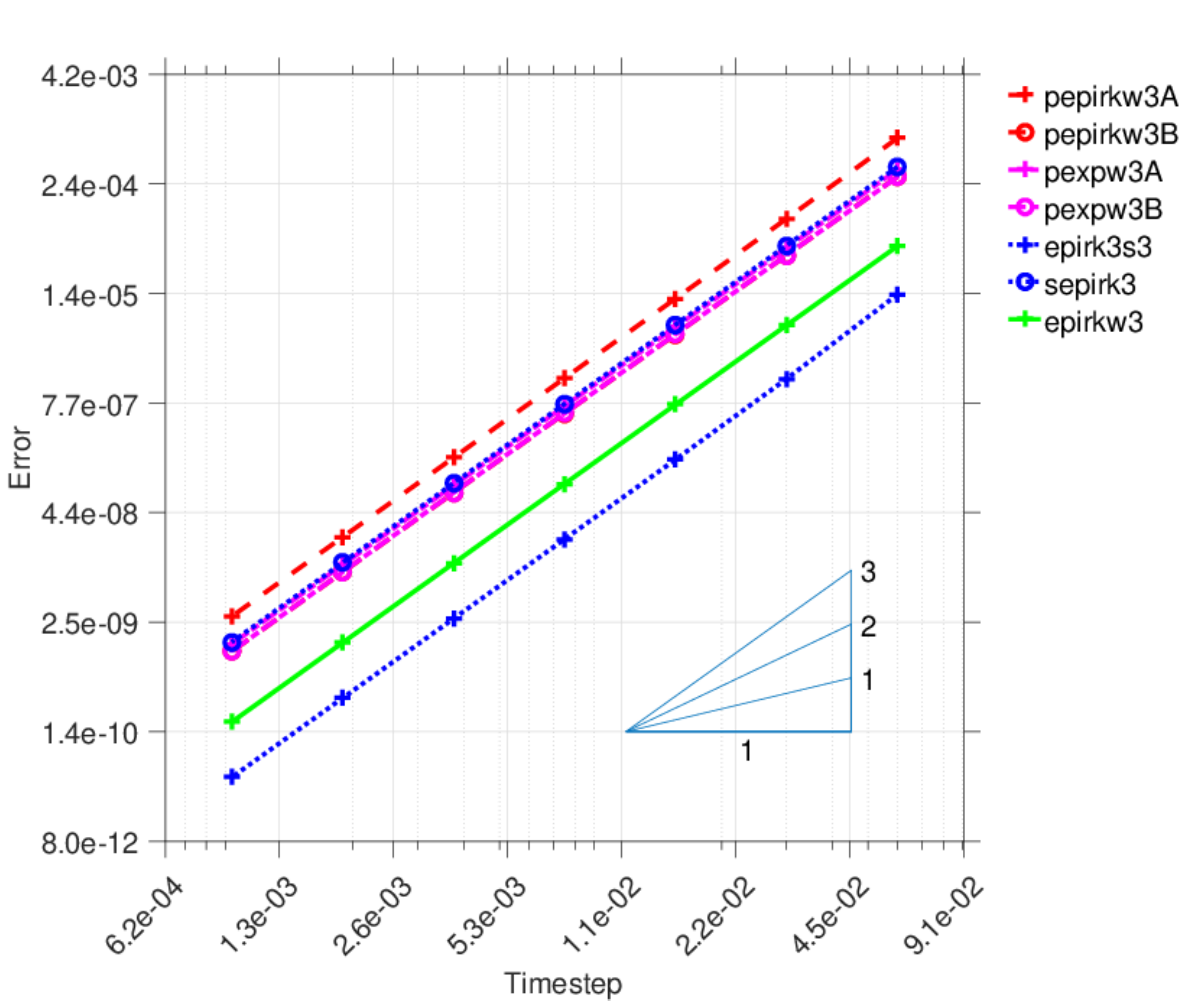}
	\caption{Fixed timestep experiment on the Lorenz-96 system 	\eqref{eqn:Lorenz}. All methods show full convergence order.\label{fig:experiment-fixed-timestep-Lorenz-stepsvserror}}
\end{figure}

\paragraph{1-D Semi-linear Parabolic Problem results} 

We integrate the resulting stiff ODE system \eqref{eqn:Semilinear} with each method by taking fixed timesteps and plot the relative error computed against a reference solution obtained using \texttt{ode15s} with absolute and relative tolerance set to $10^{-12}$. The convergence plot is shown in Figure \ref{fig:experiment-fixed-timestep-Semilinear-stepsvserror}. PEPIRKW methods fail to solve the problem. PEXPW methods show order reduction (Table \ref{table:order_of_convergence}). Unpartitioned, non-split methods which include EPIRKW3, EPIRK3s3 and EPIRKW3-D converge to highly-accurate solutions for small timesteps.
\begin{figure}[thb!]
	\centering
	\includegraphics[scale=0.6]{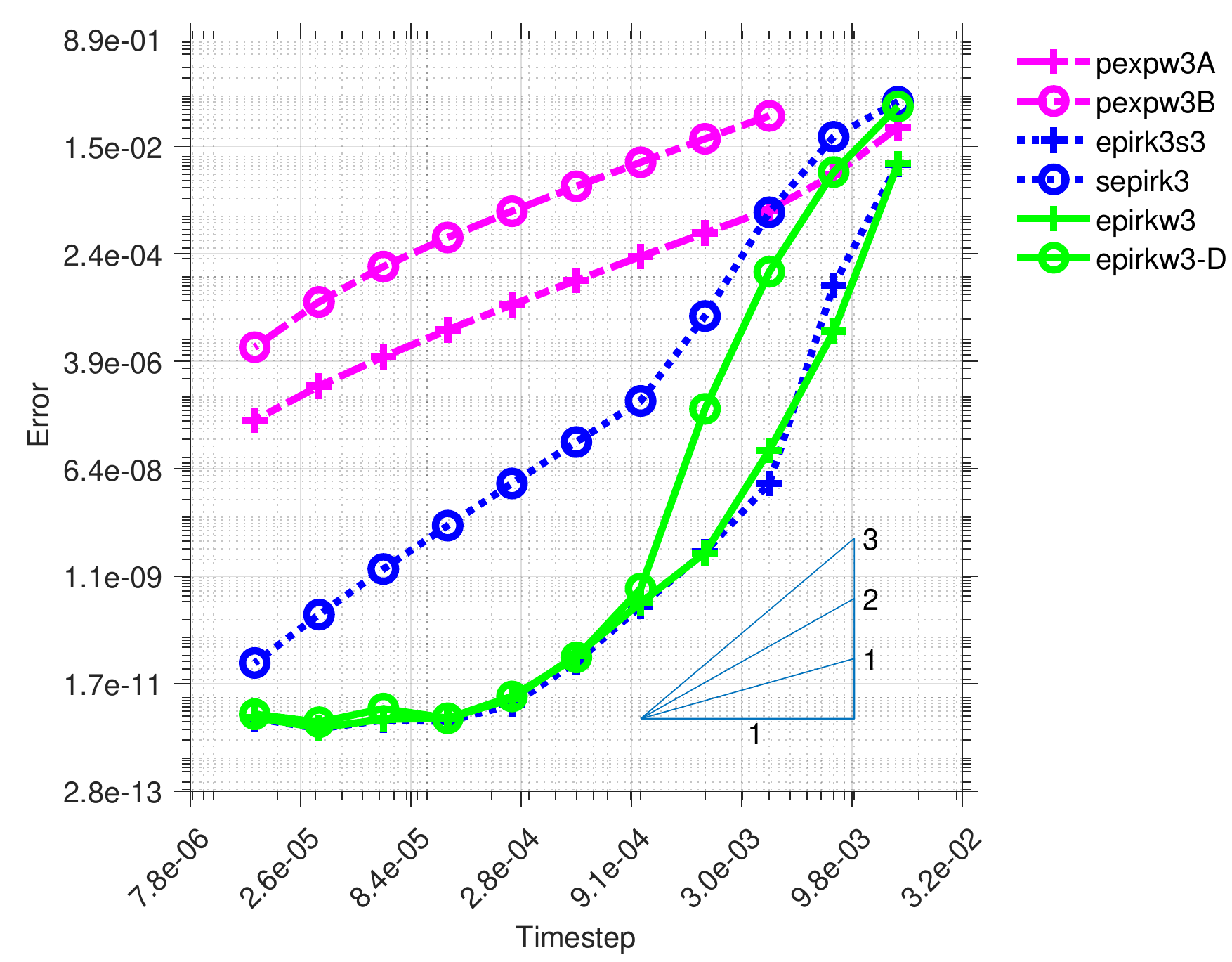}
	\caption{Fixed timestep experiment on 1-D Semi-linear Parabolic Problem 	\eqref{eqn:Semilinear}. PEPIRKW methods did not produce a solution, PEXPW methods show order $\approx 1.3$. Convergence order of other methods are shown in Table \ref{table:order_of_convergence}.\label{fig:experiment-fixed-timestep-Semilinear-stepsvserror}}
\end{figure}

\paragraph{Allen-Cahn results}

We numerically solve a two dimensional version of the PDE \eqref{eqn:Allen_Cahn}, semi-discretized using standard second-order finite differences on a $150 \times 150$ grid, with the diffusion coefficient  $\alpha = 1$, and the reaction coefficient $\gamma = 10$. The ODE system obtained after semi-discretization is integrated in time from $t = 0$ to $t = 0.5$ time units with fixed timesteps. Figure \ref{fig:experiment-fixed-timestep-AllenCahn-stepsvserror} shows the work-precision and convergence diagrams where `Error' is computed as the relative error in 2-norm against a reference solution obtained using \texttt{ode15s} with absolute and relative tolerance set to $10^{-12}$. 

The empirical convergence orders for some select methods is given in Table \ref{table:order_of_convergence}. In the asymptotic region,  partitioned-exponential methods have $\textnormal{order} \approx 1.7$. EPIRKW3 methods show $\textnormal{order} = 3$ for timesteps smaller than 1e-3. EPIRK3s3 converges to a highly-accurate solution for smaller timesteps than other methods, but its cost per timestep is higher than partitioned-exponential methods. Overall, partitioned-exponential methods obtain solutions with a higher accuracy for a given CPU time than both split and non-split EPIRK and EPIRKW methods.

\begin{figure}[thb!]
	\centering
	\begin{subfigure}[t]{\textwidth}
		\centering
		\includegraphics[scale=0.6]{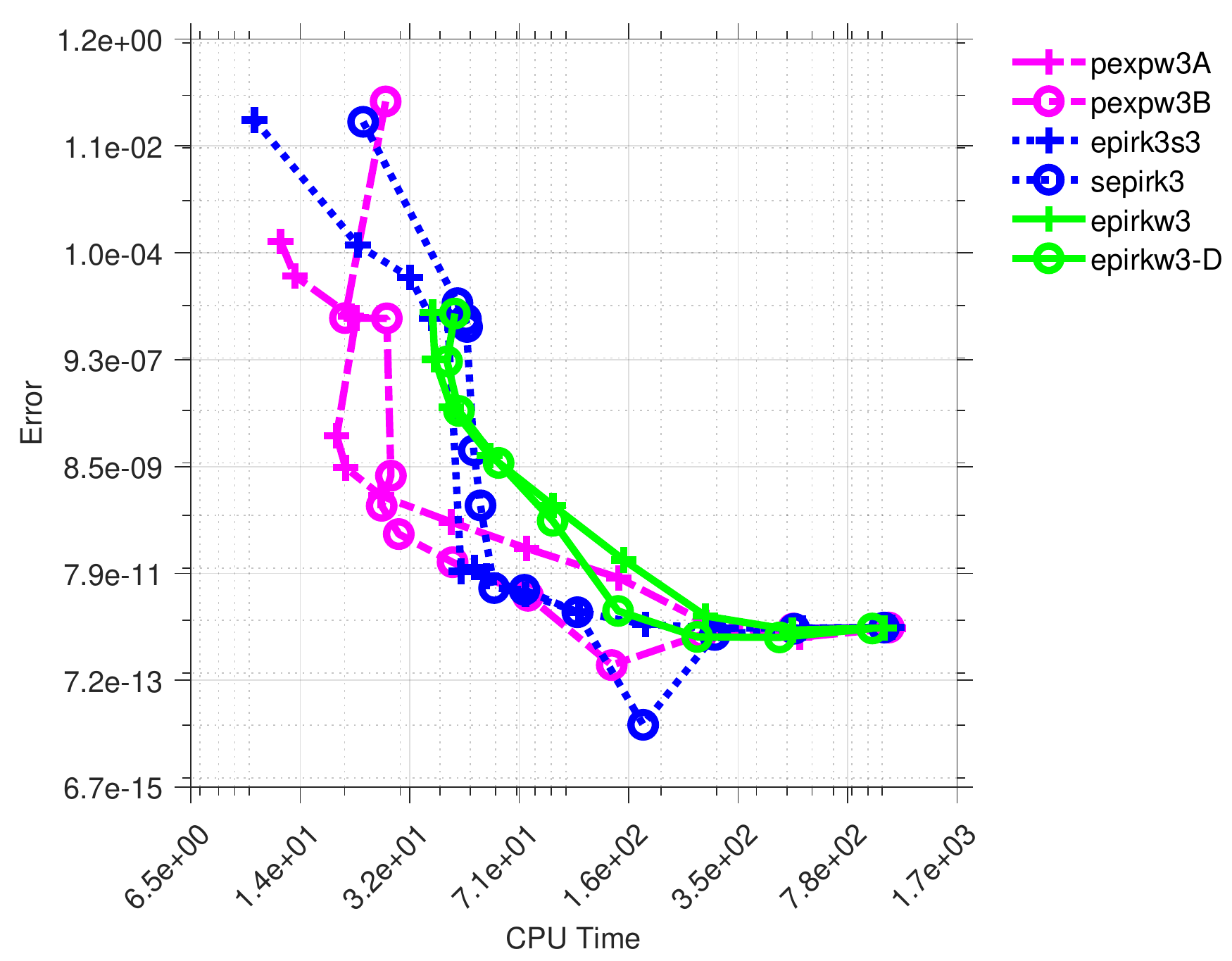}
		\caption{Work-precision diagram.}
	\end{subfigure}
	\vskip\baselineskip
	\begin{subfigure}[t]{\textwidth}
		\centering
		\includegraphics[scale=0.6]{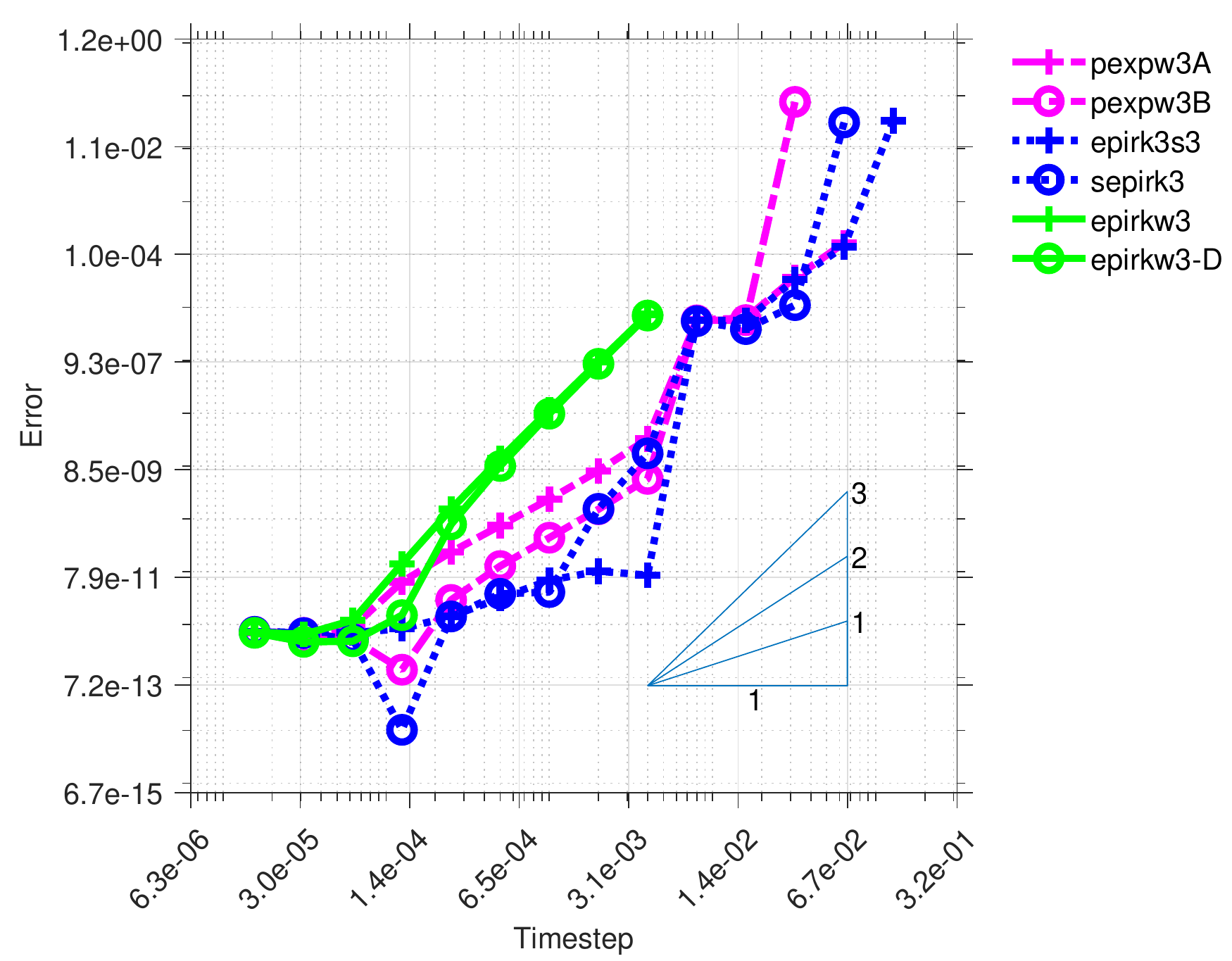}
		\caption{Convergence diagram.}
	\end{subfigure}
	\caption{Fixed timestep experiment on 2-D Allen-Cahn Problem \eqref{eqn:Allen_Cahn} on $150 \times 150$ grid. PEXPW methods show order $\approx 1.7$, PEPIRKW methods did not produce a solution, EPIRKW3 shows order $\approx 3$.\label{fig:experiment-fixed-timestep-AllenCahn-stepsvserror}}
\end{figure}

%%%%%%%%%%%%%%%%%%%%%%%%%%%%%%%%%%%%%%%
\subsection{Adaptive timestep experiments}
%%%%%%%%%%%%%%%%%%%%%%%%%%%%%%%%%%%%%%%

We perform adaptive timestep experiments for a range of solution tolerances between ${10^{-1}}$ to ${10^{-10}}$, using the \texttt{MATLODE} \cite{augustine2018} framework. Reference solutions are computed using \texttt{ode15s} with absolute and relative tolerance set to $10^{-12}$. The error controller is the same across all methods and is based on the discussion in \cite[Section II.4]{Hairer_book_I}. 

\paragraph{Allen-Cahn results on a $300 \times 300$ grid}

This time we discretize the Allen-Cahn system in space using second-order finite differences along $300$ grid points in each dimension. Each experiment deals with progressively stiffer reaction terms. The stiffness of the reaction term of Allen-Cahn increases with the value of $\gamma$. The results reported in Figures \ref{fig:experiment-adaptive-timestep-AllenCahn-GammaEquals10}, \ref{fig:experiment-adaptive-timestep-AllenCahn-GammaEquals100} and \ref{fig:experiment-adaptive-timestep-AllenCahn-GammaEquals1000} correspond to the  following parameter settings:
\begin{itemize}
	\item[I.] Figure \ref{fig:experiment-adaptive-timestep-AllenCahn-GammaEquals10}: $\alpha = 1, \gamma = 10$,
	\item[II.] Figure \ref{fig:experiment-adaptive-timestep-AllenCahn-GammaEquals100}: $\alpha = 1, \gamma = 100$, and
	\item[III.] Figure \ref{fig:experiment-adaptive-timestep-AllenCahn-GammaEquals1000}: $\alpha = 1, \gamma = 1000$.
\end{itemize}
The numerical results lead to the following conclusions.
\begin{itemize}
\item {Partitioned-exponential methods can be more stable than unpartitoned methods for some stiffness regimes.} For parameter settings I and II, the partitioned-exponential method, PEXPW3A, needs fewer timesteps than the unpartitioned EPIRKW3 method, indicating that it is more stable than the unpartitoned method. However, as the stiffness of the reaction term is increased further (parameter setting III), partitioned-exponential methods are no longer able to solve the problem (see Figure \ref{fig:experiment-adaptive-timestep-AllenCahn-GammaEquals1000}).
% \subparagraph{Stronger coupling between the partitions in partitioned-exponential methods may lead to more stable methods.} The stages of PEXPW3A are more strongly coupled than the stages of PEXPW3B. Judging from Figures \ref{fig:experiment-adaptive-timestep-AllenCahn-GammaEquals10}, \ref{fig:experiment-adaptive-timestep-AllenCahn-GammaEquals100} and \ref{fig:experiment-adaptive-timestep-AllenCahn-GammaEquals1000}, it appears that strongly coupled methods may exhibit more stability than loosely coupled ones.  

\item {Unpartitioned methods may need the full Jacobian or a very close approximation of it when both partitions are stiff, making them not very amenable to computational optimizations.} It is clear from Figures \ref{fig:experiment-adaptive-timestep-AllenCahn-GammaEquals10}, \ref{fig:experiment-adaptive-timestep-AllenCahn-GammaEquals100} and \ref{fig:experiment-adaptive-timestep-AllenCahn-GammaEquals1000} that as the stiffness of the reaction term is increased, the unpartitioned method EPIRKW3-D gets more expensive and less stable in comparison to the unpartitioned method EPIRKW3. While this is expected, it indicates that unpartitioned methods may require the full Jacobian to perform well. Using the full Jacobian makes the unpartitioned methods unable to take advantage of the computational optimizations  discussed in Section \ref{sec:implementation}.
\end{itemize}

\begin{figure*}[htb!]
	\centering
	\begin{subfigure}[t]{\textwidth}
		\centering
		\includegraphics[scale=0.6]{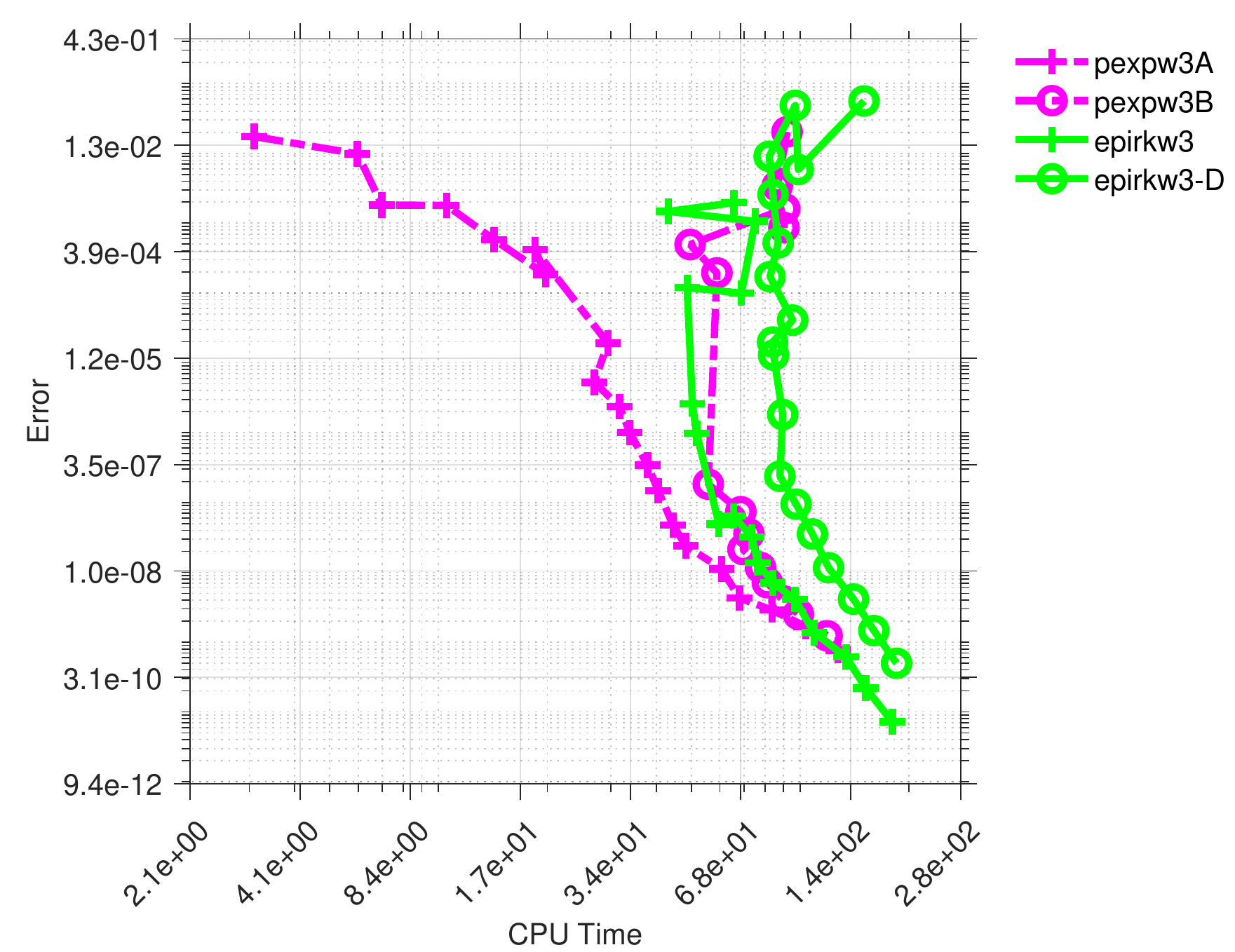}
		\caption{Work precision diagram.}
	\end{subfigure}
	\vskip\baselineskip
	\begin{subfigure}[t]{\textwidth}
		\centering
		\includegraphics[scale=0.6]{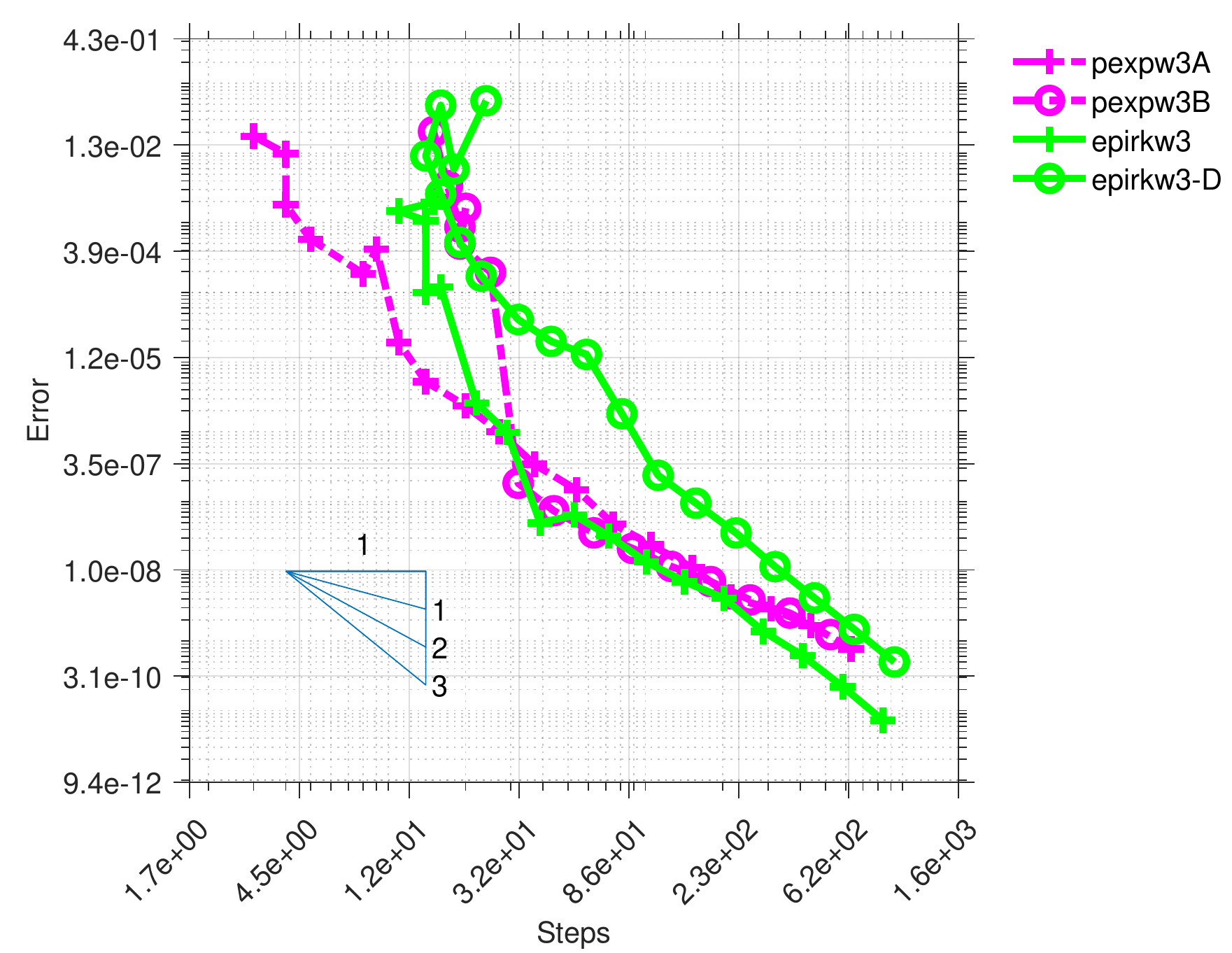}
		\caption{Convergence diagram.}
	\end{subfigure}
	\caption{Adaptive timestep experiments using Allen-Cahn \eqref{eqn:Allen_Cahn},  $300 \times 300$ grid (I). $\alpha = 1$, $\gamma = 10$ \label{fig:experiment-adaptive-timestep-AllenCahn-GammaEquals10}}
\end{figure*}

\begin{figure*}[htb!]
	\centering
	\begin{subfigure}[t]{\textwidth}
		\centering
		\includegraphics[scale=0.6]{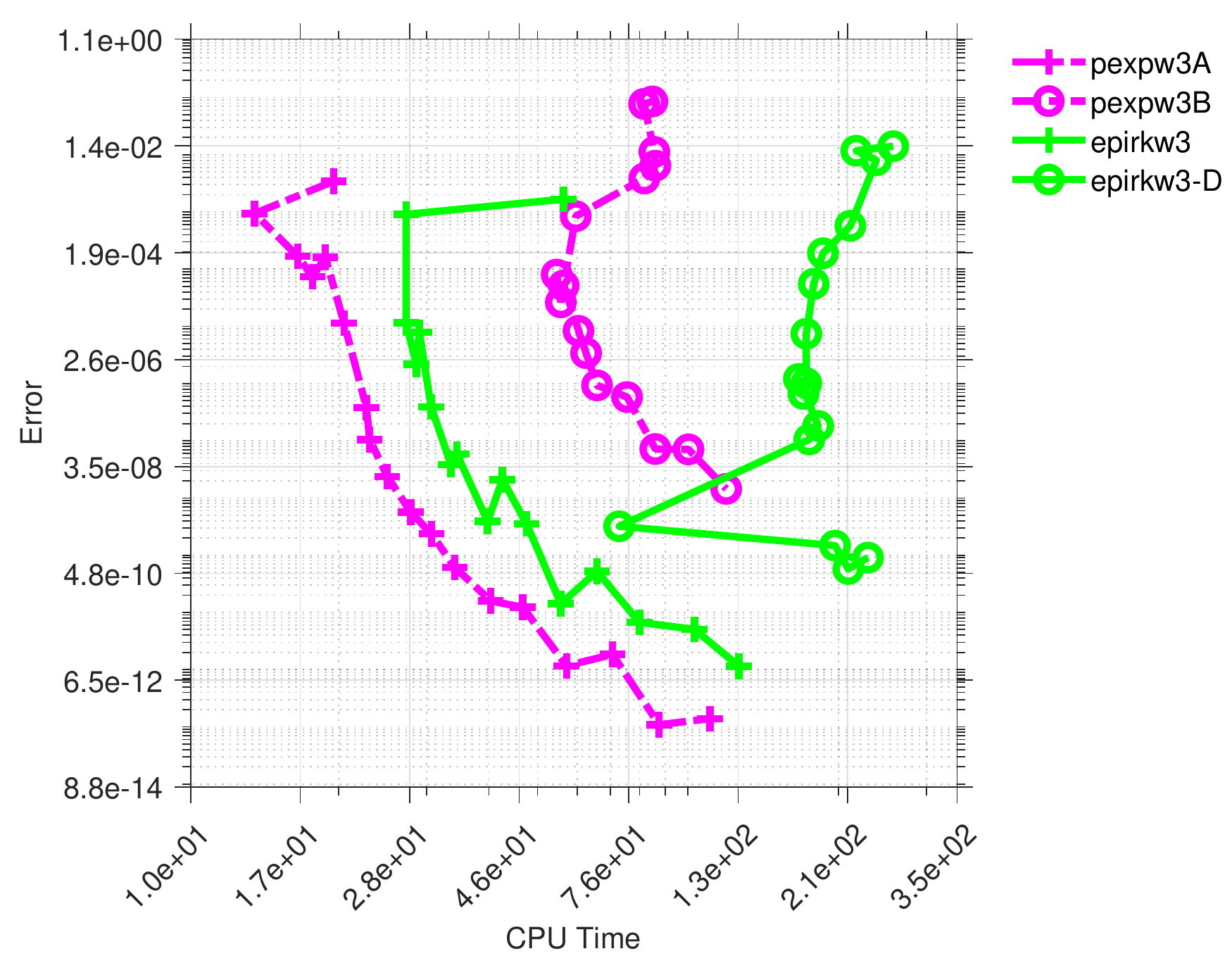}
		\caption{Work precision diagram.}
	\end{subfigure}
	\vskip\baselineskip
	\begin{subfigure}[t]{\textwidth}
		\centering
		\includegraphics[scale=0.6]{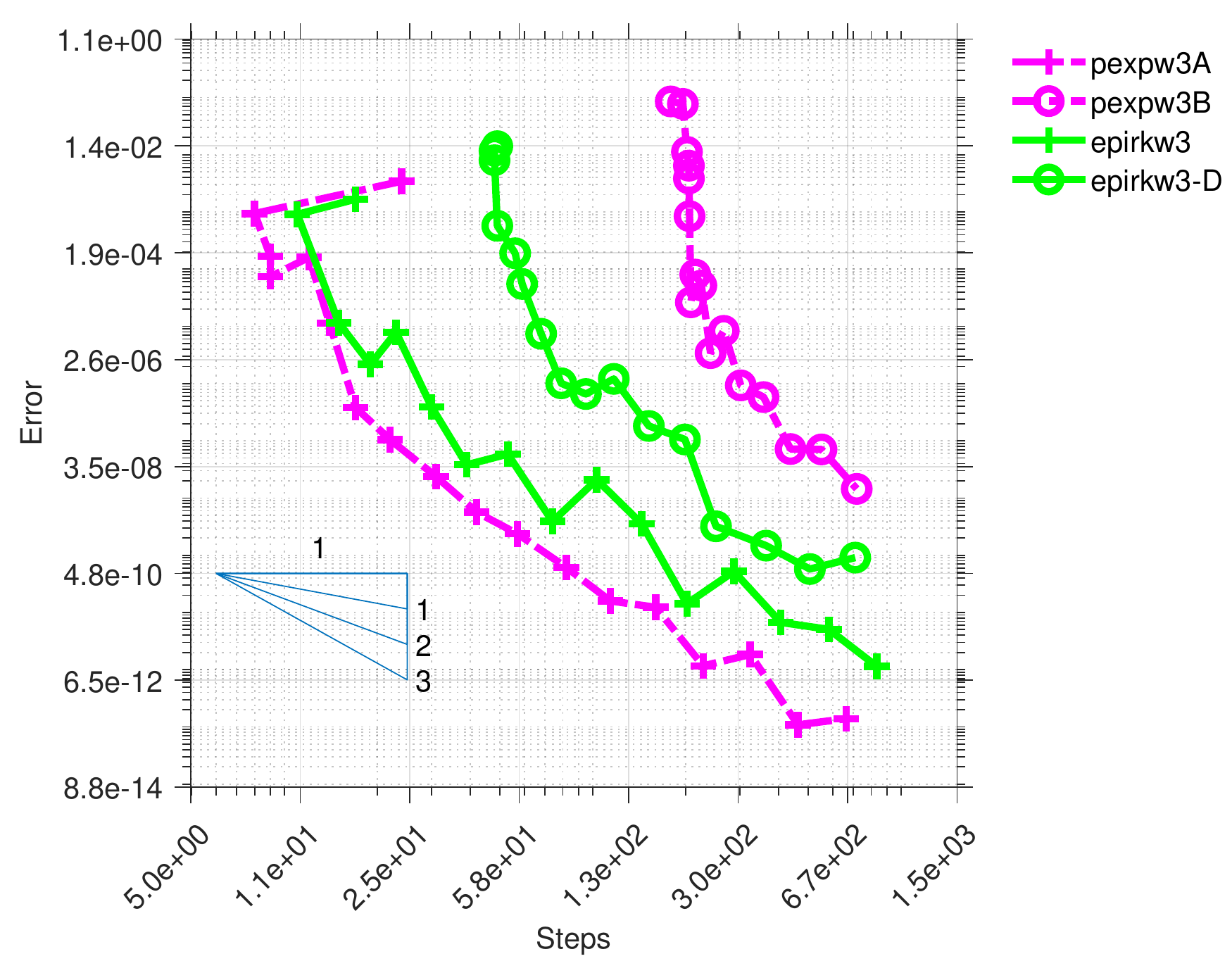}
		\caption{Convergence diagram.}
	\end{subfigure}
	\caption{Adaptive timestep experiments using Allen-Cahn \eqref{eqn:Allen_Cahn},  $300 \times 300$ grid (II). $\alpha = 1$, $\gamma = 100$ \label{fig:experiment-adaptive-timestep-AllenCahn-GammaEquals100}}
\end{figure*}

\begin{figure*}[htb!]
	\centering
	\begin{subfigure}[t]{\textwidth}
		\centering
		\includegraphics[scale=0.6]{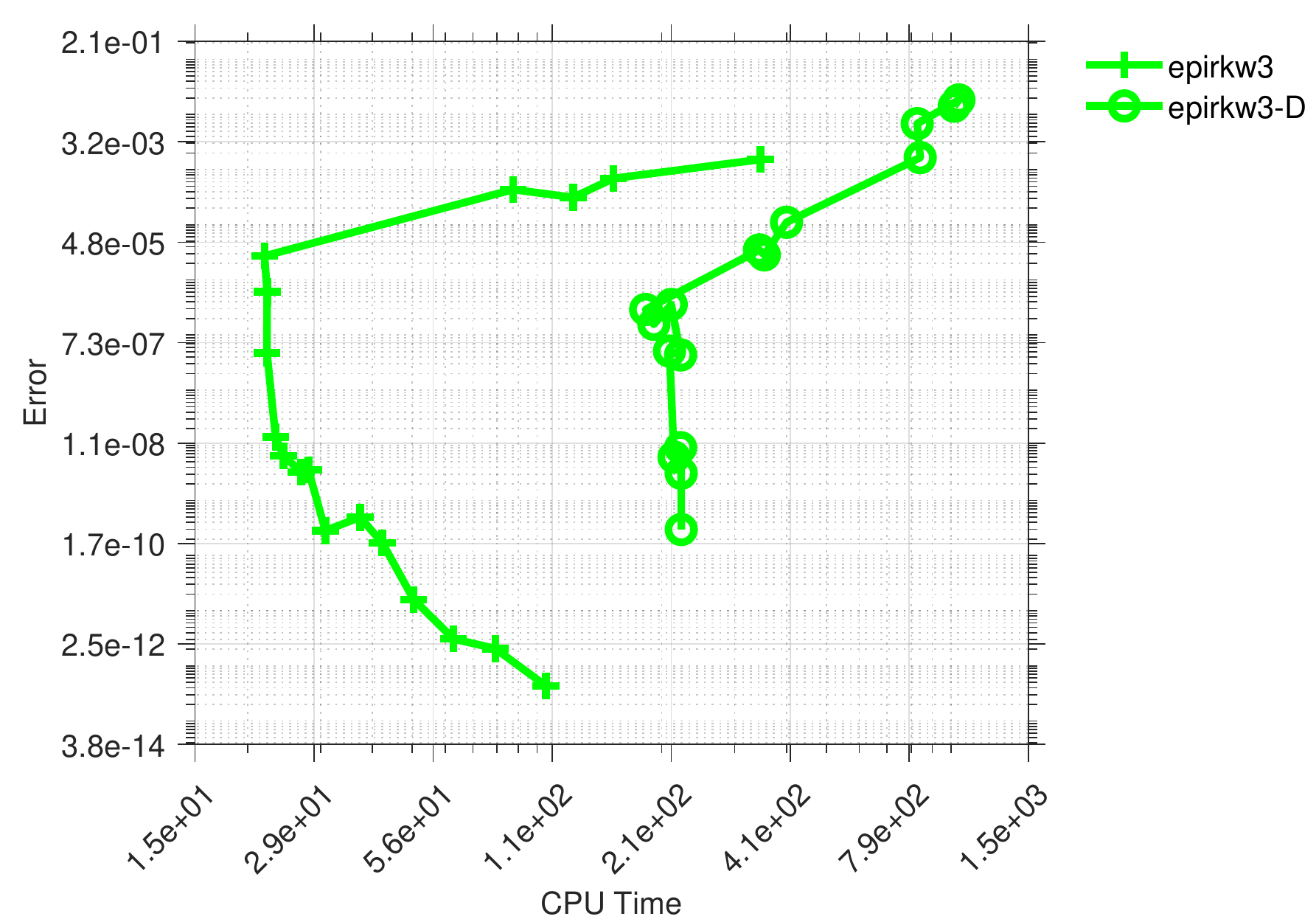}
		\caption{Work precision diagram.}
	\end{subfigure}
	~
	\begin{subfigure}[t]{\textwidth}
		\centering
		\includegraphics[scale=0.6]{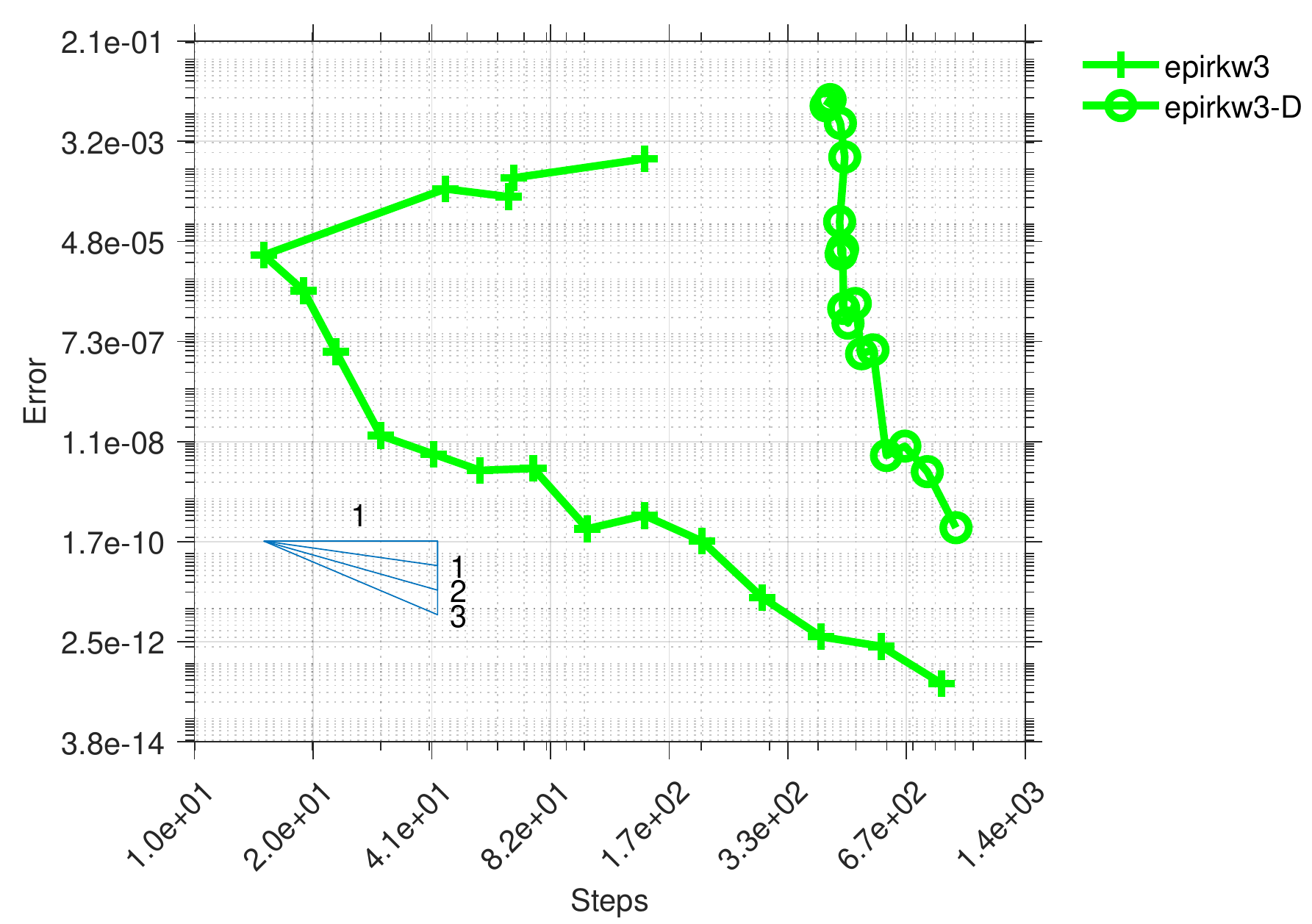}
		\caption{Convergence diagram.}
	\end{subfigure}
	\caption{Adaptive timestep experiments using Allen-Cahn \eqref{eqn:Allen_Cahn},  $300 \times 300$ grid (III). $\alpha = 1$, $\gamma = 1000$ \label{fig:experiment-adaptive-timestep-AllenCahn-GammaEquals1000}}
\end{figure*}

\paragraph{Reversible Gray-Scott results}

Figure \ref{fig:experiment-adaptive-timestep-ReversibleGrayScott} shows the work-precision and the convergence diagrams for numerical experiments with the reversible Gray-Scott system \eqref{eqn:Reversible_Gray_Scott}. We compare PEXPW3A, EPIRKW3 and a modified implementation of PEXPW3A, PEXPW3A-RD, for reaction-diffusion systems. The modified implementation PEXPW3A-RD evaluates the $\varphi$ functions on the diffusion Jacobian in parallel using \texttt{MATLAB} parallel pool, while it does not explicitly parallelize the evaluation of $\varphi$ functions on the reaction Jacobian. As is clear from the figure, PEXP3WA is more stable and efficient than the unpartitioned EPIRKW3 method on this test problem. We even test a variant of EPIRKW3 with a different set of coefficients, EPIRKW3c and find that it performs similarly to EPIRKW3. 

Furthermore, by partitioning the problem and dealing with the individual partitions separately, we have more room to implement computational optimizations discussed in Section \ref{sec:implementation} in the implementation. While PEXPW3A and PEXPW3A-RD both take the same number of steps, since PEXPW3A-RD parallelizes the computation of $\varphi$ functions on the diffusion matrices, it is twice as fast as PEXPW3A in obtaining the solution.

\begin{figure*}[htb!]
	\centering
	\begin{subfigure}[t]{\textwidth}
		\centering
		\includegraphics[scale=0.6]{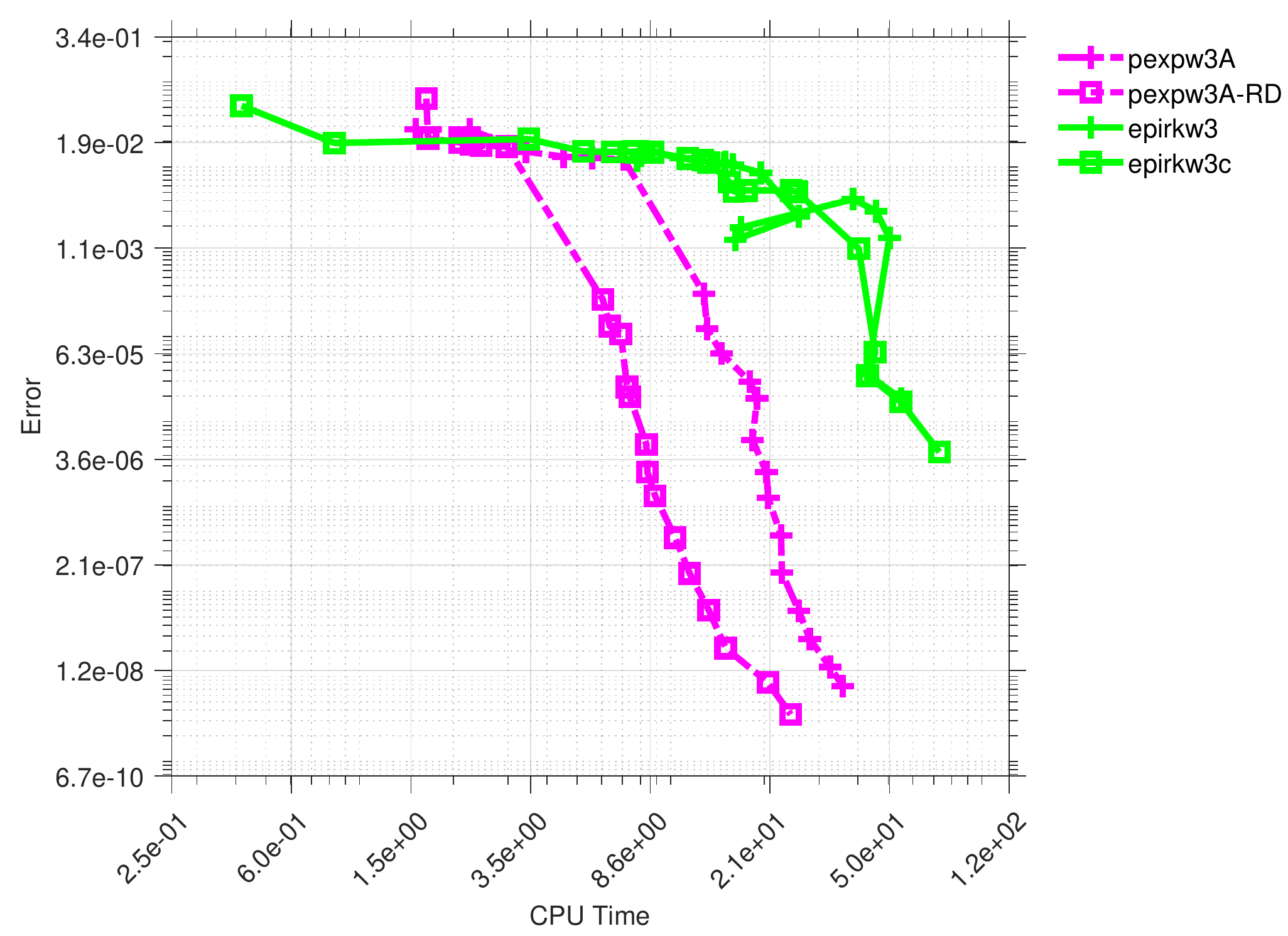}
		\caption{Work precision diagram.}
	\end{subfigure}
	\vskip\baselineskip
	\begin{subfigure}[t]{\textwidth}
		\centering
		\includegraphics[scale=0.6]{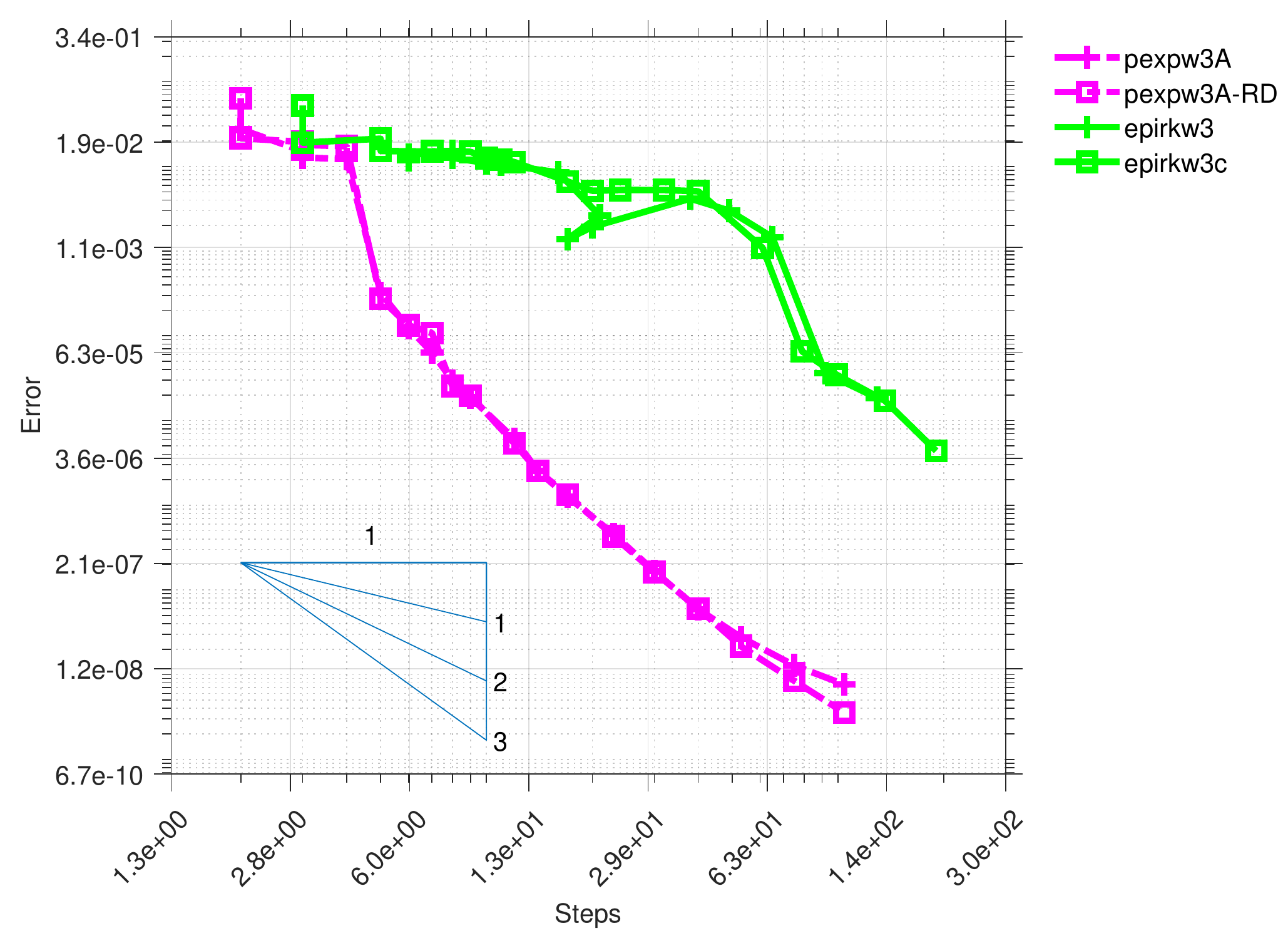}
		\caption{Convergence diagram.}
	\end{subfigure}
	\caption{Adaptive timestep experiments using Reversible Gray-Scott \eqref{eqn:Reversible_Gray_Scott} on $100 \times 100$ grid. $D_U = 2$, $D_V = 1$, $D_P = 0.1$, $k_1 = 1$, $k_2 = 0.055$, $k_{-1} = 0.001$, $k_{-2} = 0.001$, $f = 0.028$\label{fig:experiment-adaptive-timestep-ReversibleGrayScott}}
\end{figure*}

\begin{landscape}
\begin{table}[htb!]
	\begin{tabular}{|l|c|c|c|c|c|}
		\hline
		\multicolumn{6}{|c|}{\textit{Fixed Timestep Experiments}} \\ \hline\hline
		\textbf{Experiment} & \multicolumn{3}{|c|}{\textbf{PEPIRKW and PEXPW}} & \multicolumn{2}{|c|}{\textbf{expirkw3}}\\
		\hline
		\textbf{Lorenz-96} & \multicolumn{3}{|c|}{$\fone(y) = A_{N \times N} \,y$} & \multicolumn{2}{|c|}{}    \\ 
			& \multicolumn{3}{|c|}{$\ftwo(y) = F(y) - A_{N \times N} \,y$} & \multicolumn{2}{|c|}{$\Wb = J$}      \\ 
			& \multicolumn{3}{|c|}{$J^{\{1\}}(y) = \textnormal{diag}(A_{N \times N})$} &  \multicolumn{2}{|c|}{}     \\ 
			& \multicolumn{3}{|c|}{ $J^{\{2\}}(y) = \textnormal{diag}(J  - A_{N \times N})$} &  \multicolumn{2}{|c|}{}    \\ 
		\hline
		\hline
		\textbf{Experiment} & \multicolumn{3}{|c|}{\textbf{PEXPW}} & \textbf{expirkw3} & \textbf{expirkw3-D} \\
		\hline
		\textbf{Semilinear-Parabolic} & \multicolumn{3}{|c|}{{$\fone(\bar{y}) = [\mathbf{D}_{xx}\,y;~0]$}}   & &   \\ 
			& \multicolumn{3}{|c|}{$\ftwo(\bar{y}) = [{1}/{(1+y^2)} + \Phi(\bar{y});~1]$} & $\Wb = J$  &  $\Wb = J^{\{1\}}$  \\ 
			& \multicolumn{3}{|c|}{$J^{\{1\}}(\bar{y}) = \partial \fone/\partial \bar{y}$}   &  &  \\ 
			& \multicolumn{3}{|c|}{ $J^{\{2\}}(\bar{y}) = \partial \ftwo/\partial \bar{y}$} &  &     \\ 
		\hline
		\hline	
		\textbf{Allen-Cahn,  150x150 Grid} & \multicolumn{3}{|c|}{$\fone({y}) = \alpha \cdot \mathbf{D} y$}  & &   \\ 
			& \multicolumn{3}{|c|}{$\ftwo({y}) = \gamma \cdot (y - y^3)$} & $\Wb = J$  &  $\Wb = J^{\{1\}}$  \\ 
			& \multicolumn{3}{|c|}{$J^{\{1\}}({y}) = \partial \fone/\partial {y}$} &  &  \\ 
			& \multicolumn{3}{|c|}{ $J^{\{2\}}({y}) = \partial \ftwo/\partial {y}$}   &  &     \\ 
		
		\hline
		\hline
		\multicolumn{6}{|c|}{\textit{Adaptive Timestep Experiments}} \\ 
		\hline
		\hline
		\textbf{Experiment} & \multicolumn{3}{|c|}{\textbf{PEXPW}}  & \multicolumn{2}{|c|}{\textbf{expirkw3}} \\
		\hline
		\textbf{Allen-Cahn,  300x300 Grid (I, II \& III)} & \multicolumn{3}{|c|}{$\fone({y}) = \alpha \cdot \mathbf{D} y$}  &  \multicolumn{2}{|c|}{} \\ 
			& \multicolumn{3}{|c|}{$\ftwo({y}) = \gamma \cdot (y - y^3)$} & \multicolumn{2}{|c|}{$\Wb = J$}  \\ 
			& \multicolumn{3}{|c|}{$J^{\{1\}}({y}) = \partial \fone/\partial {y}$}   & \multicolumn{2}{|c|}{} \\ 
			& \multicolumn{3}{|c|}{ $J^{\{2\}}({y}) = \partial \ftwo/\partial {y}$}  & \multicolumn{2}{|c|}{} \\ 
		\hline
		\hline
		\textbf{Reversible Gray-Scott on 100x100 Grid} & \multicolumn{3}{|c|}{$\fone([U; V; P]) = \begin{bmatrix}
			D_U \mathbf{D} & 0 & 0\\
			0 & D_V \mathbf{D} & 0\\
			0 & 0 & D_P \mathbf{D}\\
		\end{bmatrix} \begin{bmatrix}
			U \\
			V \\
			P
		\end{bmatrix}$}  &  \multicolumn{2}{|c|}{} \\ 
		& \multicolumn{3}{|c|}{$\ftwo([U; V; P]) = \begin{bmatrix}
			- k_1 U V^2 + f(1 - U) + k_{-1} V^3 \\
			k_1 U V^2 - (f + k_2)V - k_{-1} V^3  + k_{-2} P\\
		    - k_{-2} P - f P \\
		\end{bmatrix}$}  &  \multicolumn{2}{|c|}{$\Wb = J$} \\ 
		& \multicolumn{3}{|c|}{$J^{\{1\}}({[U; V; P]}) = \partial \fone/\partial {[U; V; P]}$}   & \multicolumn{2}{|c|}{} \\ 
		& \multicolumn{3}{|c|}{ $J^{\{2\}}({[U; V; P]}) = \partial \ftwo/\partial {[U; V; P]}$}  & \multicolumn{2}{|c|}{} \\ 
		\hline
	\end{tabular}
	\caption{Splittings of various problems into components. $\mathbf{D}$ is the discretized diffusion operator; $\bar{y} = [y; t]$, is the augmented variable of the autonomous system; $A_{N\times N}$ is a random $N\times N$ matrix; $\textnormal{diag}(.)$ returns the diagonal of the matrix; $F(y) = \fone(y) + \ftwo(y)$ is the right-hand side of the non-split autonomous form of the ODE system; $\Jb$ is the Jacobian of the right-hand side $F(y)$.\label{table:splittings_used_in_partitioned_methods}}
\end{table}
\end{landscape}

%%%%%%%%%%%%%%%%%%%%%%%%%%%%%%%%%%%%%%%

%%%%%%%%%%%%%%%%%%%%%%%%%%%%%%%%%%%%%%%
% !TEX root = Nonstiff_pexpw.tex
%%%%%%%%%%%%%%%%%%%%%%%%%%%%%%%%%%%%%%%
\section{Conclusions}
\label{sec:conclusions}
%%%%%%%%%%%%%%%%%%%%%%%%%%%%%%%%%%%%%%%

Exponential methods developed over the last couple of decades have shown great promise as an alternative to traditional implicit or explicit time discretizations. They can be cheaper than implicit methods, and more stable than explicit methods. 

This work builds partitioned exponential methods for multiphysics systems driven by two simultaneous processes. The new family solves each component process by an exponential integrator, and information between components is exchanged using coupling terms. We consider two approaches to the construction and analysis of these methods, one based on splitting the component functions into linear and nonlinear terms (split-RHS methods), and the other based on approximating the Jacobians of individual components (W-type methods).  Two new formulations of partitioned exponential methods are proposed: PEXPW that generalizes exponential-Rosenbrock schemes, and PEPIRKW that generalizes EPIRK schemes. A third family, PSEPIRK, obtained by averaging two unpartitioned sEPIRK methods, is discussed; this family is only of theoretical interest (from the point of view of order conditions).

In the implementation of the proposed partitioned methods the matrix-exponential-like functions, which are an integral part of exponential integrators, are evaluated on the Jacobians of the individual component functions, whereas these functions are applied to the full (coupled) Jacobian in an unpartitioned method. If the individual Jacobians have computationally favorable structures, as it is often the case, then the computational expense of evaluating matrix-exponential-like functions, which constitute the bulk of the computational cost of exponential time integrators, is greatly reduced. For instance, in reaction-diffusion systems with two or more species, the diffusion Jacobian is block-diagonal, which enables the evaluation of  matrix-exponential-like functions on individual blocks in parallel. Our numerical experiments show that this strategy can lead to partitioned exponential methods that are at least twice as fast as the unpartitioned counterpart.

Partitioned exponential methods share with all splitting methods the drawback of a possibly reduced stability. Roughly speaking, while individual components are treated implicitly (exponentially), the coupling/interaction between components is treated explicitly. Thus the partitioned exponential methods can be valuable when the two components are stiff, but interact weakly with each other. Our numerical tests illustrate that the partitioned methods can exhibit more stable behavior than an unpartitioned method in some stiffness regimes; however, in some very stiff regimes, partitioned methods can fail to obtain a solution, and in such cases unpartitioned methods should be the solver of choice. 

%%%%%%%%%%%%%%%%%%%%%%%%%%%%%%%%%%%%%%%

% Turn on the line below to include the trees PDF file into the 
% document here. Or move it to the appendix after the references.
%\includepdf[landscape=true,pages=-]{TPSTrees_table.pdf}

%%%%%%%%%%%%%%%%%%%%%%%%%%%%%%%%%%%%%%%
\section*{Acknowledgements}
%%%%%%%%%%%%%%%%%%%%%%%%%%%%%%%%%%%%%%%

This work has been supported in part by NSF through awards NSF CCF--1613905, NSF ACI--1709727, by AFOSR through AFOSR DDDAS FA9550--17--1--0015 and by the Computational Science Laboratory at Virginia Tech.

\pagestyle{plain}
%%%%%%%%%%%%%%%%%%%%%%%%%%%%%%%%%%%%%%%
\section*{References}
%%%%%%%%%%%%%%%%%%%%%%%%%%%%%%%%%%%%%%%
%%%%%%%%%%%%%%%%%%%%%%%%%%%%%%%%%%%%%%%

%%%%%%%%%%%%%%%%%%%%%%%%%%%%%%%%%%%%%%%

\clearpage

\appendix
\clearpage
%%%%%%%%%%%%%%%%%%%%%%%%%%%%%%%%%%%%%%%
% !TEX root = Nonstiff_pexpw.tex
%%%%%%%%%%%%%%%%%%%%%%%%%%%%%%%%%%%%%%%
\section{Partitioned sEPIRK method using the averaging (AVG) strategy}
\label{sec:AVG-SEPIRK}
%%%%%%%%%%%%%%%%%%%%%%%%%%%%%%%%%%%%%%%
In order to build higher-order sEPIRK-style partitioned exponential methods we resort to an averaging (AVG) strategy.
For a two-way partition system, one starts with two independent methods derived from the variation-of-constants formula using $\Lone y$ and $\Ltwo y$ as the linear parts, respectively. The internal and final stages of the two methods are then averaged to build a partitioned method with unified stages. The forward difference operator in the resultant method will be computed on the nodes corresponding to the unified stage values. By choosing the abscissae of the methods to be the same, we can avoid interpolation, and ensure we are averaging solutions at the same time instant.

Consider the two sEPIRK methods with the same number of internal stages $s$,  obtained by linearizing one component function at a time and applying the  variation-of-constants formula; the two methods are presented in Formulation \ref{frm:avg_sepirk_method_pairs}.
\begin{formulation}
	\begin{eqnarray*}
		Y_i^{\{1\}} &=& y_n +  {a}_{i,1}^{\{1\}} {\Psi}_{i,1}^{\{1\}}\bigl(h{g}_{i,1}^{\{1\}}\Lone\bigr) \, h F(y_n)  \\
		&& +  \sum_{l=2}^{i} {a}_{i,l}^{\{1\}} {\Psi}_{i,l}^{\{1\}}\bigl(h{g}_{i,l}^{\{1\}}\Lone\bigr)  h\Delta^{(l - 1)}(\None\big(y_n\big) + \ftwo(y_n)), \qquad i = 1, \hdots, s - 1, \nonumber\\
		y_{n+1}^{\{1\}} &=& y_n + {b}_{1}^{\{1\}} {\Psi}_{s,1}^{\{1\}}\bigl(h{g}_{s,1}^{\{1\}}\Lone\bigr)  \, h F(y_n)  \\
		&& +  {\sum_{l=2}^{s} {b}_{l}^{\{1\}} {\Psi}_{{s},l}^{\{1\}}\bigl(h{g}_{{s},l}^{\{1\}}\Lone\bigr)  h\Delta^{(l - 1)}(\None\big(y_n\big) + \ftwo(y_n))}, \nonumber
	\end{eqnarray*}
	\begin{eqnarray*}
		Y_i^{\{2\}} &=& y_n +  {a}_{i,1}^{\{2\}} {\Psi}_{i,1}^{\{2\}}\bigl(h{g}_{i,1}^{\{2\}}\Ltwo\bigr) \, h F(y_n)  \\
		&& +  \sum_{l=2}^{i} {a}_{i,l}^{\{2\}} {\Psi}_{i,l}^{\{2\}}\bigl(h{g}_{i,l}^{\{2\}}\Ltwo\bigr)  h\Delta^{(l - 1)}(\Ntwo\big(y_n\big) + \fone(y_n)), \qquad i = 1, \hdots, s - 1, \nonumber\\
		y_{n+1}^{\{2\}} &=& y_n + {b}_{1}^{\{2\}} {\Psi}_{s,1}^{\{2\}}\bigl(h{g}_{s,1}^{\{2\}}\Ltwo\bigr)  \, h F(y_n)  \\
		&& +  {\sum_{l=2}^{s} {b}_{l}^{\{2\}} {\Psi}_{{s},l}^{\{2\}}\bigl(h{g}_{{s},l}^{\{2\}}\Ltwo\bigr)  h\Delta^{(l - 1)}(\Ntwo\big(y_n\big) + \fone(y_n))}.
	\end{eqnarray*}
	\caption{Pair of sEPIRK methods obtained by linearizing one partition at a time \label{frm:avg_sepirk_method_pairs}}
\end{formulation}

Averaging the stages and solutions of Formulation \ref{frm:avg_sepirk_method_pairs} leads to a new partitioned method  constructed using the AVG strategy. This method, named PSEPIRK, is shown in Formulation \ref{frm:avg_sepirk_method}. As already stated, we propose that the forward differences in the new method be evaluated on nodes corresponding to unified stage values in both the internal and final stages, i.e., while forward differences in the individual methods may be evaluated on their respective internal stage values ($Y_i^{\{1\}}$ and $Y_i^{\{2\}}$), the forward differences in the new method will be evaluated on the unified stage values, $Y_i$. Additionally, we redefine the coefficients in the new method by absorbing the factor $1/2$ that arises from averaging into the coefficients of both the internal and final stages. By using the AVG strategy we end up with an ARK-type method, PSEPIRK.

\begin{formulation}
	\begin{eqnarray*}
		Y_i &=& (Y_i^{\{1\}} + Y_i^{\{2\}})/2\\
		&=& y_n +  {a}_{i,1}^{\{1\}} {\Psi}_{i,1}^{\{1\}}\bigl(h{g}_{i,1}^{\{1\}}\Lone\bigr) \, h F(y_n)  \\
		&& +  \sum_{l=2}^{i} {a}_{i,l}^{\{1\}} {\Psi}_{i,l}^{\{1\}}\bigl(h{g}_{i,l}^{\{1\}}\Lone\bigr)  h\Delta^{(l - 1)}(\None\big(y_n\big) + \ftwo(y_n)) \\
		&& +~{a}_{i,1}^{\{2\}} {\Psi}_{i,1}^{\{2\}}\bigl(h{g}_{i,1}^{\{2\}}\Ltwo\bigr) \, h F(y_n)  \\
		&& +  \sum_{l=2}^{i} {a}_{i,l}^{\{2\}} {\Psi}_{i,l}^{\{2\}}\bigl(h{g}_{i,l}^{\{2\}}\Ltwo\bigr)  h\Delta^{(l - 1)}(\Ntwo\big(y_n\big) + \fone(y_n)) \\
		y_{n+1} &=& (y_{n+1}^{\{1\}} + y_{n+1}^{\{2\}})/2\\
		&=& y_n + {b}_{1}^{\{1\}} {\Psi}_{s,1}^{\{1\}}\bigl(h{g}_{s,1}^{\{1\}}\Lone\bigr)  \, h F(y_n)  \\
		&& +  {\sum_{l=2}^{s} {b}_{l}^{\{1\}} {\Psi}_{{s},l}^{\{1\}}\bigl(h{g}_{{s},l}^{\{1\}}\Lone\bigr)  h\Delta^{(l - 1)}(\None\big(y_n\big) + \ftwo(y_n))} \\
		&& +~{b}_{1}^{\{2\}} {\Psi}_{s,1}^{\{2\}}\bigl(h{g}_{s,1}^{\{2\}}\Ltwo\bigr)  \, h F(y_n)  \\
		&& +  {\sum_{l=2}^{s} {b}_{l}^{\{2\}} {\Psi}_{{s},l}^{\{2\}}\bigl(h{g}_{{s},l}^{\{2\}}\Ltwo\bigr)  h\Delta^{(l - 1)}(\Ntwo\big(y_n\big) + \fone(y_n))}
	\end{eqnarray*}
	\caption{Partitioned sEPIRK Method using AVG strategy (PSEPIRK).}
	 \label{frm:avg_sepirk_method}
\end{formulation}

Order conditions for these methods are given in
\iflong
\ref{sec:order-conditions-3stage} 
\else
{\cite[Appendix E]{narayanamurthi2019} }
\fi
and a three-stage third order method with second order embedded coefficients is given in
\iflong
\ref{sec:coefficients-PSEPIRK}. 
\else
{\cite[Appendix D]{narayanamurthi2019}.} 
\fi
We optimized all the dependent and free coefficients together to ensure that the total norm of the coefficients was minimized while satisfying the family of solutions.

%%%%%%%%%%%%%%%%%%%%%%%%%%%%%%%%%%%%%%%

\clearpage
%%%%%%%%%%%%%%%%%%%%%%%%%%%%%%%%%%%%%%%
\section{Coefficients of PEXPW methods} 
\label{sec:coefficients-PEXPW}
\begin{method}
\footnotesize
\begin{equation*}
\begin{split}
A =
.
\end{split}
\end{equation*}
\caption{PEXPW3A\label{method:pexpw3a_coefficients}}
\end{method}
\begin{method}
\footnotesize
\begin{equation*}
\begin{split}
A = %
.
\end{split}
\end{equation*}
\caption{PEXPW3B\label{method:pexpw3b_coefficients}}
\end{method}
%%%%%%%%%%%%%%%%%%%%%%%%%%%%%%%%%%%%%%%

\clearpage
%%%%%%%%%%%%%%%%%%%%%%%%%%%%%%%%%%%%%%%
\section{Coefficients of PEPIRKW methods} 
\label{sec:coefficients-PEPIRKW}

\begin{method}
\scriptsize
\begin{equation*}
\begin{split}
A_{11} = %
.
\end{split}
\end{equation*}
\caption{PEPIRKW3A\label{method:pepirkw3a_coefficients}}
\end{method}
\begin{method}
\scriptsize
\begin{equation*}
\begin{split}
A_{11} = %
.
\end{split}
\end{equation*}
\caption{PEPIRKW3B\label{method:pepirkw3b_coefficients}}
\end{method}
%%%%%%%%%%%%%%%%%%%%%%%%%%%%%%%%%%%%%%%

\iflong
\clearpage
%%%%%%%%%%%%%%%%%%%%%%%%%%%%%%%%%%%%%%%
\section{Coefficients of partitioned sEPIRK methods based on averaging} 
\label{sec:coefficients-PSEPIRK}

\begin{method}
\scriptsize
\begin{equation*}
\begin{split}
A_{1} = %
.
\end{split}
\end{equation*}
\caption{PSEPIRKB}
\label{method:psepirkb_coefficients}
\end{method}
%%%%%%%%%%%%%%%%%%%%%%%%%%%%%%%%%%%%%%%

\clearpage
%%%%%%%%%%%%%%%%%%%%%%%%%%%%%%%%%%%%%%%
\section{Order conditions for three-stage methods} 
\label{sec:order-conditions-3stage}
%%%%%%%%%%%%%%%%%%%%%%%%%%%%%%%%%%%%%%%
\subsection{PEXPW Order Conditions}
\begin{longtable}{|P{1.25cm}|P{1.2cm}|P{9.5cm}|P{2.5cm}|}
    \caption{Order conditions for the three stage PEXPW method}
    \label{Table:3-stage-PEXPW-order-condition}\\
    \toprule
    \centering \textbf{Tree} &
    \centering \textbf{Order} &
    \centering \textbf{$B^{\#}(\mbf{y}_{n+1})$} &
	\centering \textbf{$B^{\#}\big(\mbf{y}(t_{n}\,+h)\big)$} \tabularnewline 
    \midrule
    \endfirsthead
    \toprule
    \centering \textbf{Tree} &
    \centering \textbf{Order} &
    \centering \textbf{$B^{\#}(\mbf{y}_{n+1})$} &
	\centering \textbf{$B^{\#}\big(\mbf{y}(t_{n}\,+h)\big)$} \tabularnewline 
    \midrule
    \endhead
\hline
 $ \tau_{1}^{\textnormal{TPS}} $ & 1 & \begin{dmath*} {b^{\{1\}}_{1}}+{b^{\{1\}}_{2}}+{b^{\{1\}}_{3}} \end{dmath*} & $1 $  \\
\hline
 $ \tau_{3}^{\textnormal{TPS}} $ & 1 & \begin{dmath*} {b^{\{2\}}_{1}}+{b^{\{2\}}_{2}}+{b^{\{2\}}_{3}}+{b^{\{2\}}_{4}} \end{dmath*} & $1 $  \\
\hline
 $ \tau_{5}^{\textnormal{TPS}} $ & 2 & \begin{dmath*} {a^{\{1,1\}}_{2,1}} {b^{\{1\}}_{2}}+{b^{\{1\}}_{3}} ({a^{\{1,1\}}_{3,1}}+{a^{\{1,1\}}_{3,2}}) \end{dmath*} & $\frac{1}{2} $  \\
\hline
 $ \tau_{7}^{\textnormal{TPS}} $ & 2 & \begin{dmath*} {a^{\{1,2\}}_{2,1}} {b^{\{1\}}_{2}}+{b^{\{1\}}_{3}} ({a^{\{1,2\}}_{3,1}}+{a^{\{1,2\}}_{3,2}}) \end{dmath*} & $\frac{1}{2} $  \\
\hline
 $ \tau_{9}^{\textnormal{TPS}} $ & 2 & \begin{dmath*} \frac{{b^{\{1\}}_{1}} {g^{\{1,1\}}_{1,1}}}{2}+{b^{\{1\}}_{2}} \left({g^{\{1,1\}}_{2,1}}+\frac{{g^{\{1,1\}}_{2,2}}}{2}\right)+{b^{\{1\}}_{3}} \left({g^{\{1,1\}}_{3,1}}+{g^{\{1,1\}}_{3,2}}+\frac{{g^{\{1,1\}}_{3,3}}}{2}\right) \end{dmath*} & $0 $  \\
\hline
 $ \tau_{11}^{\textnormal{TPS}} $ & 2 & \begin{dmath*} {b^{\{1\}}_{2}} {g^{\{1,2\}}_{2,1}}+{b^{\{1\}}_{3}} ({g^{\{1,2\}}_{3,1}}+{g^{\{1,2\}}_{3,2}}) \end{dmath*} & $0 $  \\
\hline
 $ \tau_{13}^{\textnormal{TPS}} $ & 2 & \begin{dmath*} {a^{\{2,1\}}_{2,1}} {b^{\{2\}}_{2}}+{b^{\{2\}}_{3}} ({a^{\{2,1\}}_{3,1}}+{a^{\{2,1\}}_{3,2}})+{b^{\{2\}}_{4}} ({a^{\{2,1\}}_{4,1}}+{a^{\{2,1\}}_{4,2}}+{a^{\{2,1\}}_{4,3}}) \end{dmath*} & $\frac{1}{2} $  \\
\hline
 $ \tau_{15}^{\textnormal{TPS}} $ & 2 & \begin{dmath*} {a^{\{2,2\}}_{2,1}} {b^{\{2\}}_{2}}+{b^{\{2\}}_{3}} ({a^{\{2,2\}}_{3,1}}+{a^{\{2,2\}}_{3,2}})+{b^{\{2\}}_{4}} ({a^{\{2,2\}}_{4,1}}+{a^{\{2,2\}}_{4,2}}+{a^{\{2,2\}}_{4,3}}) \end{dmath*} & $\frac{1}{2} $  \\
\hline
\end{longtable}

\subsection{PEPIRKW Order Conditions}
\begin{longtable}{|P{1.25cm}|P{1.2cm}|P{9.5cm}|P{2.5cm}|}
    \caption{Order conditions for the three stage PEPIRKW method}
    \label{Table:3-stage-PEPIRKW-order-condition}\\
    \toprule
    \centering \textbf{Tree} &
    \centering \textbf{Order} &
    \centering \textbf{$B^{\#}(\mbf{y}_{n+1})$} &
	\centering \textbf{$B^{\#}\big(\mbf{y}(t_{n}\,+h)\big)$} \tabularnewline 
    \midrule
    \endfirsthead
    \toprule
    \centering \textbf{Tree} &
    \centering \textbf{Order} &
    \centering \textbf{$B^{\#}(\mbf{y}_{n+1})$} &
	\centering \textbf{$B^{\#}\big(\mbf{y}(t_{n}\,+h)\big)$} \tabularnewline 
    \midrule
    \endhead
\hline
 $\tau_{1}^{\textnormal{TPS}}$ & 1 & \begin{dmath*} {b^{\{1\}}_{1}} {p^{\{1,1\}}_{1,1}}  \end{dmath*} &  $ 1 $ \\
\hline
 $\tau_{3}^{\textnormal{TPS}}$ & 1 & \begin{dmath*} {b^{\{2\}}_{1}} {p^{\{2,2\}}_{1,1}}  \end{dmath*} &  $ 1 $ \\
\hline
 $\tau_{5}^{\textnormal{TPS}}$ & 2 & \begin{dmath*} {b^{\{1\}}_{3}} \left({p^{\{1,1\}}_{3,1}} ({a^{\{1,1\}}_{2,1}} {p^{\{1,1\}}_{1,1}}-2 {a^{\{1,1\}}_{1,1}} {p^{\{1,1\}}_{1,1}})+\frac{1}{2} {p^{\{1,1\}}_{3,2}} ({a^{\{1,1\}}_{2,1}} {p^{\{1,1\}}_{1,1}}-2 {a^{\{1,1\}}_{1,1}} {p^{\{1,1\}}_{1,1}})+\frac{1}{6} {p^{\{1,1\}}_{3,3}} ({a^{\{1,1\}}_{2,1}}   {p^{\{1,1\}}_{1,1}}-2 {a^{\{1,1\}}_{1,1}} {p^{\{1,1\}}_{1,1}})\right)+{b^{\{1\}}_{2}} \left({a^{\{1,1\}}_{1,1}} {p^{\{1,1\}}_{1,1}} {p^{\{1,1\}}_{2,1}}+\frac{{a^{\{1,1\}}_{1,1}} {p^{\{1,1\}}_{1,1}} {p^{\{1,1\}}_{2,2}}}{2}\right) \end{dmath*} &  $\frac{1}{2}$ \\
\hline
 $\tau_{7}^{\textnormal{TPS}}$ & 2 & \begin{dmath*} {b^{\{1\}}_{3}} \left({p^{\{1,1\}}_{3,1}} ({a^{\{1,2\}}_{2,1}} {p^{\{1,2\}}_{1,1}}-2 {a^{\{1,2\}}_{1,1}} {p^{\{1,2\}}_{1,1}})+\frac{1}{2} {p^{\{1,1\}}_{3,2}} ({a^{\{1,2\}}_{2,1}} {p^{\{1,2\}}_{1,1}}-2 {a^{\{1,2\}}_{1,1}} {p^{\{1,2\}}_{1,1}})+\frac{1}{6} {p^{\{1,1\}}_{3,3}} ({a^{\{1,2\}}_{2,1}} {p^{\{1,2\}}_{1,1}}-2 {a^{\{1,2\}}_{1,1}} {p^{\{1,2\}}_{1,1}})\right)+{b^{\{1\}}_{2}} \left({a^{\{1,2\}}_{1,1}} {p^{\{1,1\}}_{2,1}} {p^{\{1,2\}}_{1,1}}+\frac{{a^{\{1,2\}}_{1,1}} {p^{\{1,1\}}_{2,2}} {p^{\{1,2\}}_{1,1}}}{2}\right) \end{dmath*} &  $\frac{1}{2}$ \\
\hline
 $\tau_{9}^{\textnormal{TPS}}$ & 2 & \begin{dmath*} \frac{{b^{\{1\}}_{1}} {g^{\{1,1\}}_{3,1}} {p^{\{1,1\}}_{1,1}}}{2}  \end{dmath*} &  $ 0 $ \\
\hline
\end{longtable}

\begin{landscape}
	\subsection{PSEPIRK Order Conditions}
	\setlength\LTleft{-1.25in}
\setlength\LTright{-1.25in}
\begin{longtable}{|P{1.25cm}|P{1.2cm}|P{9.5cm}|P{2.5cm}|}
    \caption{Order conditions for the three stage PSEPIRK method}
    \label{Table:3-stage-PSEPIRK-order-condition}\\
    \toprule
    \centering \textbf{Tree} &
    \centering \textbf{Order} &
    \centering \textbf{$B^{\#}(\mbf{y}_{n+1})$} &
	\centering \textbf{$B^{\#}\big(\mbf{y}(t_{n}\,+h)\big)$} \tabularnewline 
    \midrule
    \endfirsthead
    \toprule
    \centering \textbf{Tree} &
    \centering \textbf{Order} &
    \centering \textbf{$B^{\#}(\mbf{y}_{n+1})$} &
	\centering \textbf{$B^{\#}\big(\mbf{y}(t_{n}\,+h)\big)$} \tabularnewline 
    \midrule
    \endhead
\hline
 $ \tau_{1}^{\textnormal{TPS}} $ & 1 & \begin{dmath*} {b^{\{1\}}_{1}} {p^{\{1\}}_{1,1}}+{b^{\{2\}}_{1}} {p^{\{2\}}_{1,1}} \end{dmath*} & $1 $  \\
\hline
 $ \tau_{2}^{\textnormal{TPS}} $ & 1 & \begin{dmath*} {b^{\{1\}}_{1}} {p^{\{1\}}_{1,1}}+{b^{\{2\}}_{1}} {p^{\{2\}}_{1,1}} \end{dmath*} & $1 $  \\
\hline
 $ \tau_{3}^{\textnormal{TPS}} $ & 1 & \begin{dmath*} {b^{\{1\}}_{1}} {p^{\{1\}}_{1,1}}+{b^{\{2\}}_{1}} {p^{\{2\}}_{1,1}} \end{dmath*} & $1 $  \\
\hline
 $ \tau_{4}^{\textnormal{TPS}} $ & 1 & \begin{dmath*} {b^{\{1\}}_{1}} {p^{\{1\}}_{1,1}}+{b^{\{2\}}_{1}} {p^{\{2\}}_{1,1}} \end{dmath*} & $1 $  \\
\hline
\end{longtable}
\end{landscape}

\pagestyle{empty}
\newgeometry{bindingoffset=0.2in,%
left=0.5in,right=0.5in,top=0.5in,bottom=0.5in,%
footskip=.25in}

\begin{landscape}
%%%%%%%%%%%%%%%%%%%%%%%%%%%%%%%%%%%%%%%
\section{TPS-trees and the Corresponding B-Series}
\label{sec:TPS}
%%%%%%%%%%%%%%%%%%%%%%%%%%%%%%%%%%%%%%%
	\setcounter{table}{0}
	\renewcommand{\thetable}{F\arabic{table}}
	\begin{table}[!htb]
\caption{TPS-trees up to order three (part one of eleven).\label{table:tps_trees1}}
\label{Table:PWtrees1} \centering 
\global\long\def\arraystretch{1.5}%
{\footnotesize{}{}}%
\begin{tabular}{|>{\centering\arraybackslash}m{1.1in}|>{\centering\arraybackslash}m{1.3in}|>{\centering\arraybackslash}m{1.3in}|>{\centering\arraybackslash}m{1.3in}|>{\centering\arraybackslash}m{1.3in}|>{\centering\arraybackslash}m{1.3in}|}
\hline 
{\footnotesize{}{}{$\tau$ } }  & {\footnotesize{}{}{ \begin{tikzpicture}[scale=.5,
							meagrecirc/.style={circle,draw,fill=black!100,thick},
							fatcirc/.style={circle,draw,thick},
							fatrect/.style = {rectangle,draw,thick},
							meagrerect/.style = {rectangle,fill=black!100,draw,thick}]
						\node[meagrecirc] (j) at (0,0) [label=right:$j$] {};
					\end{tikzpicture} } }  & {\footnotesize{}{}{ \begin{tikzpicture}[scale=.5,
							meagrecirc/.style={circle,draw, fill=black!100,thick},
							fatcirc/.style={circle,draw,thick},
							fatrect/.style = {rectangle,draw,thick},
							meagrerect/.style = {rectangle,fill=black!100,draw,thick}]
						\node[meagrerect] (j) at (0,0) [label=right:$j$] {};
					\end{tikzpicture} } }  & {\footnotesize{}{}{ \begin{tikzpicture}[scale=.5,
							meagrecirc/.style={circle,draw, fill=black!100,thick},
							fatcirc/.style={circle,draw,thick},
							fatrect/.style = {rectangle,draw,thick},
							meagrerect/.style = {rectangle,fill=black!100,draw,thick}]
						\node[fatcirc] (j) at (0,0) [label=right:$j$] {};
					\end{tikzpicture} }}  & {\footnotesize{}{}{ \begin{tikzpicture}[scale=.5,
							meagrecirc/.style={circle,draw, fill=black!100,thick},
							fatcirc/.style={circle,draw,thick},
							fatrect/.style = {rectangle,draw,thick},
							meagrerect/.style = {rectangle,fill=black!100,draw,thick}]
						\node[fatrect] (j) at (0,0) [label=right:$j$] {};
					\end{tikzpicture}} }  & {\footnotesize{}{}{ \begin{tikzpicture}[scale=.5,
							meagrecirc/.style={circle,draw, fill=black!100,thick},
							fatcirc/.style={circle,draw,thick},
							fatrect/.style = {rectangle,draw,thick},
							meagrerect/.style = {rectangle,fill=black!100,draw,thick}]
						\node[meagrecirc] (j) at (0,0) [label=right:$j$] {};
						\node[meagrecirc] (k) at (1,1) [label=right:$k$] {};
						\draw[-] (j) -- (k);
					\end{tikzpicture} } } \tabularnewline
\hline
\end{tabular}
\end{table}

\begin{table}[!htb]
\caption{TPS-trees up to order three (part two of eleven).\label{table:tps_trees2}}
\label{Table:PWtrees2} \centering 
\global\long\def\arraystretch{1.5}%
{\footnotesize{}{}}%
\begin{tabular}{|>{\centering\arraybackslash}m{1.1in}|>{\centering\arraybackslash}m{1.3in}|>{\centering\arraybackslash}m{1.3in}|>{\centering\arraybackslash}m{1.3in}|>{\centering\arraybackslash}m{1.3in}|>{\centering\arraybackslash}m{1.3in}|}
\hline 
{\footnotesize{}{}{$\tau$ } }  & {\footnotesize{}{}{ \begin{tikzpicture}[scale=.5,
							meagrecirc/.style={circle,draw, fill=black!100,thick},
							fatcirc/.style={circle,draw,thick},
							fatrect/.style = {rectangle,draw,thick},
							meagrerect/.style = {rectangle,fill=black!100,draw,thick}]
						\node[meagrerect] (j) at (0,0) [label=right:$j$] {};
						\node[fatcirc] (k) at (1,1) [label=right:$k$] {};
						\draw[-] (j) -- (k);
					\end{tikzpicture} } }  & {\footnotesize{}{}{ \begin{tikzpicture}[scale=.5,
							meagrecirc/.style={circle,draw, fill=black!100,thick},
							fatcirc/.style={circle,draw,thick},
							fatrect/.style = {rectangle,draw,thick},
							meagrerect/.style = {rectangle,fill=black!100,draw,thick}]
						\node[meagrerect] (j) at (0,0) [label=right:$j$] {};
						\node[fatrect] (k) at (1,1) [label=right:$k$] {};
						\draw[-] (j) -- (k);
					\end{tikzpicture}} } & {\footnotesize{}{}{ \begin{tikzpicture}[scale=.5,
							meagrecirc/.style={circle,draw, fill=black!100,thick},
							fatcirc/.style={circle,draw,thick},
							fatrect/.style = {rectangle,draw,thick},
							meagrerect/.style = {rectangle,fill=black!100,draw,thick}]
						\node[fatcirc] (j) at (0,0) [label=right:$j$] {};
						\node[meagrecirc] (k) at (1,1) [label=right:$k$] {};
						\draw[-] (j) -- (k);
					\end{tikzpicture} } }  & {\footnotesize{}{}{ \begin{tikzpicture}[scale=.5,
							meagrecirc/.style={circle,draw, fill=black!100,thick},
							fatcirc/.style={circle,draw,thick},
							fatrect/.style = {rectangle,draw,thick},
							meagrerect/.style = {rectangle,fill=black!100,draw,thick}]
						\node[fatcirc] (j) at (0,0) [label=right:$j$] {};
						\node[meagrerect] (k) at (1,1) [label=right:$k$] {};
						\draw[-] (j) -- (k);
					\end{tikzpicture} } }  & {\footnotesize{}{}{ \begin{tikzpicture}[scale=.5,
							meagrecirc/.style={circle,draw, fill=black!100,thick},
							fatcirc/.style={circle,draw,thick},
							fatrect/.style = {rectangle,draw,thick},
							meagrerect/.style = {rectangle,fill=black!100,draw,thick}]
						\node[fatcirc] (j) at (0,0) [label=right:$j$] {};
						\node[fatcirc] (k) at (1,1) [label=right:$k$] {};
						\draw[-] (j) -- (k);
					\end{tikzpicture} } }\tabularnewline
\hline
\end{tabular}
\end{table}

\begin{table}[!htb]
\caption{TPS-trees up to order three (part three of eleven).\label{table:tps_trees3}}
\label{Table:PWtrees3} \centering 
\global\long\def\arraystretch{1.5}%
{\footnotesize{}{}}%
\begin{tabular}{|>{\centering\arraybackslash}m{1.1in}|>{\centering\arraybackslash}m{1.3in}|>{\centering\arraybackslash}m{1.3in}|>{\centering\arraybackslash}m{1.3in}|>{\centering\arraybackslash}m{1.3in}|>{\centering\arraybackslash}m{1.3in}|}
\hline 
{\footnotesize{}{}{$\tau$ } }  & {\footnotesize{}{}{ \begin{tikzpicture}[scale=.5,
							meagrecirc/.style={circle,draw, fill=black!100,thick},
							fatcirc/.style={circle,draw,thick},
							fatrect/.style = {rectangle,draw,thick},
							meagrerect/.style = {rectangle,fill=black!100,draw,thick}]
						\node[meagrecirc] (j) at (0,0) [label=right:$j$] {};
						\node[meagrecirc] (l) at (1,1) [label=right:$l$] {};
						\node[meagrecirc] (k) at (-1,1) [label=left:$k$] {};
						\draw[-] (j) -- (k);
						\draw[-] (j) -- (l);
					\end{tikzpicture} } } & {\footnotesize{}{}{ \begin{tikzpicture}[scale=.5,
							meagrecirc/.style={circle,draw, fill=black!100,thick},
							fatcirc/.style={circle,draw,thick},
							fatrect/.style = {rectangle,draw,thick},
							meagrerect/.style = {rectangle,fill=black!100,draw,thick}]
						\node[meagrecirc] (j) at (0,0) [label=right:$j$] {};
						\node[meagrerect] (l) at (1,1) [label=right:$l$] {};
						\node[meagrecirc] (k) at (-1,1) [label=left:$k$] {};
						\draw[-] (j) -- (k);
						\draw[-] (j) -- (l);
					\end{tikzpicture} } }  & {\footnotesize{}{}{ \begin{tikzpicture}[scale=.5,
							meagrecirc/.style={circle,draw, fill=black!100,thick},
							fatcirc/.style={circle,draw,thick},
							fatrect/.style = {rectangle,draw,thick},
							meagrerect/.style = {rectangle,fill=black!100,draw,thick}]
						\node[meagrecirc] (j) at (0,0) [label=right:$j$] {};
						\node[fatcirc] (l) at (1,1) [label=right:$l$] {};
						\node[meagrecirc] (k) at (-1,1) [label=left:$k$] {};
						\draw[-] (j) -- (k);
						\draw[-] (j) -- (l);
					\end{tikzpicture} } }  & {\footnotesize{}{}{ \begin{tikzpicture}[scale=.5,
							meagrecirc/.style={circle,draw, fill=black!100,thick},
							fatcirc/.style={circle,draw,thick},
							fatrect/.style = {rectangle,draw,thick},
							meagrerect/.style = {rectangle,fill=black!100,draw,thick}]
						\node[meagrecirc] (j) at (0,0) [label=right:$j$] {};
						\node[fatrect] (l) at (1,1) [label=right:$l$] {};
						\node[meagrecirc] (k) at (-1,1) [label=left:$k$] {};
						\draw[-] (j) -- (k);
						\draw[-] (j) -- (l);
					\end{tikzpicture}} } & {\footnotesize{}{}{ \begin{tikzpicture}[scale=.5,
							meagrecirc/.style={circle,draw, fill=black!100,thick},
							fatcirc/.style={circle,draw,thick},
							fatrect/.style = {rectangle,draw,thick},
							meagrerect/.style = {rectangle,fill=black!100,draw,thick}]
						\node[meagrecirc] (j) at (0,0) [label=right:$j$] {};
						\node[meagrerect] (l) at (1,1) [label=right:$l$] {};
						\node[meagrerect] (k) at (-1,1) [label=left:$k$] {};
						\draw[-] (j) -- (k);
						\draw[-] (j) -- (l);
					\end{tikzpicture} } } \tabularnewline
\hline
\end{tabular}
\end{table}

\begin{table}[!htb]
\caption{TPS-trees up to order three (part four of eleven).\label{table:tps_trees4}}
\label{Table:PWtrees4} \centering 
\global\long\def\arraystretch{1.5}%
{\footnotesize{}{}}%
\begin{tabular}{|>{\centering\arraybackslash}m{1.1in}|>{\centering\arraybackslash}m{1.3in}|>{\centering\arraybackslash}m{1.3in}|>{\centering\arraybackslash}m{1.3in}|>{\centering\arraybackslash}m{1.3in}|>{\centering\arraybackslash}m{1.3in}|}
\hline 
{\footnotesize{}{}{$\tau$ } }  & {\footnotesize{}{}{ \begin{tikzpicture}[scale=.5,
							meagrecirc/.style={circle,draw, fill=black!100,thick},
							fatcirc/.style={circle,draw,thick},
							fatrect/.style = {rectangle,draw,thick},
							meagrerect/.style = {rectangle,fill=black!100,draw,thick}]
						\node[fatcirc] (j) at (0,0) [label=right:$j$] {};
						\node[meagrecirc] (l) at (1,1) [label=right:$l$] {};
						\node[meagrecirc] (k) at (-1,1) [label=left:$k$] {};
						\draw[-] (j) -- (k);
						\draw[-] (j) -- (l);
					\end{tikzpicture} } }  & {\footnotesize{}{}{ \begin{tikzpicture}[scale=.5,
							meagrecirc/.style={circle,draw, fill=black!100,thick},
							fatcirc/.style={circle,draw,thick},
							fatrect/.style = {rectangle,draw,thick},
							meagrerect/.style = {rectangle,fill=black!100,draw,thick}]
						\node[fatcirc] (j) at (0,0) [label=right:$j$] {};
						\node[meagrerect] (l) at (1,1) [label=right:$l$] {};
						\node[meagrecirc] (k) at (-1,1) [label=left:$k$] {};
						\draw[-] (j) -- (k);
						\draw[-] (j) -- (l);
					\end{tikzpicture}} }  & {\footnotesize{}{}{ \begin{tikzpicture}[scale=.5,
							meagrecirc/.style={circle,draw, fill=black!100,thick},
							fatcirc/.style={circle,draw,thick},
							fatrect/.style = {rectangle,draw,thick},
							meagrerect/.style = {rectangle,fill=black!100,draw,thick}]
						\node[fatcirc] (j) at (0,0) [label=right:$j$] {};
						\node[fatcirc] (l) at (1,1) [label=right:$l$] {};
						\node[meagrecirc] (k) at (-1,1) [label=left:$k$] {};
						\draw[-] (j) -- (k);
						\draw[-] (j) -- (l);
					\end{tikzpicture} } } & {\footnotesize{}{}{ \begin{tikzpicture}[scale=.5,
							meagrecirc/.style={circle,draw, fill=black!100,thick},
							fatcirc/.style={circle,draw,thick},
							fatrect/.style = {rectangle,draw,thick},
							meagrerect/.style = {rectangle,fill=black!100,draw,thick}]
						\node[fatcirc] (j) at (0,0) [label=right:$j$] {};
						\node[fatrect] (l) at (1,1) [label=right:$l$] {};
						\node[meagrecirc] (k) at (-1,1) [label=left:$k$] {};
						\draw[-] (j) -- (k);
						\draw[-] (j) -- (l);
					\end{tikzpicture} } }  & {\footnotesize{}{}{ \begin{tikzpicture}[scale=.5,
							meagrecirc/.style={circle,draw, fill=black!100,thick},
							fatcirc/.style={circle,draw,thick},
							fatrect/.style = {rectangle,draw,thick},
							meagrerect/.style = {rectangle,fill=black!100,draw,thick}]
						\node[fatcirc] (j) at (0,0) [label=right:$j$] {};
						\node[meagrerect] (l) at (1,1) [label=right:$l$] {};
						\node[meagrerect] (k) at (-1,1) [label=left:$k$] {};
						\draw[-] (j) -- (k);
						\draw[-] (j) -- (l);
					\end{tikzpicture} } } \tabularnewline
\hline
\end{tabular}
\end{table}

\begin{table}[!htb]
\caption{TPS-trees up to order three (part five of eleven).\label{table:tps_trees5}}
\label{Table:PWtrees5} \centering 
\global\long\def\arraystretch{1.5}%
{\footnotesize{}{}}%
\begin{tabular}{|>{\centering\arraybackslash}m{1.1in}|>{\centering\arraybackslash}m{1.3in}|>{\centering\arraybackslash}m{1.3in}|>{\centering\arraybackslash}m{1.3in}|>{\centering\arraybackslash}m{1.3in}|>{\centering\arraybackslash}m{1.3in}|}
\hline 
{\footnotesize{}{}{$\tau$ } }  & {\footnotesize{}{}{ \begin{tikzpicture}[scale=.5,
							meagrecirc/.style={circle,draw, fill=black!100,thick},
							fatcirc/.style={circle,draw,thick},
							fatrect/.style = {rectangle,draw,thick},
							meagrerect/.style = {rectangle,fill=black!100,draw,thick}]
						\node[meagrecirc] (j) at (0,0) [label=right:$j$] {};
						\node[meagrecirc] (k) at (1,1) [label=right:$k$] {};
						\node[meagrecirc] (l) at (0,2) [label=left:$l$] {};
						\draw[-] (j) -- (k);
						\draw[-] (k) -- (l);
					\end{tikzpicture} } }  & {\footnotesize{}{}{ \begin{tikzpicture}[scale=.5,
							meagrecirc/.style={circle,draw, fill=black!100,thick},
							fatcirc/.style={circle,draw,thick},
							fatrect/.style = {rectangle,draw,thick},
							meagrerect/.style = {rectangle,fill=black!100,draw,thick}]
						\node[meagrecirc] (j) at (0,0) [label=right:$j$] {};
						\node[meagrecirc] (k) at (1,1) [label=right:$k$] {};
						\node[meagrerect] (l) at (0,2) [label=left:$l$] {};
						\draw[-] (j) -- (k);
						\draw[-] (k) -- (l);
					\end{tikzpicture} } } & {\footnotesize{}{}{ \begin{tikzpicture}[scale=.5,
							meagrecirc/.style={circle,draw, fill=black!100,thick},
							fatcirc/.style={circle,draw,thick},
							fatrect/.style = {rectangle,draw,thick},
							meagrerect/.style = {rectangle,fill=black!100,draw,thick}]
						\node[meagrecirc] (j) at (0,0) [label=right:$j$] {};
						\node[meagrecirc] (k) at (1,1) [label=right:$k$] {};
						\node[fatcirc] (l) at (0,2) [label=left:$l$] {};
						\draw[-] (j) -- (k);
						\draw[-] (k) -- (l);
					\end{tikzpicture} } }  & {\footnotesize{}{}{ \begin{tikzpicture}[scale=.5,
							meagrecirc/.style={circle,draw, fill=black!100,thick},
							fatcirc/.style={circle,draw,thick},
							fatrect/.style = {rectangle,draw,thick},
							meagrerect/.style = {rectangle,fill=black!100,draw,thick}]
						\node[meagrecirc] (j) at (0,0) [label=right:$j$] {};
						\node[meagrecirc] (k) at (1,1) [label=right:$k$] {};
						\node[fatrect] (l) at (0,2) [label=left:$l$] {};
						\draw[-] (j) -- (k);
						\draw[-] (k) -- (l);
					\end{tikzpicture}} }  & {\footnotesize{}{}{ \begin{tikzpicture}[scale=.5,
							meagrecirc/.style={circle,draw, fill=black!100,thick},
							fatcirc/.style={circle,draw,thick},
							fatrect/.style = {rectangle,draw,thick},
							meagrerect/.style = {rectangle,fill=black!100,draw,thick}]
						\node[meagrecirc] (j) at (0,0) [label=right:$j$] {};
						\node[meagrerect] (k) at (1,1) [label=right:$k$] {};
						\node[meagrecirc] (l) at (0,2) [label=left:$l$] {};
						\draw[-] (j) -- (k);
						\draw[-] (k) -- (l);
					\end{tikzpicture} } }\tabularnewline
\hline
\end{tabular}
\end{table}

\begin{table}[!htb]
\caption{TPS-trees up to order three (part six of eleven).\label{table:tps_trees6}}
\label{Table:PWtrees6} \centering 
\global\long\def\arraystretch{1.5}%
{\footnotesize{}{}}%
\begin{tabular}{|>{\centering\arraybackslash}m{1.1in}|>{\centering\arraybackslash}m{1.3in}|>{\centering\arraybackslash}m{1.3in}|>{\centering\arraybackslash}m{1.3in}|>{\centering\arraybackslash}m{1.3in}|>{\centering\arraybackslash}m{1.3in}|}
\hline 
{\footnotesize{}{}{$\tau$ }} & {\footnotesize{}{}{ \begin{tikzpicture}[scale=.5,
							meagrecirc/.style={circle,draw, fill=black!100,thick},
							fatcirc/.style={circle,draw,thick},
							fatrect/.style = {rectangle,draw,thick},
							meagrerect/.style = {rectangle,fill=black!100,draw,thick}]
						\node[meagrecirc] (j) at (0,0) [label=right:$j$] {};
						\node[fatcirc] (k) at (1,1) [label=right:$k$] {};
						\node[fatcirc] (l) at (0,2) [label=left:$l$] {};
						\draw[-] (j) -- (k);
						\draw[-] (k) -- (l);
					\end{tikzpicture} } } & {\footnotesize{}{}{ \begin{tikzpicture}[scale=.5,
							meagrecirc/.style={circle,draw, fill=black!100,thick},
							fatcirc/.style={circle,draw,thick},
							fatrect/.style = {rectangle,draw,thick},
							meagrerect/.style = {rectangle,fill=black!100,draw,thick}]
						\node[meagrecirc] (j) at (0,0) [label=right:$j$] {};
						\node[fatcirc] (k) at (1,1) [label=right:$k$] {};
						\node[fatrect] (l) at (0,2) [label=left:$l$] {};
						\draw[-] (j) -- (k);
						\draw[-] (k) -- (l);
					\end{tikzpicture}} }  & {\footnotesize{}{}{ \begin{tikzpicture}[scale=.5,
							meagrecirc/.style={circle,draw, fill=black!100,thick},
							fatcirc/.style={circle,draw,thick},
							fatrect/.style = {rectangle,draw,thick},
							meagrerect/.style = {rectangle,fill=black!100,draw,thick}]
						\node[meagrecirc] (j) at (0,0) [label=right:$j$] {};
						\node[fatrect] (k) at (1,1) [label=right:$k$] {};
						\node[meagrecirc] (l) at (0,2) [label=left:$l$] {};
						\draw[-] (j) -- (k);
						\draw[-] (k) -- (l);
					\end{tikzpicture} } }  & {\footnotesize{}{}{ \begin{tikzpicture}[scale=.5,
							meagrecirc/.style={circle,draw, fill=black!100,thick},
							fatcirc/.style={circle,draw,thick},
							fatrect/.style = {rectangle,draw,thick},
							meagrerect/.style = {rectangle,fill=black!100,draw,thick}]
						\node[meagrecirc] (j) at (0,0) [label=right:$j$] {};
						\node[fatrect] (k) at (1,1) [label=right:$k$] {};
						\node[meagrerect] (l) at (0,2) [label=left:$l$] {};
						\draw[-] (j) -- (k);
						\draw[-] (k) -- (l);
					\end{tikzpicture} } } & {\footnotesize{}{}{ \begin{tikzpicture}[scale=.5,
							meagrecirc/.style={circle,draw, fill=black!100,thick},
							fatcirc/.style={circle,draw,thick},
							fatrect/.style = {rectangle,draw,thick},
							meagrerect/.style = {rectangle,fill=black!100,draw,thick}]
						\node[meagrecirc] (j) at (0,0) [label=right:$j$] {};
						\node[fatrect] (k) at (1,1) [label=right:$k$] {};
						\node[fatcirc] (l) at (0,2) [label=left:$l$] {};
						\draw[-] (j) -- (k);
						\draw[-] (k) -- (l);
					\end{tikzpicture} } } \tabularnewline
\hline
\end{tabular}
\end{table}

\begin{table}[!htb]
\caption{TPS-trees up to order three (part seven of eleven).\label{table:tps_trees7}}
\label{Table:PWtrees7} \centering 
\global\long\def\arraystretch{1.5}%
{\footnotesize{}{}}%
\begin{tabular}{|>{\centering\arraybackslash}m{1.1in}|>{\centering\arraybackslash}m{1.3in}|>{\centering\arraybackslash}m{1.3in}|>{\centering\arraybackslash}m{1.3in}|>{\centering\arraybackslash}m{1.3in}|>{\centering\arraybackslash}m{1.3in}|}
\hline 
{\footnotesize{}{}{$\tau$ } }  & {\footnotesize{}{}{ \begin{tikzpicture}[scale=.5,
							meagrecirc/.style={circle,draw, fill=black!100,thick},
							fatcirc/.style={circle,draw,thick},
							fatrect/.style = {rectangle,draw,thick},
							meagrerect/.style = {rectangle,fill=black!100,draw,thick}]
						\node[meagrerect] (j) at (0,0) [label=right:$j$] {};
						\node[meagrerect] (k) at (1,1) [label=right:$k$] {};
						\node[meagrecirc] (l) at (0,2) [label=left:$l$] {};
						\draw[-] (j) -- (k);
						\draw[-] (k) -- (l);
					\end{tikzpicture} } }  & {\footnotesize{}{}{ \begin{tikzpicture}[scale=.5,
							meagrecirc/.style={circle,draw, fill=black!100,thick},
							fatcirc/.style={circle,draw,thick},
							fatrect/.style = {rectangle,draw,thick},
							meagrerect/.style = {rectangle,fill=black!100,draw,thick}]
						\node[meagrerect] (j) at (0,0) [label=right:$j$] {};
						\node[meagrerect] (k) at (1,1) [label=right:$k$] {};
						\node[meagrerect] (l) at (0,2) [label=left:$l$] {};
						\draw[-] (j) -- (k);
						\draw[-] (k) -- (l);
					\end{tikzpicture} } }  & {\footnotesize{}{}{ \begin{tikzpicture}[scale=.5,
							meagrecirc/.style={circle,draw, fill=black!100,thick},
							fatcirc/.style={circle,draw,thick},
							fatrect/.style = {rectangle,draw,thick},
							meagrerect/.style = {rectangle,fill=black!100,draw,thick}]
						\node[meagrerect] (j) at (0,0) [label=right:$j$] {};
						\node[meagrerect] (k) at (1,1) [label=right:$k$] {};
						\node[fatcirc] (l) at (0,2) [label=left:$l$] {};
						\draw[-] (j) -- (k);
						\draw[-] (k) -- (l);
					\end{tikzpicture} } } & {\footnotesize{}{}{ \begin{tikzpicture}[scale=.5,
							meagrecirc/.style={circle,draw, fill=black!100,thick},
							fatcirc/.style={circle,draw,thick},
							fatrect/.style = {rectangle,draw,thick},
							meagrerect/.style = {rectangle,fill=black!100,draw,thick}]
						\node[meagrerect] (j) at (0,0) [label=right:$j$] {};
						\node[meagrerect] (k) at (1,1) [label=right:$k$] {};
						\node[fatrect] (l) at (0,2) [label=left:$l$] {};
						\draw[-] (j) -- (k);
						\draw[-] (k) -- (l);
					\end{tikzpicture}} }  & {\footnotesize{}{}{ \begin{tikzpicture}[scale=.5,
							meagrecirc/.style={circle,draw, fill=black!100,thick},
							fatcirc/.style={circle,draw,thick},
							fatrect/.style = {rectangle,draw,thick},
							meagrerect/.style = {rectangle,fill=black!100,draw,thick}]
						\node[meagrerect] (j) at (0,0) [label=right:$j$] {};
						\node[fatcirc] (k) at (1,1) [label=right:$k$] {};
						\node[meagrecirc] (l) at (0,2) [label=left:$l$] {};
						\draw[-] (j) -- (k);
						\draw[-] (k) -- (l);
					\end{tikzpicture} } } \tabularnewline
\hline
\end{tabular}
\end{table}

\begin{table}[!htb]
\caption{TPS-trees up to order three (part eight of eleven).\label{table:tps_trees8}}
\label{Table:PWtrees8} \centering 
\global\long\def\arraystretch{1.5}%
{\footnotesize{}{}}%
\begin{tabular}{|>{\centering\arraybackslash}m{1.1in}|>{\centering\arraybackslash}m{1.3in}|>{\centering\arraybackslash}m{1.3in}|>{\centering\arraybackslash}m{1.3in}|>{\centering\arraybackslash}m{1.3in}|>{\centering\arraybackslash}m{1.3in}|}
\hline 
{\footnotesize{}{}{$\tau$ } }  & {\footnotesize{}{}{ \begin{tikzpicture}[scale=.5,
							meagrecirc/.style={circle,draw, fill=black!100,thick},
							fatcirc/.style={circle,draw,thick},
							fatrect/.style = {rectangle,draw,thick},
							meagrerect/.style = {rectangle,fill=black!100,draw,thick}]
						\node[meagrerect] (j) at (0,0) [label=right:$j$] {};
						\node[fatrect] (k) at (1,1) [label=right:$k$] {};
						\node[fatcirc] (l) at (0,2) [label=left:$l$] {};
						\draw[-] (j) -- (k);
						\draw[-] (k) -- (l);
					\end{tikzpicture} } }  & {\footnotesize{}{}{ \begin{tikzpicture}[scale=.5,
							meagrecirc/.style={circle,draw, fill=black!100,thick},
							fatcirc/.style={circle,draw,thick},
							fatrect/.style = {rectangle,draw,thick},
							meagrerect/.style = {rectangle,fill=black!100,draw,thick}]
						\node[meagrerect] (j) at (0,0) [label=right:$j$] {};
						\node[fatrect] (k) at (1,1) [label=right:$k$] {};
						\node[fatrect] (l) at (0,2) [label=left:$l$] {};
						\draw[-] (j) -- (k);
						\draw[-] (k) -- (l);
					\end{tikzpicture}} } & {\footnotesize{}{}{ \begin{tikzpicture}[scale=.5,
							meagrecirc/.style={circle,draw, fill=black!100,thick},
							fatcirc/.style={circle,draw,thick},
							fatrect/.style = {rectangle,draw,thick},
							meagrerect/.style = {rectangle,fill=black!100,draw,thick}]
						\node[fatcirc] (j) at (0,0) [label=right:$j$] {};
						\node[meagrecirc] (k) at (1,1) [label=right:$k$] {};
						\node[meagrecirc] (l) at (0,2) [label=left:$l$] {};
						\draw[-] (j) -- (k);
						\draw[-] (k) -- (l);
					\end{tikzpicture} } }  & {\footnotesize{}{}{ \begin{tikzpicture}[scale=.5,
							meagrecirc/.style={circle,draw, fill=black!100,thick},
							fatcirc/.style={circle,draw,thick},
							fatrect/.style = {rectangle,draw,thick},
							meagrerect/.style = {rectangle,fill=black!100,draw,thick}]
						\node[fatcirc] (j) at (0,0) [label=right:$j$] {};
						\node[meagrecirc] (k) at (1,1) [label=right:$k$] {};
						\node[meagrerect] (l) at (0,2) [label=left:$l$] {};
						\draw[-] (j) -- (k);
						\draw[-] (k) -- (l);
					\end{tikzpicture} } }  & {\footnotesize{}{}{ \begin{tikzpicture}[scale=.5,
							meagrecirc/.style={circle,draw, fill=black!100,thick},
							fatcirc/.style={circle,draw,thick},
							fatrect/.style = {rectangle,draw,thick},
							meagrerect/.style = {rectangle,fill=black!100,draw,thick}]
						\node[fatcirc] (j) at (0,0) [label=right:$j$] {};
						\node[meagrecirc] (k) at (1,1) [label=right:$k$] {};
						\node[fatcirc] (l) at (0,2) [label=left:$l$] {};
						\draw[-] (j) -- (k);
						\draw[-] (k) -- (l);
					\end{tikzpicture} } }\tabularnewline
\hline
\end{tabular}
\end{table}

\begin{table}[!htb]
\caption{TPS-trees up to order three (part nine of eleven).\label{table:tps_trees9}}
\label{Table:PWtrees9} \centering 
\global\long\def\arraystretch{1.5}%
{\footnotesize{}{}}%
\begin{tabular}{|>{\centering\arraybackslash}m{1.1in}|>{\centering\arraybackslash}m{1.3in}|>{\centering\arraybackslash}m{1.3in}|>{\centering\arraybackslash}m{1.3in}|>{\centering\arraybackslash}m{1.3in}|>{\centering\arraybackslash}m{1.3in}|}
\hline 
{\footnotesize{}{}{$\tau$ } }  & {\footnotesize{}{}{ \begin{tikzpicture}[scale=.5,
							meagrecirc/.style={circle,draw, fill=black!100,thick},
							fatcirc/.style={circle,draw,thick},
							fatrect/.style = {rectangle,draw,thick},
							meagrerect/.style = {rectangle,fill=black!100,draw,thick}]
						\node[fatcirc] (j) at (0,0) [label=right:$j$] {};
						\node[fatcirc] (k) at (1,1) [label=right:$k$] {};
						\node[meagrecirc] (l) at (0,2) [label=left:$l$] {};
						\draw[-] (j) -- (k);
						\draw[-] (k) -- (l);
					\end{tikzpicture} } } & {\footnotesize{}{}{ \begin{tikzpicture}[scale=.5,
							meagrecirc/.style={circle,draw, fill=black!100,thick},
							fatcirc/.style={circle,draw,thick},
							fatrect/.style = {rectangle,draw,thick},
							meagrerect/.style = {rectangle,fill=black!100,draw,thick}]
						\node[fatcirc] (j) at (0,0) [label=right:$j$] {};
						\node[fatcirc] (k) at (1,1) [label=right:$k$] {};
						\node[meagrerect] (l) at (0,2) [label=left:$l$] {};
						\draw[-] (j) -- (k);
						\draw[-] (k) -- (l);
					\end{tikzpicture} } }  & {\footnotesize{}{}{ \begin{tikzpicture}[scale=.5,
							meagrecirc/.style={circle,draw, fill=black!100,thick},
							fatcirc/.style={circle,draw,thick},
							fatrect/.style = {rectangle,draw,thick},
							meagrerect/.style = {rectangle,fill=black!100,draw,thick}]
						\node[fatcirc] (j) at (0,0) [label=right:$j$] {};
						\node[fatcirc] (k) at (1,1) [label=right:$k$] {};
						\node[fatcirc] (l) at (0,2) [label=left:$l$] {};
						\draw[-] (j) -- (k);
						\draw[-] (k) -- (l);
					\end{tikzpicture} } }  & {\footnotesize{}{}{ \begin{tikzpicture}[scale=.5,
							meagrecirc/.style={circle,draw, fill=black!100,thick},
							fatcirc/.style={circle,draw,thick},
							fatrect/.style = {rectangle,draw,thick},
							meagrerect/.style = {rectangle,fill=black!100,draw,thick}]
						\node[fatcirc] (j) at (0,0) [label=right:$j$] {};
						\node[fatcirc] (k) at (1,1) [label=right:$k$] {};
						\node[fatrect] (l) at (0,2) [label=left:$l$] {};
						\draw[-] (j) -- (k);
						\draw[-] (k) -- (l);
					\end{tikzpicture}} } & {\footnotesize{}{}{ \begin{tikzpicture}[scale=.5,
							meagrecirc/.style={circle,draw, fill=black!100,thick},
							fatcirc/.style={circle,draw,thick},
							fatrect/.style = {rectangle,draw,thick},
							meagrerect/.style = {rectangle,fill=black!100,draw,thick}]
						\node[fatcirc] (j) at (0,0) [label=right:$j$] {};
						\node[fatrect] (k) at (1,1) [label=right:$k$] {};
						\node[meagrecirc] (l) at (0,2) [label=left:$l$] {};
						\draw[-] (j) -- (k);
						\draw[-] (k) -- (l);
					\end{tikzpicture} } } \tabularnewline
\hline
\end{tabular}
\end{table}

\begin{table}[!htb]
\caption{TPS-trees up to order three (part ten of eleven).\label{table:tps_trees10}}
\label{Table:PWtrees10} \centering 
\global\long\def\arraystretch{1.5}%
{\footnotesize{}{}}%
\begin{tabular}{|>{\centering\arraybackslash}m{1.1in}|>{\centering\arraybackslash}m{1.3in}|>{\centering\arraybackslash}m{1.3in}|>{\centering\arraybackslash}m{1.3in}|>{\centering\arraybackslash}m{1.3in}|>{\centering\arraybackslash}m{1.3in}|}
\hline 
{\footnotesize{}{}{$\tau$ } }  & {\footnotesize{}{}{ \begin{tikzpicture}[scale=.5,
							meagrecirc/.style={circle,draw, fill=black!100,thick},
							fatcirc/.style={circle,draw,thick},
							fatrect/.style = {rectangle,draw,thick},
							meagrerect/.style = {rectangle,fill=black!100,draw,thick}]
						\node[fatrect] (j) at (0,0) [label=right:$j$] {};
						\node[meagrecirc] (k) at (1,1) [label=right:$k$] {};
						\node[fatcirc] (l) at (0,2) [label=left:$l$] {};
						\draw[-] (j) -- (k);
						\draw[-] (k) -- (l);
					\end{tikzpicture} } }  & {\footnotesize{}{}{ \begin{tikzpicture}[scale=.5,
							meagrecirc/.style={circle,draw, fill=black!100,thick},
							fatcirc/.style={circle,draw,thick},
							fatrect/.style = {rectangle,draw,thick},
							meagrerect/.style = {rectangle,fill=black!100,draw,thick}]
						\node[fatrect] (j) at (0,0) [label=right:$j$] {};
						\node[meagrecirc] (k) at (1,1) [label=right:$k$] {};
						\node[fatrect] (l) at (0,2) [label=left:$l$] {};
						\draw[-] (j) -- (k);
						\draw[-] (k) -- (l);
					\end{tikzpicture}} }  & {\footnotesize{}{}{ \begin{tikzpicture}[scale=.5,
							meagrecirc/.style={circle,draw, fill=black!100,thick},
							fatcirc/.style={circle,draw,thick},
							fatrect/.style = {rectangle,draw,thick},
							meagrerect/.style = {rectangle,fill=black!100,draw,thick}]
						\node[fatrect] (j) at (0,0) [label=right:$j$] {};
						\node[meagrerect] (k) at (1,1) [label=right:$k$] {};
						\node[meagrecirc] (l) at (0,2) [label=left:$l$] {};
						\draw[-] (j) -- (k);
						\draw[-] (k) -- (l);
					\end{tikzpicture} } } & {\footnotesize{}{}{ \begin{tikzpicture}[scale=.5,
							meagrecirc/.style={circle,draw, fill=black!100,thick},
							fatcirc/.style={circle,draw,thick},
							fatrect/.style = {rectangle,draw,thick},
							meagrerect/.style = {rectangle,fill=black!100,draw,thick}]
						\node[fatrect] (j) at (0,0) [label=right:$j$] {};
						\node[meagrerect] (k) at (1,1) [label=right:$k$] {};
						\node[meagrerect] (l) at (0,2) [label=left:$l$] {};
						\draw[-] (j) -- (k);
						\draw[-] (k) -- (l);
					\end{tikzpicture} } }  & {\footnotesize{}{}{ \begin{tikzpicture}[scale=.5,
							meagrecirc/.style={circle,draw, fill=black!100,thick},
							fatcirc/.style={circle,draw,thick},
							fatrect/.style = {rectangle,draw,thick},
							meagrerect/.style = {rectangle,fill=black!100,draw,thick}]
						\node[fatrect] (j) at (0,0) [label=right:$j$] {};
						\node[meagrerect] (k) at (1,1) [label=right:$k$] {};
						\node[fatcirc] (l) at (0,2) [label=left:$l$] {};
						\draw[-] (j) -- (k);
						\draw[-] (k) -- (l);
					\end{tikzpicture} } } \tabularnewline
\hline
\end{tabular}
\end{table}

\begin{table}[!htb]
\caption{TPS-trees up to order three (part eleven of eleven).\label{table:tps_trees11}}
\label{Table:PWtrees11} \centering 
\global\long\def\arraystretch{1.5}%
{\footnotesize{}{}}%
\begin{tabular}{|>{\centering\arraybackslash}m{1.1in}|>{\centering\arraybackslash}m{1.3in}|>{\centering\arraybackslash}m{1.3in}|>{\centering\arraybackslash}m{1.3in}|>{\centering\arraybackslash}m{1.3in}|>{\centering\arraybackslash}m{1.3in}|}
\hline 
{\footnotesize{}{}{$\tau$ } }  & {\footnotesize{}{}{ \begin{tikzpicture}[scale=.5,
							meagrecirc/.style={circle,draw, fill=black!100,thick},
							fatcirc/.style={circle,draw,thick},
							fatrect/.style = {rectangle,draw,thick},
							meagrerect/.style = {rectangle,fill=black!100,draw,thick}]
						\node[fatrect] (j) at (0,0) [label=right:$j$] {};
						\node[fatrect] (k) at (1,1) [label=right:$k$] {};
						\node[meagrecirc] (l) at (0,2) [label=left:$l$] {};
						\draw[-] (j) -- (k);
						\draw[-] (k) -- (l);
					\end{tikzpicture} } }  & {\footnotesize{}{}{ \begin{tikzpicture}[scale=.5,
							meagrecirc/.style={circle,draw, fill=black!100,thick},
							fatcirc/.style={circle,draw,thick},
							fatrect/.style = {rectangle,draw,thick},
							meagrerect/.style = {rectangle,fill=black!100,draw,thick}]
						\node[fatrect] (j) at (0,0) [label=right:$j$] {};
						\node[fatrect] (k) at (1,1) [label=right:$k$] {};
						\node[meagrerect] (l) at (0,2) [label=left:$l$] {};
						\draw[-] (j) -- (k);
						\draw[-] (k) -- (l);
					\end{tikzpicture} }} & {\footnotesize{}{}{ \begin{tikzpicture}[scale=.5,
							meagrecirc/.style={circle,draw, fill=black!100,thick},
							fatcirc/.style={circle,draw,thick},
							fatrect/.style = {rectangle,draw,thick},
							meagrerect/.style = {rectangle,fill=black!100,draw,thick}]
						\node[fatrect] (j) at (0,0) [label=right:$j$] {};
						\node[fatrect] (k) at (1,1) [label=right:$k$] {};
						\node[fatcirc] (l) at (0,2) [label=left:$l$] {};
						\draw[-] (j) -- (k);
						\draw[-] (k) -- (l);
					\end{tikzpicture} } }  & {\footnotesize{}{}{ \begin{tikzpicture}[scale=.5,
							meagrecirc/.style={circle,draw, fill=black!100,thick},
							fatcirc/.style={circle,draw,thick},
							fatrect/.style = {rectangle,draw,thick},
							meagrerect/.style = {rectangle,fill=black!100,draw,thick}]
						\node[fatrect] (j) at (0,0) [label=right:$j$] {};
						\node[fatrect] (k) at (1,1) [label=right:$k$] {};
						\node[fatrect] (l) at (0,2) [label=left:$l$] {};
						\draw[-] (j) -- (k);
						\draw[-] (k) -- (l);
					\end{tikzpicture}} }  & {\footnotesize{}{}{$\phi$} } \tabularnewline
\hline
\end{tabular}
\end{table}
\end{landscape}
\restoregeometry
\fi


\begin{thebibliography}{10}

\bibitem{akrivis2004}
Georgios Akrivis and Michel Crouzeix.
\newblock Linearly implicit methods for nonlinear parabolic equations.
\newblock {\em Mathematics of computation}, 73(246):613--635, 2004.

\bibitem{allen1979}
Samuel~M. Allen and John~W. Cahn.
\newblock A microscopic theory for antiphase boundary motion and its
  application to antiphase domain coarsening.
\newblock {\em Acta Metallurgica}, 27(6):1085 -- 1095, 1979.

\bibitem{Ascher_1995_IMEX}
U.~M. Ascher, S.J. Ruuth, and B.T.R Wetton.
\newblock Implicit explicit methods for time-dependent partial-differential
  equations.
\newblock {\em SIAM J. Numer. Anal.}, 32:797--823, 1995.

\bibitem{augustine2018}
A.~Augustine and A.~Sandu.
\newblock {MATLODE: A MATLAB ODE} solver and sensitivity analysis toolbox.
\newblock {\em {ACM TOMS}}, in preparation, \the\year.

\bibitem{belytschko1979}
T.~Belytschko, H.-J. Yen, and R.~Mullen.
\newblock Mixed methods for time integration.
\newblock {\em Computer Methods in Applied Mechanics and Engineering},
  17-18:259--275, feb 1979.

\bibitem{bhatt2017}
Ashish Bhatt and Brian~E. Moore.
\newblock {Structure-preserving Exponential Runge--Kutta Methods}.
\newblock {\em SIAM Journal on Scientific Computing}, 39(2):A593--A612, 2017.

\bibitem{butcher1972}
J.~C. Butcher.
\newblock An algebraic theory of integration methods.
\newblock {\em Mathematics of Computation}, 26(117):79--79, jan 1972.

\bibitem{butcher2009}
J.~C. Butcher.
\newblock Trees and numerical methods for ordinary differential equations.
\newblock {\em Numerical Algorithms}, 53(2-3):153--170, mar 2009.

\bibitem{Butcher_2011_exp-order}
J.C. Butcher.
\newblock Trees, {B}-series and exponential integrators.
\newblock {\em IMA Journal of Numerical Analysis}, 30:131---140, 2010.

\bibitem{Butcher_2016_book}
J.C. Butcher.
\newblock {\em Numerical Methods for Ordinary Differential Equations}.
\newblock Wiley, 3 edition, 2016.

\bibitem{Kennedy_2003_additiveRK}
{C.A. Kennedy and M.H. Carpenter}.
\newblock {Additive Runge-Kutta schemes for convection-diffusion-reaction
  equations}.
\newblock {\em {Applied Numerical Mathematics}}, {44}({1-2}):{139--181},
  {2003}.

\bibitem{Calvo_2001_IMEX_RK}
M.~P. Calvo, J.~de~Frutos, and J.~Novo.
\newblock Linearly implicit runge-kutta methods for
  advection-reaction-diffusion equations.
\newblock {\em Applied Numerical Mathematics}, 37(4):535--549, 2001.

\bibitem{calvo1994}
MP~Calvo and JM~Sanz-Serna.
\newblock Canonical {B}-series.
\newblock {\em Numerische Mathematik}, 67(2):161--175, 1994.

\bibitem{Sandu_2014_IMEX-RK}
A.~Cardone, Z.~Jackiewicz, A.~Sandu, and H.~Zhang.
\newblock Extrapolated {IMEX Runge-Kutta} methods.
\newblock {\em Mathematical Modelling and Analysis}, 19(2):18--43, 2014.

\bibitem{Sandu_2014_IMEX_GLM_Extrap}
A.~Cardone, Z.~Jackiewicz, A.~Sandu, and H.~Zhang.
\newblock Extrapolation-based implicit-explicit general linear methods.
\newblock {\em Numerical Algorithms}, 65(3):377--399, 2014.

\bibitem{Sandu_2015_Stable_IMEX-GLM}
A.~Cardone, Z.~Jackiewicz, A.~Sandu, and H.~Zhang.
\newblock Construction of highly-stable implicit-explicit general linear
  methods.
\newblock In M.~de~Leon, W.~Feng, Z.~Feng, J.L. Gomez, X.~Lu, J.M. Martell,
  J.~Parcet, D.~Peralta-Salas, and W.~Ruan, editors, {\em AIMS proceedings},
  volume 0133-0189_2015_special_185 of {\em Dynamical Systems, Differential
  Equations, and Applications}, pages 185--194, Madrid, Spain, 2015.

\bibitem{celledoni2009}
Elena Celledoni and Bawfeh~Kingsley Kometa.
\newblock {Semi-Lagrangian {Runge--Kutta} Exponential Integrators for
  Convection Dominated Problems}.
\newblock {\em Journal of Scientific Computing}, 41(1):139--164, apr 2009.

\bibitem{certaine1960}
John Certaine.
\newblock The solution of ordinary differential equations with large time
  constants.
\newblock {\em Mathematical methods for digital computers}, 1:128--132, 1960.

\bibitem{chartier2005}
Philippe Chartier, Ernst Hairer, and Gilles Vilmart.
\newblock {\em A substitution law for B-series vector fields}.
\newblock PhD thesis, INRIA, 2005.

\bibitem{chou2007}
Ching-Shan Chou, Yong-Tao Zhang, Rui Zhao, and Qing Nie.
\newblock {Numerical methods for stiff reaction-diffusion systems}.
\newblock {\em Discrete and Continuous Dynamical Systems - Series B},
  7(3):515--525, feb 2007.

\bibitem{Sandu_2010_extrapolatedIMEX}
E.M. Constantinescu and A.~Sandu.
\newblock {Extrapolated implicit-explicit time stepping}.
\newblock {\em SIAM Journal on Scientific Computing}, 31(6):4452--4477, 2010.

\bibitem{Dettmer2012}
Wulf~G. Dettmer and Djordje Peri{\'{c}}.
\newblock A new staggered scheme for fluid-structure interaction.
\newblock {\em International Journal for Numerical Methods in Engineering},
  93(1):1--22, jul 2012.

\bibitem{ehle1975}
B.~L. Ehle and J.~D. Lawson.
\newblock {Generalized Runge--Kutta Processes for Stiff Initial-value
  Problems}.
\newblock {\em {IMA} Journal of Applied Mathematics}, 16(1):11--21, 1975.

\bibitem{farago2007}
Istv{\'a}n Farag{\'o} and J{\"u}rgen Geiser.
\newblock Iterative operator-splitting methods for linear problems.
\newblock {\em International Journal of Computational Science and Engineering},
  3(4):255--263, 2007.

\bibitem{farago2008}
Istv{\'a}n Farag{\'o}, Bogl{\'a}rka Gnandt, and {\'A}gnes Havasi.
\newblock Additive and iterative operator splitting methods and their numerical
  investigation.
\newblock {\em Computers \& Mathematics with Applications}, 55(10):2266--2279,
  2008.

\bibitem{Farhat1991}
Charbel Farhat, K.C. Park, and Yves Dubois-Pelerin.
\newblock An unconditionally stable staggered algorithm for transient finite
  element analysis of coupled thermoelastic problems.
\newblock {\em Computer Methods in Applied Mechanics and Engineering},
  85(3):349--365, feb 1991.

\bibitem{fornberg1999}
Bengt Fornberg and Tobin~A Driscoll.
\newblock A fast spectral algorithm for nonlinear wave equations with linear
  dispersion.
\newblock {\em Journal of Computational Physics}, 155(2):456--467, 1999.

\bibitem{Sandu_2016_GARK-MR}
M.~G\"{u}nther and A.~Sandu.
\newblock Multirate generalized additive {Runge-Kutta} methods.
\newblock {\em Numerische Mathematik}, 133(3):497--524, 2016.

\bibitem{Hairer_book_I}
E.~Hairer, S.P. Norsett, and G.~Wanner.
\newblock {\em Solving ordinary differential equations {I}: {N}onstiff
  problems}.
\newblock Number~8 in Springer Series in Computational Mathematics.
  Springer-Verlag Berlin Heidelberg, 1993.

\bibitem{Hairer_book_II}
E.~Hairer and G.~Wanner.
\newblock {\em Solving ordinary differential equations {II}: {S}tiff and
  differential-algebraic problems}.
\newblock Number~14 in Springer Series in Computational Mathematics.
  Springer-Verlag Berlin Heidelberg, 2 edition, 1996.

\bibitem{hairer2006}
Ernst Hairer, Gerhard Wanner, and Christian Lubich.
\newblock Order conditions, trees and {B}-series.
\newblock In {\em Geometric Numerical Integration}, pages 51--96. Springer,
  2006.

\bibitem{hersch1958}
Joseph Hersch.
\newblock Contribution {\`a} la m{\'e}thode des {\'e}quations aux
  diff{\'e}rences.
\newblock {\em Zeitschrift f{\"u}r angewandte Mathematik und Physik ZAMP},
  9(2):129--180, 1958.

\bibitem{Hochbruck_1997_exp}
M.~Hochbruck and C.~Lubich.
\newblock On {Krylov} subspace approximations to the matrix exponential
  operator.
\newblock {\em SIAM Journal on Numerical Analysis}, 34(5):1911--1925, 1997.

\bibitem{Hochbruck_1998_exp}
M.~Hochbruck, C.~Lubich, and H.~Selhofer.
\newblock Exponential integrators for large systems of differential equations.
\newblock {\em SIAM Journal on Scientific Computing}, 19(5):1552--1574, 1998.

\bibitem{Hochbruck_2005_expRK}
M.~Hochbruck and A.~Ostermann.
\newblock Explicit exponential {Runge--Kutta} methods for semilinear parabolic
  problems.
\newblock {\em SIAM Journal on Numerical Analysis}, 43:1069---1090, 2005.

\bibitem{Hochbruck_2010_exp}
M.~Hochbruck and A.~Ostermann.
\newblock Exponential integrators.
\newblock {\em Acta Numerica}, 19:209--286, 2012.

\bibitem{Hochbruck_2009_exp}
M.~Hochbruck, A.~Ostermann, and J.~Schweitzer.
\newblock Exponential {Rosenbrock}-type methods.
\newblock {\em {SIAM Journal on Numerical Analysis}}, 47:786--803, 2009.

\bibitem{hochbruck1998}
Marlis Hochbruck, Christian Lubich, and Hubert Selhofer.
\newblock {Exponential Integrators for Large Systems of Differential
  Equations}.
\newblock {\em SIAM Journal on Scientific Computing}, 19(5):1552--1574, sep
  1998.

\bibitem{issa1986}
Raad~I Issa.
\newblock Solution of the implicitly discretised fluid flow equations by
  operator-splitting.
\newblock {\em Journal of computational physics}, 62(1):40--65, 1986.

\bibitem{karlsen1997}
Kenneth~Hvistendahl Karlsen and Nils~Henrik Risebro.
\newblock An operator splitting method for nonlinear convection-diffusion
  equations.
\newblock {\em Numerische Mathematik}, 77(3):365--382, 1997.

\bibitem{lawson1967}
J.~Douglas Lawson.
\newblock {Generalized Runge-Kutta Processes for Stable Systems with Large
  Lipschitz Constants}.
\newblock {\em {SIAM} Journal on Numerical Analysis}, 4(3):372--380, sep 1967.

\bibitem{li2015}
Dongping Li, Yuhao Cong, and Kaifeng Xia.
\newblock Flexible exponential integration methods for large systems of
  differential equations.
\newblock {\em Journal of Applied Mathematics and Computing}, 51(1-2):545--567,
  aug 2015.

\bibitem{lie1998}
K-A Lie, VIDAR Haugse, and K~Hvistendahl Karlsen.
\newblock Dimensional splitting with front tracking and adaptive grid
  refinement.
\newblock {\em Numerical Methods for Partial Differential Equations: An
  International Journal}, 14(5):627--648, 1998.

\bibitem{loffeld2013}
J.~Loffeld and M.~Tokman.
\newblock Comparative performance of exponential, implicit, and explicit
  integrators for stiff systems of {ODEs}.
\newblock {\em Journal of Computational and Applied Mathematics}, 241:45--67,
  mar 2013.

\bibitem{lorenz1996}
Edward~N. Lorenz.
\newblock Predictability {\textendash} a problem partly solved.
\newblock In Tim Palmer and Renate Hagedorn, editors, {\em Predictability of
  Weather and Climate}, pages 40--58. Cambridge University Press ({CUP}), 1996.

\bibitem{luan2013}
Vu~Thai Luan and Alexander Ostermann.
\newblock {Exponential B-Series: The Stiff Case}.
\newblock {\em SIAM Journal on Numerical Analysis}, 51(6):3431--3445, 2013.

\bibitem{luan2014a}
Vu~Thai Luan and Alexander Ostermann.
\newblock {Exponential Rosenbrock methods of order five - Construction,
  analysis and numerical comparisons}.
\newblock {\em Journal of Computational and Applied Mathematics}, 255:417--431,
  2014.

\bibitem{luan2016}
Vu~Thai Luan, Mayya Tokman, and Greg Rainwater.
\newblock Preconditioned implicit-exponential integrators ({IMEXP}) for stiff
  {PDEs}.
\newblock {\em Journal of Computational Physics}, 335:846--864, apr 2017.

\bibitem{macnamara2017}
Shev MacNamara and Gilbert Strang.
\newblock {\em Operator Splitting}, pages 95--114.
\newblock Springer International Publishing, Cham, 2016.

\bibitem{mahara2004}
Hitoshi Mahara, Nobuhiko~J. Suematsu, Tomohiko Yamaguchi, Kunishige Ohgane,
  Yasumasa Nishiura, and Masatsugu Shimomura.
\newblock {Three-variable reversible Gray-Scott model}.
\newblock {\em Journal of Chemical Physics}, 121(18):8968--8972, 2004.

\bibitem{mahara2005}
Hitoshi Mahara, Tomohiko Yamaguchi, and Masatsugu Shimomura.
\newblock {Entropy production in a two-dimensional reversible Gray-Scott
  system}.
\newblock {\em Chaos}, 15(4):8968, 2005.

\bibitem{mclachlan2002}
Robert~I. McLachlan and G.~Reinout~W. Quispel.
\newblock {Splitting methods}.
\newblock {\em Acta Numerica}, 11(2002):341--434, jan 2002.

\bibitem{minchev2005}
Borislav~V Minchev and Will~M Wright.
\newblock {A review of exponential integrators for first order semi-linear
  problems}.
\newblock 2005.

\bibitem{nakamura2001}
Takashi Nakamura, Ryotaro Tanaka, Takashi Yabe, and Kenji Takizawa.
\newblock Exactly conservative {semi-Lagrangian} scheme for multi-dimensional
  hyperbolic equations with directional splitting technique.
\newblock {\em Journal of computational physics}, 174(1):171--207, 2001.

\bibitem{Sandu_2018_EPIRK-adjoint}
M.~Narayanamurthi, U.~Romer, and A.~Sandu.
\newblock Solving parameter estimation problems with discrete adjoint
  exponential integrators.
\newblock {\em Optimization Methods and Software}, 33(4--6):750--770, 2018.

\bibitem{Sandu_2019_EPIRKW}
M.~Narayanamurthi, P.~Tranquilli, A.~Sandu, and M.~Tokman.
\newblock {EPIRK-W} and {EPIRK-K} time discretization methods.
\newblock {\em Journal of Scientific Computing}, 78(1):167--201, 2019.

\bibitem{nie2006}
Qing Nie, Yong-Tao Zhang, and Rui Zhao.
\newblock {Efficient semi-implicit schemes for stiff systems}.
\newblock {\em Journal of Computational Physics}, 214(2):521--537, may 2006.

\bibitem{niesen2012}
Jitse Niesen and Will~M. Wright.
\newblock {Algorithm 919: A Krylov subspace algorithm for evaluating the
  $\varphi$-functions appearing in exponential integrators}.
\newblock {\em ACM Trans. Math. Softw}, 38(3):22:1--22:19, 2012.

\bibitem{pearson1993}
John~E. Pearson.
\newblock {Complex patterns in a simple system}.
\newblock Technical Report 5118, 1993.

\bibitem{Rainwater_2014_semilinear}
G~Rainwater and Mayya Tokman.
\newblock A new class of split exponential propagation iterative methods of
  {Runge--Kutt}a type {(sEPIRK)} for semilinear systems of {ODE}s.
\newblock {\em Journal of Computational Physics}, 269:40--60, 2014.

\bibitem{saad1992}
Y.~Saad.
\newblock {Analysis of Some Krylov Subspace Approximations to the Matrix
  Exponential Operator}.
\newblock {\em {SIAM} Journal on Numerical Analysis}, 29(1):209--228, feb 1992.

\bibitem{Sandu_2003_aerosolFramework}
A.~Sandu and C.~T. Borden.
\newblock A framework for the numerical treatment of aerosol dynamics.
\newblock {\em Applied Numerical Mathematics}, 45(4):475--497, 2003.

\bibitem{Sandu_2015_GARK}
A.~Sandu and M.~G\"{u}nther.
\newblock A generalized-structure approach to additive {Runge-Kutta} methods.
\newblock {\em SIAM Journal on Numerical Analysis}, 53(1):17--42, 2015.

\bibitem{Sidje_1998}
Roger~B. Sidje.
\newblock Expokit: a software package for computing matrix exponentials.
\newblock {\em {ACM} Transactions on Mathematical Software}, 24(1):130--156,
  mar 1998.

\bibitem{smith2004}
Barry Smith, Petter Bjorstad, and William Gropp.
\newblock {\em Domain decomposition: parallel multilevel methods for elliptic
  partial differential equations}.
\newblock Cambridge university press, 2004.

\bibitem{sportisse2000}
Bruno Sportisse.
\newblock An analysis of operator splitting techniques in the stiff case.
\newblock {\em Journal of Computational Physics}, 161(1):140 -- 168, 2000.

\bibitem{steihaug1979}
Trond Steihaug and Arne Wolfbrandt.
\newblock An attempt to avoid exact {Jacobian} and nonlinear equations in the
  numerical solution of stiff differential equations.
\newblock {\em Mathematics of Computation}, 33(146):521--521, may 1979.

\bibitem{Sandu_2004_multiScale}
Y.~Tang, G.R. Carmichael, L.W. Horowitz, I.~Uno, J.H. Woo, D.G. Streets,
  D.~Dabdub, G.~Kurata, A.~Sandu, J.~Allan, E.~Atlas, F.~Flocke, L.G. Huey,
  R.O. Jakoubek, D.B. Millet, D.D. Parrish, P.K. Quinn, J.M. Roberts, T.B.
  Ryerson, E.~Williams, J.B. Nowak, D.~Worsnop, A.~Goldstein, S.~Donnelly,
  S.~Schauffler, V.~Stroud, K.~Johnson, M.A. Avery, H.B. Singh, and E.C. Apel.
\newblock {Multi-scale simulations of tropospheric chemistry in the Eastern
  Pacific and U.S. West coast during spring 2002}.
\newblock {\em Journal of Geophysical Research - Atmospheres}, 109(D23), 2004.

\bibitem{Tokman_2006_EPI}
M.~Tokman.
\newblock Efficient integration of large stiff systems of {ODE}s with
  exponential propagation iterative ({EPI}) methods.
\newblock {\em Journal of Computational Physics}, 213(2):748--776, 2006.

\bibitem{Tokman_2011_EPIRK}
M.~Tokman.
\newblock A new class of exponential propagation iterative methods of
  {R}unge--{K}utta type {(EPIRK)}.
\newblock {\em Journal of Computational Physics}, 230:8762---8778, 2011.

\bibitem{tokman2012}
Mayya Tokman, John Loffeld, and Paul Tranquilli.
\newblock {New Adaptive Exponential Propagation Iterative Methods of
  Runge--Kutta Type}.
\newblock {\em SIAM Journal on Scientific Computing}, 34(5):A2650--A2669, 2012.

\bibitem{toselli2005}
Andrea Toselli and Olof~B. Widlund.
\newblock {\em Domain Decomposition Methods {\textemdash} Algorithms and
  Theory}.
\newblock Springer Berlin Heidelberg, 2005.

\bibitem{Sandu_2014_expK}
P.~Tranquilli and A.~Sandu.
\newblock Exponential-{Krylov} methods for ordinary differential equations.
\newblock {\em Journal of Computational Physics}, 278:31--46, 2014.

\bibitem{Sandu_2014_ROK}
P.~Tranquilli and A.~Sandu.
\newblock {R}osenbrock-{K}rylov methods for large systems of differential
  equations.
\newblock {\em SIAM Journal on Scientific Computing}, 36(3):A1313--A1338, 2014.

\bibitem{Ascher_1997_IMEX_RK}
{U.M. Ascher and S.J. Ruuth and R.J. Spiteri}.
\newblock {Implicit-explicit Runge-Kutta methods for time-dependent partial
  differential equations}.
\newblock {\em {Applied Numerical Mathematics}}, 25:151--167, {1997}.

\bibitem{Verwer_2004_IMEX_RKC}
J.~G. Verwer and B.~P. Sommeijer.
\newblock {An implicit-explicit Runge--Kutta--Chebyshev scheme for
  diffusion-reaction equations}.
\newblock {\em SIAM Journal on Scientific Computing}, 25(5):1824--1835, 2004.

\bibitem{verwer1984}
Jan~G Verwer.
\newblock Contractivity of locally one-dimensional splitting methods.
\newblock {\em Numerische Mathematik}, 44(2):247--259, 1984.

\bibitem{Sandu_2012_ICCS-IMEX}
H.~Zhang and A.~Sandu.
\newblock A second-order diagonally-implicit-explicit multi-stage integration
  method.
\newblock In {\em Proceedings of the International Conference on Computational
  Science ICCS 2012}, volume~9, pages 1039--1046, April 2012.

\bibitem{Sandu_2014_IMEX-GLM}
H.~Zhang, A.~Sandu, and S.~Blaise.
\newblock Partitioned and implicit-explicit general linear methods for ordinary
  differential equations.
\newblock {\em Journal of Scientific Computing}, 61(1):119--144, 2014.

\bibitem{Sandu_2016_highOrderIMEX-GLM}
H.~Zhang, A.~Sandu, and S.~Blaise.
\newblock High order implicit--explicit general linear methods with optimized
  stability regions.
\newblock {\em SIAM Journal on Scientific Computing}, 38(3):A1430--A1453, 2016.

\bibitem{zhao2011}
Su~Zhao, Jeremy Ovadia, Xinfeng Liu, Yong-Tao Zhang, and Qing Nie.
\newblock {Operator splitting implicit integration factor methods for stiff
  reaction–diffusion–advection systems}.
\newblock {\em Journal of Computational Physics}, 230(15):5996--6009, jul 2011.

\bibitem{Sandu_2015_IMEX-TSRK}
E.~Zharovsky, A.~Sandu, and H.~Zhang.
\newblock A class of {IMEX} two-step {Runge-Kutta} methods.
\newblock {\em SIAM Journal on Numerical Analysis}, 53(1):321--341, 2015.

\bibitem{Zienkiewicz1988}
O.~C. Zienkiewicz, D.~K. Paul, and A.~H.~C. Chan.
\newblock Unconditionally stable staggered solution procedures for soil-pore
  fluid interaction problems.
\newblock {\em International Journal of Rock Mechanics and Mining Sciences {\&}
  Geomechanics Abstracts}, 25(5):233, oct 1988.

\end{thebibliography}
\end{document}